\documentclass[11pt,A4]{article}
\usepackage{amsmath,amsfonts,amssymb,amsthm}
\usepackage{mathrsfs}
\usepackage{latexsym}
\usepackage[english]{babel}

\usepackage{graphicx}
\usepackage[usenames,dvipsnames]{xcolor}

\newcommand{\blue}{\color{black}}

\newcommand{\green}{\color{black}}

\renewenvironment{proof}[1][\proofname]{{\bfseries #1.} }{\qed}

\def\Cov{{\rm Cov\,}}

\newcommand{\field}[1]{\mathbb{#1}}

\newcommand{\hT}{\widehat{T}_n}
\newcommand{\R}{\field{R}}

\newcommand{\T}{\field{T}}

\newcommand{\N}{\field{N}}

\newcommand{\Z}{\field{Z}}

\newcommand{\Var}{{\rm Var}}

\newcommand{\e}{{\rm e}}
\newcommand{\F}{{\mathscr{F}}}

\newcommand{\eps}{\varepsilon}

\newcommand{\Tb}{{\mathbb{T}}}

\def\authors#1{{ \begin{center} #1 \vspace{0pt} \end{center} } \smallskip}
\def\institution#1{{\sl \begin{center} #1 \vspace{0pt} \end{center} } }
\def\inst#1{\unskip $^{#1}$}
\def\title#1{{\huge\bf  \begin{center} #1 \vspace{0pt} \end{center}  } \smallskip}

\def\E{{\mathbb{ E}}}

\def\P{{\mathbb{P}}}
\def\F{{\mathscr{F}}}

\def\paref#1{(\ref{#1})}

\newtheorem{theorem}{Theorem}[section]
\newtheorem{proposition}[theorem]{Proposition}
\newtheorem{definition}[theorem]{Definition}
\newtheorem{lemma}[theorem]{Lemma}

\newtheorem{defn}[theorem]{Definition}

\newtheorem{remark}[theorem]{Remark}
\newtheorem{assumption}[theorem]{Assumption}

\begin{document}

\title{{\bf Phase Singularities in Complex \\ Arithmetic Random Waves}}
\date{June 24, 2016}
\authors{%
 \sc Federico Dalmao\inst{1}, Ivan Nourdin \inst{2}, \\ Giovanni Peccati \inst{2} and Maurizia Rossi \inst{2}
}
\institution{\small \inst{1}Departamento de Matem\'atica y Estad\'istica del Litoral, \\ Universidad de la Rep\'ublica, Uruguay \\
\inst{2}Unit\'e de Recherche en Math\'ematiques,
 Universit\'e du Luxembourg
}

\begin{abstract}

{{ Complex arithmetic random waves} are stationary Gaussian complex-valued solutions of the Helmholtz equation on the two-dimensional flat torus. We use Wiener-It\^o chaotic expansions in order to derive a complete characterization of the second order high-energy behaviour of the total number of phase singularities of these functions. Our main result is that, while such random quantities verify a universal law of large numbers, they also exhibit non-universal and non-central second order fluctuations that are dictated by the arithmetic nature of the underlying spectral measures. Such fluctuations are qualitatively consistent with the cancellation phenomena predicted by Berry (2002) in the case of complex random waves on compact planar domains.  Our results extend to the complex setting recent pathbreaking findings by Rudnick and Wigman (2008), Krishnapur, Kurlberg and Wigman (2013) and Marinucci, Peccati, Rossi and Wigman (2016). The exact asymptotic characterization of the variance is based on a fine analysis of the Kac-Rice kernel around the origin, as well as on a novel use of combinatorial moment formulae for controlling long-range weak correlations. }

\medskip


\medskip

\end{abstract}

\section{Introduction}

\subsection{Overview and main results}

Let $\T := \R^2/\Z^2$ be the two-dimensional flat torus, and define $\Delta = \partial^2/\partial x_1^2+ \partial^2/\partial x_2^2$ to be the associated Laplace-Beltrami operator. Our aim in this paper is to characterize the high-energy behaviour of the zero set of complex-valued random eigenfunctions of $\Delta$, that is, of solutions $f$ of the {Helmholtz equation}
\begin{equation}\label{Hequation}
\Delta f + Ef = 0,
\end{equation}
for some adequate $E>0$. {{} In order to understand such a setting, recall} that the eigenvalues of $-\Delta$ are the positive reals of the form $E_n:=4\pi^2n$, where $n=a^2+b^2$ for some $a,b\in \Z$ (that is, $n$ is an integer that can be represented as the sum of two squares). {Here, and throughout the paper}, we set
$$
S := \{ n \in \N : a^2+b^2=n, \,\, \mbox{for some } \,\, a,b\in \Z\},
$$
and for $n\in S$ we define
$$\Lambda_n:=\lbrace \lambda=(\lambda_1,\lambda_2)\in \Z^2: \|\lambda\|^2:=\lambda_1^2+\lambda_2^2=n\rbrace$$
to be the set of {\bf energy levels} associated with $n$, while $\mathcal N_n:=|\Lambda_n|$ denotes its cardinality. An orthonormal basis (in $L^2(\mathbb T)$) for the eigenspace associated with $E_n$ is given by the set of complex exponentials $ \lbrace e_\lambda : \lambda\in \Lambda_n \rbrace$, defined as
$$
e_\lambda(x):=\e^{i2\pi\langle \lambda,x\rangle}, \quad x\in \mathbb T,
$$
with $i = \sqrt{-1}$.

\smallskip

For every $n\in S$, the integer $\mathcal N_n  =: r_2(n)$ counts the number of distinct ways of representing $n$ as the sum of two squares: it is a standard fact (proved e.g. by using Landau's theorem) that $\mathcal{N}_n$ grows on average as $\sqrt{\log n}$, and also that there exists an infinite sequence of prime numbers $p\in S$, $p\equiv 1\, {\rm mod}\, 4$, such that $\mathcal{N}_p=8$. A classical discussion of the properties of $S$ and $\mathcal N_n$ can be found e.g. in \cite[Section 16.9 and 16.10]{H-W}. In the present paper, we will systematically consider sequences $\{n_j\}\subset S$ such that $\mathcal{N}_{n_j}\to \infty$ (this is what we refer to as the {\bf high-energy limit}).

\medskip

The {complex waves} considered in this paper are natural generalizations of the real-valued arithmetic waves introduced by Rudnick and Wigman in \cite{RW}, and further studied in \cite{KKW, MPRW, ORW, RW2}; as such, they are close relatives of the complex fields considered in the physical literature --- see e.g. \cite{BD, Berry 2002, N-survey, NV}, {{} as well as the discussion provided below}.  For every $n\in S$, we define the {\bf complex arithmetic random wave of order} $n$ to be the random field
\begin{equation}\label{e:defrf}
\Theta_n(x):=\frac{1}{\sqrt{\mathcal N_n}}\sum_{\lambda\in \Lambda_n} v_\lambda \,  e_\lambda(x),\quad x\in \mathbb T,
\end{equation}
where the $v_\lambda,\,  \lambda\in \Lambda_n$, are independent and identically distributed (i.i.d.) complex-valued Gaussian random variables such that, for every $\lambda \in \Lambda_n$, ${\rm Re}(v_\lambda)$ and ${\rm Im}(v_\lambda)$ are two independent centered Gaussian random variables with mean zero and variance one\footnote{{} Considering random variables $v_\lambda$ with variance 2 (instead of a more usual unit variance) will allow us to slightly simplify the discussion contained in Section \ref{ss:cr}.}.
The family $\{v_\lambda : \lambda\in \Lambda_n, \, n\in S\}$ is tacitly assumed to be defined on a common probability space $(\Omega, \mathscr{F}, \P)$, with $\E$ indicating expectation with respect to $\P$. It is immediately verified that $\Theta_n$ satisfies the equation \paref{Hequation}, {that is,} $\Delta \Theta_n + E_n \Theta = 0$, and also that $\Theta_n$ is {\bf stationary}, in the sense that, for every $y\in \T$, the translated process $x\mapsto \Theta_n(y+x)$ has the same distribution as $\Theta_n$ (this follows from the fact that the distribution of $\{v_\lambda : \lambda\in \Lambda_n\}$ is invariant with respect to unitary transformations; see Section \ref{ss:cr} for further details on this {straightforward} but fundamental point).

\medskip

The principal focus of our investigation are the high-energy fluctuations of the {{} following} {\bf zero sets}:

\begin{eqnarray}
\mathscr{I}_n &:=& \{  x\in \T : \Theta_n(x) = 0     \} \label{e:zeroset} \\ &=& \{  x\in \T : {\rm Re} (\Theta_n(x)) = 0     \}\cap \{  x\in \T : {\rm Im} (\Theta_n(x)) = 0     \}, \quad n\in S.\notag
\end{eqnarray}
We will show below (Part 1 of Theorem \ref{t:main}) that, with probability one, $\mathscr{I}_n$ is a finite {{} collection of isolated points} for every $n\in S$; throughout the paper, we will write
\begin{equation}\label{ienne}
I_n := |\mathscr{I}_n | = {\rm Card} (\mathscr{I}_n),\qquad n\in S.
\end{equation}
In accordance with the title of this work, the points of $\mathscr{I}_n$ are {called} {\bf phase singularities} for the field $\Theta_n$, in the sense that, for every $x\in \mathscr{I}_n$, the phase of $\Theta_n(x)$ (as a complex-valued random quantity) is not defined.

\medskip

As for nodal lines of real arithmetic waves \cite{KKW, MPRW}, our main results crucially involve the following collection of probability measures on the unit circle $S^1\subset \R^2$:
\begin{equation}\label{e:spectral}
\mu_n (dz) := \frac{1}{\mathcal{N}_n} \sum_{\lambda\in \Lambda_n} \delta_{\lambda/\sqrt{n}} (dz), \quad n\in S,
\end{equation}
as well as the associated Fourier coefficients
\begin{equation}\label{e:fourier}
\widehat{\mu}_n (k ) := \int_{S^1} z^{-k} \mu_n (dz), \quad k\in \Z.
\end{equation}

In view of the definition of $\Lambda_n$, the probability measure $\mu_n$ defined in \eqref{e:spectral} is trivially invariant with respect to the transformations $z\mapsto \overline{z}$ and $z\mapsto i \cdot z$.  The somewhat erratic behaviours of such objects in the high-energy limit are studied in detail in \cite{KKW, KW}. Here, we only record the following statement, implying in particular that the sequences $\{\mu_n : n\in S\}$ and $\{\widehat{\mu}_n (4 ) : n\in S\}$ {\it do not} admit limits as $\mathcal{N}_n$ diverges to infinity within the set $S$.

Recall from \cite{KKW, KW} that a measure $\mu$ on $(S^1, \mathscr{B})$ (where $\mathscr{B}$ is the Borel $\sigma$-field) is said to be {\bf attainable} if there exists a sequence $\{n_j\}\subset S$ such that $\mathcal{N}_{n_j}\to \infty$ and $\mu_{n_j}$ converges to $\mu$ in the sense of the weak-$\star$ topology.

\begin{proposition}[\bf See \cite{KW, KKW}]\label{p:kw} The class of attainable measures is an infinite strict subset of the collection of all probability measures on $S^1$ that are invariant with respect to the transformations $z\mapsto \overline{z}$ and $z\mapsto i \cdot z$. Also, for every $\eta\in [0,1]$ there exists a sequence $\{n_j\}\subset S$ such that $\mathcal{N}_{n_j}\to \infty$ and $| \widehat{\mu}_{n_j}(4)| \to \eta$.
\end{proposition}

\medskip

Note that, if $\mu_{n_j}$ converges to $\mu_\infty$ in the weak-$\star$ topology, then $\widehat{\mu}_{n_j} (4) \to \widehat{\mu}_\infty(4)$. For instance, one knows from \cite{EH, KKW} that there exists a density one sequence $\{n_j\}\subset S$ such that $\mathcal{N}_{n_j}\to \infty$ and $\mu_{n_j}$ converges to the uniform measure on $S^1$, in which case $\widehat{\mu}_{n_j}(4)\to 0$.

\medskip

\noindent{\it Some conventions.} Given two sequences of positive numbers $\lbrace a_m\rbrace$ and $\lbrace b_m \rbrace$, we shall write $a_m\sim b_m$ if $a_m/b_m \to 1$,
and $a_m\ll b_m$ or (equivalently and depending on notational convenience) $a_m = O(b_m)$ if $a_m/b_m$ is asymptotically bounded. The notation $a_m=o(b_m)$ means as usual that $a_m/b_m\to 0$.
Convergence in distribution for random variables on $(\Omega, \F, \P)$ will be denoted by $\overset{\rm law}{\Longrightarrow}$, whereas equality in distribution will be indicated by the symbol $\overset{\rm law}{=}$.

\medskip

The main result of the present work is the following exact characterization of the first and second order behaviours of $I_n$, {as defined by} \paref{ienne}, in the high-energy limit. As discussed below, it is a highly non-trivial extension of the results proved in \cite{KKW, MPRW}, as well as the first rigorous description of the {\bf Berry's cancellation phenomenon} \cite{Berry 2002} in the context of phase singularities of complex random waves.

\medskip

\begin{theorem} \label{t:main}
\begin{enumerate}

\item[\rm 1.] {\bf (Finiteness and mean)} With probability one, for every $n\in S$ the set $\mathscr{I}_n$ is composed of a finite collection of {isolated} points, and
\begin{equation}\label{e:exp}
\E[I_n] = \frac{E_n}{4\pi} = \pi n.
\end{equation}

\item[\rm 2.] {\bf (Non-universal variance asymptotics)}
As $\mathcal{N}_n \to \infty$,
\begin{equation}\label{e:variance}
\Var(I_n) = d_n \times \frac{E_n^2}{\mathcal N_n^2} \, (1+o(1)),
\end{equation}
where
\begin{equation}\label{e:dn}
d_n := \frac{3\widehat \mu_n(4)^2 + 5}{128\pi^2}.
\end{equation}

\item[\rm 3.] {\bf (Universal law of large numbers)} Let $\lbrace n_j\rbrace\subset S$ be a subsequence such that $\mathcal N_{n_j}\to +\infty$. Then, for every sequence $\{\epsilon_{n_j}\}$ such that $\epsilon_{n_j}\mathcal{N}_{n_j}\to \infty$, one has that

\begin{equation}\label{e:wlln}
\mathbb{P}\left[ \left| \frac{I_{n_j}}{\pi n_j} - 1 \right| >\epsilon_{n_j} \right]  \to 0.
\end{equation}

\item[\rm 4.] {\bf (Non-universal and non-central second order fluctuations)} Let $\lbrace n_j\rbrace\subset S$ be such that $\mathcal N_{n_j}\to +\infty$ and $|\widehat \mu_{n_j}(4)|\to \eta \in [0,1]$. Then,
\begin{eqnarray}\notag
\widetilde I_{n_j} &:=& \frac{I_{n_j} -\E[I_{n_j}] }{\sqrt{\Var(I_{n_j}) }}\\  &\overset{\rm law}{\Longrightarrow}&
\frac{1}{2\sqrt{10+6\eta^2}}\left (\frac{1+\eta}{2}A
+\frac{1-\eta}{2}B
-2(C-2)\right) =: \mathcal{J}_\eta,\label{e:j}
\end{eqnarray}
with $A,B,C$ independent random variables such that $A\overset{\rm law}{=}B\overset{\rm law}{=}2X_1^2+2X_2^2-4X_3^2$
and $C \overset{\rm law}{=} X_1^2+X_2^2$, where
$(X_1,X_2,X_3)$ is a standard Gaussian vector of $\R^3$.

\end{enumerate}

\end{theorem}

\begin{remark}\label{r:postmain}{\rm

\begin{enumerate}

\item The arguments leading to the proof of \eqref{e:exp} show also that, for every measurable $A\subset \Tb$,
\begin{equation}\label{e:locexp}
\E[| \mathscr{I}_n\cap A |] =  {\rm Leb}(A)\times  \pi n,
\end{equation}
where `${\rm Leb}$' indicates the Lebesgue measure on the torus. The details are left to the reader.

\item For every $n\in S$, write $w(n) := \Var(I_n) \left( {E_n^2}/{\mathcal N_n^2}\right)^{-1}$. Standard arguments, based on compactness and on the fact that $\mu_n(4)\in [-1,1]$, yield that \eqref{e:variance} is equivalent to the following statement: for every $\{n_j\}\subset S$ such that $\mathcal{N}_{n_j}\to \infty$ and $|\mu_{n_j}(4)|\to \eta \in [0,1]$, one has that $w(n_j)\to d(\eta) := (3\eta^2+5)/128\pi^2$.

\item Relations \eqref{e:variance}--\eqref{e:dn} are {completely} new {and are {{} among} the main findings of the present paper; in particular, they} show that the asymptotic behaviour of the variance of $I_n$ is {\bf non-universal}. Indeed, when $\mathcal{N}_{n_j}\to \infty$, the fluctuations of the sequence $d_{n_j}$ depend on the chosen subsequence $\{n_j\}\subset S$, via the squared Fourier coefficients $\mu_{n_j}(4)^2$: in particular, the possible limit values of the sequence $\{d_{n_j}\}$ correspond to the whole interval $\left[\frac{5}{128\pi^2}, \frac{1}{16\pi^2}\right]$. As discussed in the sections to follow, such a non-universal behaviour {echoes} the findings of \cite{KKW}, in the framework of the length of nodal lines associated with arithmetic random waves. We will see that our derivation of the exact asymptotic relation \eqref{e:variance} follows a route that is different from the one exploited in \cite{KKW} --- as our techniques are based on chaos expansions, combinatorial cumulant formulae, as well as on a novel local Taylor expansion of the second order Kac-Rice kernel around the origin.

\item The support of the distribution of each variable $\mathcal{J}_\eta$ is the whole real line, but the distribution of $\mathcal{J}_\eta$ is {\it not} Gaussian (this follows e.g. by observing that law of $\mathcal{J}_\eta$ has exponential tails). As discussed in the forthcoming Section \ref{ss:cr}, similar non-central and non-universal second order fluctuations have been proved in \cite{MPRW} for the total nodal length of real arithmetic random waves. We will show below that this striking common feature originates from the same {\bf chaotic cancellation phenomenon} exploited in \cite{MPRW}, that is: {\it in the Wiener chaos expansion of the quantity $I_n$, the projection on the second chaos vanishes, and the limiting fluctuations of such a random variable are completely determined by its projection on the fourth Wiener chaos.} It will be clear from our analysis that, should the second chaotic projection of $I_n$ not disappear in the limit, then the order of ${\rm Var}(I_n)$ would be proportional to $E_n^2/\mathcal{N}_n$, as $\mathcal{N}_n\to \infty$.

\item Choosing $\epsilon_{n_j} \equiv \epsilon$ (constant sequence), one deduces from Point 3 of Theorem \ref{t:main} that the ratio $I_{n}/n$ converges in probability to $\pi$, whenever $\mathcal{N}_{n}\to \infty$. See e.g. \cite[p. 261]{D} for definitions.

\item It is easily checked (for instance, by computing the third moment) that the law of $\mathcal{J}_{\eta_0}$ differs from that of $\mathcal{J}_{\eta_1}$ for every $0\leq \eta_0< \eta_1\leq 1$. This fact implies that the sequence $\{| \mu_{n}(4)|\}$ dictates not only the asymptotic behaviour of the variance of $I_n$, but also the precise nature of its second order fluctuations.


\item Reasoning as in \cite[Theorem 1.2]{MPRW}, it is possible to suitably apply {\it Skorohod representation theorem} (see e.g. \cite[Chapter 11]{D}), in order to express relation \eqref{e:j} in a way that does not involve the choice of a subsequence $\{n_j\}$. We leave this routine exercise to the interested reader.

\end{enumerate}

}
\end{remark}

In the next section, we will discuss several explicit connections with the model of real arithmetic random waves studied in \cite{KKW, ORW, RW}.

\subsection{Complex zeros as nodal intersections}\label{ss:cr}

For simplicity, from now on we will write
\begin{equation}\label{e:tn}
T_n(x) := {\rm Re}(\Theta_n(x)), \quad \widehat{T}_n(x) := {\rm Im}(\Theta_n(x)),
\end{equation}
for every $x\in \T$ and $n\in S$; in this way, one has that
\begin{eqnarray*}
\mathscr{I}_n &=&T^{-1}_n(0)\cap \hT^{-1}(0) \quad \mbox{and} \quad I_n = |T^{-1}_n(0)\cap \hT^{-1}(0)|.
\end{eqnarray*}
We will also adopt the shorthand notation $${\bf T}_n := \{ {\bf T}_n(x) =(T_n(x),  \widehat{T}_n(x)) : x\in \T\}, \quad n\in S.$$

Our next statement yields a complete characterization of the distribution of the vector-valued process ${\bf T}_n$, as a two-dimensional field whose components are independent and identically distributed real arithmetic random waves, in the sense of \cite{KKW, MPRW, RW}.

\begin{proposition}\label{p:realim} Fix $n \in S$. Then, $T_n$ and $\hT$ are two real-valued independent centered Gaussian fields such that
\begin{equation}\label{e:cov}
\E\Big[T_n(x)T_n(y)\Big] = \E\Big[\hT(x)\hT(y)\Big] =  \frac{1}{\mathcal{N}_n} \sum_{\lambda\in \Lambda_n} \cos(2\pi \langle \lambda, x-y\rangle )=:r_n(x-y).
\end{equation}
Also, there exist two collections of complex random variables
\begin{equation}\label{e:a}
 {\bf A}(n) = \{ a_\lambda : \lambda \in \Lambda_n\} \quad \mbox{and} \quad \widehat{\bf A}(n)=\{ \widehat{a}_\lambda : \lambda \in \Lambda_n\},
\end{equation}
with the following properties:
\begin{itemize}
\item[\rm (i)] $ {\bf A}(n)$ and $\widehat{\bf A}(n)$ are stochastically independent and identically distributed as random vectors indexed by $\Lambda_n$;

\item[\rm (ii)] For every $\lambda\in \Lambda_n$, $a_\lambda$ is a complex-valued Gaussian random variable whose real and imaginary parts are independent Gaussian random variables with mean zero and variance $1/2$;

\item[\rm (iii)] If $\lambda \notin \{\sigma, -\sigma\}$, then $a_\lambda$ and $a_\sigma$ are stochastically independent;

\item[\rm (iv)] $a_\lambda = \overline{a_{-\lambda}}$;

\item[\rm (v)] For every $x\in \T$,
\begin{equation}\label{e:realim}
T_n(x) = \frac{1}{\sqrt{\mathcal{N}_n}}\sum_{\lambda\in \Lambda_n} a_\lambda e_\lambda(x), \,\,\, \mbox{and} \,\,\,\,\, \widehat{T}_n(x) = \frac{1}{\sqrt{\mathcal{N}_n}}\sum_{\lambda\in \Lambda_n} \widehat{a}_\lambda e_\lambda(x).
\end{equation}

\end{itemize}
\end{proposition}
\noindent\begin{proof} One need only show that $T_n$ and $\hT$ are two centered independent Gaussian fields such that \eqref{e:cov} holds; the existence of the two families ${\bf A}(n)$ and $\widehat{\bf A}(n)$ can consequently be derived by first expanding $T_n$ and $\hT$ in the basis $\{e_\lambda : \lambda\in \Lambda_n\}$, and then by explicitly computing the covariance matrix of the resulting Fourier coefficients. Relation \eqref{e:cov} follows from a direct computation based on the symmetric structure of the set $\Lambda_n$, once it is observed that $T_n$ and $\hT$ can be written in terms of the complex Gaussian random variables $\{v_\lambda\}$ appearing in \eqref{e:defrf}, as follows:
\begin{eqnarray*}
T_n(x) &=& \frac{1}{\sqrt{\mathcal{N}_n}} \sum_{\lambda\in \Lambda_n} \Big\{{\rm Re}(v_\lambda)\cos(2\pi \langle \lambda, x\rangle)-{\rm Im}(v_\lambda)\sin(2\pi \langle \lambda, x\rangle)\Big\},\\
\hT(x) &=& \frac{1}{\sqrt{\mathcal{N}_n}} \sum_{\lambda\in \Lambda_n} \Big\{{\rm Im}(v_\lambda)\cos(2\pi \langle \lambda, x\rangle)+{\rm Re}(v_\lambda)\sin(2\pi \langle \lambda, x\rangle)\Big\}.
\end{eqnarray*}
\end{proof}

\medskip

The fact that $r_n$ only depends on the difference $x-y$ confirms in particular that ${\bf T}_n$ is a two-dimensional Gaussian stationary process.

\begin{assumption}\label{a:ass}{\rm Without loss of generality, for the rest of the paper we will assume that, for $n\neq m$, the two Gaussian families
$$
{\bf A}(n)\cup \widehat{\bf A}(n)\quad \mbox{and} \quad {\bf A}(m) \cup \widehat{\bf A}(m)
$$
are stochastically independent; this is the same as assuming that the two vector-valued fields ${\bf T}_n$ and ${\bf T}_m$ are stochastically independent.

}
\end{assumption}

As anticipated, relation \eqref{e:realim} implies that $T_n$ and $\hT$ are two independent and indentically distributed {\it real arithmetic random waves} of order $n$, such as the ones introduced in \cite{RW}, and then further studied in \cite{KKW, MPRW, ORW, RW2}. We recall that, according to \cite{C}, with probability one both $T_n^{-1}(0)$ and $\hT^{-1}(0)$ are unions of rectifiable curves, called {\bf nodal lines}, containing a finite set of isolated singular points. The following statement yields a further geometric characterisation of $\mathscr{I}_n$ and $I_n$: its proof is a direct by-product of the arguments involved in the proof of Part 1 of Theorem \ref{t:main}.

\begin{proposition}\label{p:inter} Fix $n\in S$. Then, with probability one the nodal lines of $T_n$ and $\hT$ have a finite number of {isolated intersection points}, whose collection coincides with the set $\mathscr{I}_n$; moreover, $\mathscr{I}_n$ does not contain any singular point for ${\bf T}_n$.
\end{proposition}

In view of Proposition \ref{p:inter}, it is eventually instructive to focus on the random {\bf nodal lengths}
$$
L_n := {\rm length}\, (T_n^{-1}(0)), \quad n\in S,
$$
for which we will present a statement collecting some of the most relevant findings from \cite{RW} (Point 1), \cite{KKW} (Point 2) and \cite{MPRW} (Point 3).

\begin{theorem}[\bf See \cite{RW, KKW, MPRW}]\label{t:lmain} \begin{enumerate}

\item[\rm 1.] For every $n\in S$
\begin{equation}\label{e:lexp}
\E[L_n] = \frac{E_n}{2\sqrt{2}}.
\end{equation}

\item[\rm 2.]
As $\mathcal{N}_n \to \infty$,
\begin{equation}\label{e:lvariance}
\Var(L_n) = c_n \times \frac{E_n}{\mathcal N_n^2} \, (1+o(1)),
\end{equation}
where
\begin{equation}\label{e:cn}
c_n := \frac{1 + \widehat \mu_n(4)^2}{512}.
\end{equation}

\item[\rm 3.] Let $\lbrace n_j\rbrace\subset S$ be such that $\mathcal N_{n_j}\to +\infty$ and $|\widehat \mu_{n_j}(4)|\to \eta \in [0,1]$. Then,
\begin{eqnarray}\notag
\widetilde L_{n_j} &:=& \frac{L_{n_j} -\E[L_{n_j}] }{\sqrt{\Var(L_{n_j}) }}\\  &\overset{\rm law}{\Longrightarrow}&
\mathcal{M}_\eta := \frac{1}{2\sqrt{1+\eta^2}} (2 - (1+\eta) X_1^2-(1-\eta) X_2^2),\label{e:m}
\end{eqnarray}
where $(X_1,X_2)$ is a standard Gaussian vector of $\R^2$.

\end{enumerate}

\end{theorem}

As discussed e.g. in \cite{KKW,RW}, relations \eqref{e:lexp}--\eqref{e:lvariance} yield immediately a law of large numbers analogous to \eqref{e:wlln}. We stress that Theorem \ref{t:main} and Theorem \ref{t:lmain} share three common striking features (explained below in terms of a common chaotic cancellation phenomenon), namely: (a) an {inverse quadratic dependence} on $\mathcal{N}_n$, as displayed in formulae \eqref{e:variance} and \eqref{e:lvariance}, (b)  non-universal variance fluctuations, determined by the quantities $d_n$ and $c_n$ defined in \eqref{e:dn} and \eqref{e:cn}, respectively, and (c) non-universal and non-central second order fluctuations (see \eqref{e:j} and \eqref{e:m}).

The estimate \eqref{e:lvariance} largely improves upon a conjecture formulated by Rudnick and Wigman in \cite{RW}, according to which one should have $\Var(L_n) = O(E_n/\mathcal{N}_n)$. The fact that the natural leading term $E_n/\mathcal{N}_n$ actually disappears in the high-energy limit, thus yielding \eqref{e:lvariance}, is connected to some striking discoveries by Berry \cite{Berry 2002}, discussed in the forthcoming section.

\subsection{More about relevant previous work}

\underline{\it Random waves and cancellation phenomena.} To the best of our knowledge, the first systematic analysis of phase singularities in wave physics appears in the seminal contribution by Nye and Berry \cite{NB}. Since then, zeros of complex waves have been the object of an intense study in a variety of branches of modern physics, often under different names, such as {\it nodal points}, {\it wavefront dislocations}, {\it screw dislocations}, {\it optical vortices} and {\it topological charges}. The reader is referred e.g. to \cite{D-survey, N-survey, UR-survey}, and the references therein, for detailed surveys on the topic, focussing in particular on optical physics, quantum chaos and quantum statistical physics.

One crucial reference for our analysis is Berry \cite{Berry 1978}, where the author studies several statistical quantities involving singularities of random waves on the plane. Such an object, usually called the (complex) {\bf Berry's random wave model} (RWM), is defined as a complex centered Gaussian field, whose real and imaginary parts are independent Gaussian functions on the plane, with covariance
\begin{equation}\label{e:rwm}
r_{\rm RWM}(x,y) := J_0\left(\sqrt{E}\,  \| x-y\|\right), \quad x,y\in \R^2,
\end{equation}
where $E>0$ is an energy parameter, and $J_0$ is the standard Bessel function (see also \cite{Berry 1977}).  Formula \eqref{e:rwm} implies in particular that Berry's RWM is {\bf stationary and isotropic}, that is: its distribution is invariant both with respect to translations and rotations. As discussed e.g. in \cite[Section 1.6.1]{KKW}, if $\{n_j\} \subset S$ is a sequence such that $\mathcal{N}_{n_j}\to \infty$ and $\mu_{n_j}$ converges weakly to the uniform measure on the circle, then, for every $x\in \Tb$ and using the notation \eqref{e:cov},
\begin{equation}\label{e:scaling}
r_{n_j}\left(\sqrt{\frac{E}{n_j}}\cdot \frac{x}{2\pi}\right) \longrightarrow r_{\rm RWM}(x),
\end{equation}
showing that Berry's RWM is indeed the local scaling limit of the arithmetic random waves considered in the present paper.

Reference \cite{Berry 2002}, building upon previous findings of Berry and Dennis \cite{BD}, contains the following remarkable results: {\bf (a)} the expected nodal length per unit area of the real RWM equals $\sqrt{E} /(2\sqrt{2})$ \cite[Section 3.1]{Berry 2002}, {\bf (b)} as $E\to \infty$ the variance of the nodal length at Point {\bf (a)} is proportional to $\log E$  \cite[Section 3.2]{Berry 2002}, {\bf (c)} the expected number of phase singularities for unit area of the complex RWM is $E /(4\pi)$  \cite[Section 4.1]{Berry 2002}, and {\bf (d)} as $E\to \infty$ the variance of the number of singularities at Point {\bf (c)} is proportional to $E \log E$  \cite[Section 4.2]{Berry 2002}. Point {\bf (a)} and {\bf (c)} are perfectly consistent with \eqref{e:lexp} and \eqref{e:exp}, respectively. Following \cite{Berry 2002}, the estimates at Points {\bf (b)} and {\bf (d)} are due to an `obscure' cancellation phenomenon, according to which the natural leading term in variance (that should be of the order of $\sqrt{E}$ and $E^{3/2}$, respectively) cancels out in the high-energy limit. The content of Point {\bf (b)} has been rigorously confirmed by Wigman \cite{wig} in the related model of real {\it random spherical harmonics}, whose scaling limit is again the real RWM. See also \cite{ALW}.

As explained in \cite{KKW}, albeit improving conjectures from \cite{RW}, the order of the variance established in \eqref{e:lvariance} differs from that predicted in {\bf (b)}: this discrepancy is likely due to the fact that, differently from random spherical harmonics, the convergence in \eqref{e:scaling} does not take place uniformly over suitable regions. As already discussed, in \cite{MPRW} it was shown that the asymptotic relation \eqref{e:lvariance} is generated by a remarkable chaotic cancellation phenomenon, which also explains the non-central limit theorem stated in \eqref{e:m}.

The main result of the present paper (see Theorem \ref{t:main}) confirms that such a chaotic cancellation continues to hold for phase singularities of complex arithmetic waves, and that it generates non-universal and non-central second order fluctuations for such a random quantity. This fact lends further evidence to the natural conjecture that cancellation phenomena analogous to those described in {\cite{Berry 2002, wig, KKW, MPRW, Ro}} should hold for global quantities associated with the zero set of Laplace eigenfunctions on general manifolds, as long as such quantities can be expressed in terms of some area/co-area integral formula.

We stress that the fact that the order of the variance stated in \eqref{e:variance} differs from the one predicted at Point {\bf (d)} above, can once again be explained by the non-uniform nature of the scaling relation \eqref{e:scaling}.

\medskip

\noindent\underline{\it Leray measures and occupation densities}. While the present paper can be seen as a natural continuation of the analysis developed in \cite{KKW, MPRW}, the methods implemented below will substantially diverge from those in the existing literature. One fundamental difference stems from the following point: in order to deal with strong correlations between vectors of the type $(T_n(x), \partial/\partial_{1} T_n(x), \partial/\partial_{2} T_n(x))$ and $(T_n(y), \partial/\partial_{1} T_n(y), \partial/\partial_{2} T_n(y))$, $x\neq y$, the authors of \cite{KKW} extensively use results from \cite{ORW} (see in particular \cite[Section 4.1]{KKW}) about the fluctuations of the {\bf Leray measure}
$$
A_n := \int_{\Tb} \delta_0(T_n(x)) \, dx,
$$
which is defined as the limit in $L^2(\P)$ of the sequence $k\mapsto \int_{\Tb} \varphi_k(T_n(x)) \, dx$, with $\{\varphi_k\}$ a suitable approximation of the identity; on the other hand, following such a route in the framework of random phase singularities is impossible, since the formal quantity
$$
B_n:= \int_{\Tb} \delta_{(0,0)}(T_n(x), \hT(x)) \, dx
$$
{\it cannot}  be defined as an element of $L^2(\P)$. The technical analysis of singular points developed in Section \ref{proof varianza sec} is indeed how we manage to circumvent this difficulty.  We observe that, in the parlance of stochastic calculus, the quantity $A_n$ (resp. $B_n$) is the {\bf occupation density at zero} of the random field $T_n$ (resp. ${\bf T}_n$) --- in particular, the fact that $A_n$ is well-defined  in $L^2(\P)$ and $B_n$ {is not} --- follows from the classical criterion stated in \cite[Theorem 22.1]{GH}, as well as from the relations
\begin{equation}\label{e:explosion}
\int_\Tb \frac{dx}{\sqrt{1-r^2_n(x)}}  <\infty \quad \mbox{and} \quad \int_\T \frac{dx}{1-r^2_n(x)}  =\infty,
\end{equation}
where we have used the fact that, according e.g. to \cite[Lemma 5.3]{ORW}, the mapping $x\mapsto (1-r^2_n(x))^{-1}$ behaves like a multiple of $1/\|x-x_0\|^2$ around any point $x_0$ such that $r_n(x_0)=\pm1$.

\medskip

\noindent{\underline{\it Nodal intersections of arithmetic random waves with a fixed curve}.
{{} A natural problem related to the subject of our paper is that of studying the number of} nodal intersections with a fixed {{} deterministic} curve $\mathcal C\subset \mathbb T$ whose length equals $L$, i.e. number of zeroes of $T_n$ that lie on $\mathcal C$:
$$
\mathcal Z_n := T_n^{-1}(0)\cap \mathcal C.
$$
In \cite{RW2}, the case where $\mathcal C$ is a smooth curve with nowhere zero-curvature {has been investigated}. The expected number of nodal intersections is $\E[|\mathcal Z_n|]=(\pi \sqrt 2)^{-1}\times E_n \times L$, hence proportional to the length $L$ of the curve times the wave number, independent of the geometry. The asymptotic behaviour of the nodal intersections variance in the high energy limit is a subtler matter: it {depends} on both the angular distribution of lattice points lying on the circle with radius corresponding to the given wavenumber, in particular on the sequence of measures $\lbrace \mu_n\rbrace$, and {on} the geometry of $\mathcal C$. The asymptotic distribution of $|\mathcal Z_n|$ is}{analyzed in \cite{RoW}}. {See \cite{Ma} for the case where $\mathcal C$ is a segment.}


\medskip

\noindent\underline{\it Zeros of random analytic functions/Systems of polynomials}. To the best of our expertise, our limit result \eqref{e:j} is the first non-central limit theorem for the number of zeros of random complex analytic functions defined on some manifold $\mathcal{M}$. As such, our findings should be contrasted with the works by Sodin and Tsirelson \cite{ST1, ST2}, where one can find central limit results for local statistics of zeros of analytic functions corresponding to three different models (elliptic, flat and hyperbolic). As argued in \cite[Section 1.6.4]{wig}, these results are roughly comparable to those one would obtain by studying zeros of complex random spherical harmonics, for which a central high-energy behaviour is therefore likely to be expected. References \cite{SZ1, SZ2}, by Shiffman and Zelditch, contain central limit result for the volume of the intersection of the zero sets of independent Gaussian sections of high powers of holomorphic line bundles on a K\"ahler manifold of a fixed dimension.

In view of Proposition \ref{p:inter}, our results have to be compared with works dealing with the number of roots of random system of polynomials.
{The first important result about the number of roots of random systems is due to Shub and Smale \cite{ss},
where the authors compute the expectation of the number of roots of a square system with independent
centered Gaussian coefficients with a particular choice of the variances that makes the distribution of the polynomials invariant under the action of the orthogonal group of the parameter space.
This model is called the {\bf Shub-Smale model}.
Later, Edelman and Kostlan, see \cite{k} and references therein, and Aza\"is and Wschebor \cite{aw-pol} extend these results to more general Gaussian distributions.
Wschebor \cite{w} studies the asymptotic for the variance of the number of roots of a Shub-Smale system in the case where the number of equations and variables tends to infinity.
Armentano and Wschebor \cite{ar-w} consider the expectation of non-centered (perturbed) systems.
McLennan \cite{mcl} studies the expected number of roots of multihomogeneous systems. Rojas \cite{ro} consider the expected number of roots of certain sparse systems. }

%
%
%
%
%
%

\subsection{Short plan of the paper}

{Section \ref{s:prelim} gathers some preliminary results that will be needed in the sequel. In Section \ref{s:approach} we explain the main ideas and steps of the proof of our main result (Theorem \ref{t:main}). Finally, the remaining sections are devoted to the detailed proofs. In particular, we collect in Section \ref{appendix} some technical computations and proofs of intermediate results.
}

\subsection{Acknowledgments}

This research was supported by the grant F1R-MTH-PUL-15CONF (CONFLUENT) at the University of Luxembourg (Federico Dalmao and Ivan Nourdin), by the grant F1RMTH-PUL-15STAR (STARS) at the University of Luxembourg (Giovanni Peccati and Maurizia Rossi) and by the FNR grant O17/11756789/FoRGES (Giovanni Peccati).

The authors heartily thank Domenico Marinucci for suggesting the problem studied in this paper, and for several inspiring discussions.

\section{Some preliminary result}\label{s:prelim}

We will now present some useful notions and results connected to Wiener chaos, gradients, combinatorial moment formulae and arithmetic estimates.

\subsection{Wiener chaos }\label{ss:wiener}

Let $\{H_k : k=0,1,...\}$ be the sequence of {\bf Hermite polynomials} on $\mathbb{R}$, recursively defined as follows: $H_0 \equiv 1$, and, for $k\geq 1$,

 $$H_{k}(t) = tH_{k-1}(t) - H'_{k-1}(t),\qquad t\in \R.$$
It is a standard fact that the collection $\mathbb{H} := \{H_k/\sqrt{k!} : k\ge 0\}$ is a complete orthonormal system for $$L^2(\mathbb{R}, \mathscr{B}(\mathbb{R}), \gamma) :=L^2(\gamma),$$ where $\gamma(dt):= \phi(t)dt =\frac{ e^{-t^2/2}}{\sqrt{2\pi}} dt$ is the standard Gaussian measure on the real line. By construction, for every $k\geq 0$, one has that
\begin{equation}\label{e:sym}
H_{2k}(-t) = H_{2k}(t), \quad \mbox{and} \quad H_{2k+1}(-t) = -H_{2k+1}(t), \quad t\in \R.
\end{equation}

\medskip

In view of Proposition \ref{p:realim} (recall also Assumption \ref{a:ass}), every random object considered in the present paper is a measurable functional of the family of complex-valued Gaussian random variables $$  \bigcup_{n\in S} \Big( {\bf A}(n) \cup \widehat{{\bf A}}(n) \Big),$$ where ${\bf A}(n) $ and $\widehat{{\bf A}}(n)$ are defined in \eqref{e:a}. Now define the space ${\bf A}$ to be the closure in $L^2(\mathbb{P})$ of all real finite linear combinations of random variables $\xi$ of the form $$\xi = c_1 (z \, a_\lambda + \overline{z} \, a_{-\lambda}) + c_2(u \, \widehat{a}_\tau + \overline{u} \, \widehat{a}_{-\tau})$$ where $\lambda, \tau \in \mathbb{Z}^2$, $z,u\in \mathbb{C}$ and $c_1, c_2\in \R$. The space ${\bf A}$ is a real centered Gaussian Hilbert subspace
of $L^2(\mathbb{P})$.

\begin{defn}\label{d:chaos}{\rm For a given integer $q\ge 0$, the $q$-th {\bf Wiener chaos} associated with ${\bf A}$, denoted by $C_q$, is the closure in $L^2(\mathbb{P})$ of all real finite linear combinations of random variables of the type
$$
\prod_{j=1}^k H_{p_j}(\xi_j),
$$
with $k\ge 1$, where the integers $p_1,...,p_k \geq 0$ verify $p_1+\cdots+p_k = q$, and $(\xi_1,...,\xi_k)$ is a centered standard real Gaussian vector contained in ${\bf A}$ (so that $C_0 = \mathbb{R}$).}
\end{defn}

In view of the orthonormality and completeness of $\mathbb{H}$ in $L^2(\gamma)$, it is not difficult to show that $C_q \,\bot\, C_{q'}$ (where the orthogonality holds in  $L^2(\mathbb{P})$) for every $q\neq q'$, and moreover
\begin{equation*}
L^2(\Omega, \sigma({\bf A}), \mathbb{P}) = \bigoplus_{q=0}^\infty C_q;
\end{equation*}
the previous relation simply indicates that every real-valued functional $F$ of ${\bf A}$ can be uniquely represented in the form
\begin{equation}\label{e:chaos2}
F = \sum_{q=0}^\infty {\rm proj}(F \, | \, C_q)=\sum_{q=0}^\infty F[q],
\end{equation}
where  $F[q]:={\rm proj}(F \, | \, C_q)$ stands for the the projection of $F$ onto $C_q$, and the series converges in $L^2(\mathbb{P})$. By definition, one has $F[0]={\rm proj}(F \, | \, C_0) = \E [F]$. See e.g. \cite[Theorem 2.2.4]{NP} for further details.

\subsection{About gradients}

Differentiating both terms in \eqref{e:realim} yields that, for $j=1,2$,
\begin{equation}\label{e:partial}
\partial_j T_n(x) = \frac{2\pi i}{\sqrt{\mathcal{N}_n} }\sum_{(\lambda_1,\lambda_2)\in \Lambda_n} \lambda_j a_\lambda e_\lambda(x),\,\, \mbox{and} \,\,\, \partial_j \hT(x) = \frac{2\pi i}{\sqrt{\mathcal{N}_n} }\sum_{(\lambda_1,\lambda_2)\in \Lambda_n} \lambda_j \widehat{a}_\lambda e_\lambda(x)
\end{equation}
(where we used the shorthand notation $\partial_j = \frac{\partial}{\partial x_j}$).
It follows that, for every $n \in S$ and every $x\in\Tb$,
\begin{equation}\label{e:six}
T_{n}(x),\, \partial_{1}T_{n}(x), \, \partial_{2}T_{n}(x), \hT(x),\, \partial_{1}\hT(x), \, \partial_{2}\hT(x) \in \bf A.
\end{equation}
Another important fact (that one can check by a direct computation) is that, for fixed $x\in \mathbb{T}$, the six random variables appearing in \eqref{e:six} are stochastically independent.

\medskip

Routine computations (see also \cite[Lemma 2.3]{RW}) yield that $${\rm Var}(\partial_j T_n(x)) ={\rm Var}(\partial_j \hT(x)) =\frac{E_n}{2},$$ for any $j=1,2$,  any $n$ and any $x\in\mathbb{T}$. Accordingly, we will denote by $\widetilde{\partial}_j$ the normalised derivative
\[
\widetilde{\partial}_j:=\sqrt{\frac{2}{E_n}}\,\frac{\partial}{\partial x_j},
\]
and adopt the following (standard) notation for the gradient and its normalised version:
$$
\nabla  :=\left( \begin{matrix} \partial_1 \\ \partial_2  \end{matrix}\right),\qquad \widetilde \nabla  :=\left( \begin{matrix} \widetilde \partial_1 \\ \widetilde \partial_2  \end{matrix}\right).
$$

\subsection{Leonov-Shiryaev formulae}\label{ss:ls}

In the proof of our variance estimates, we will crucially use the following special case of the so-called {\bf Leonov-Shiryaev combinatorial formulae} for computing joint moments. It follows immediately e.g. from \cite[Proposition 3.2.1]{PT}, by taking into account the independence of $T_n$ and $\hT$, the independence of the six random variables appearing in \eqref{e:six}, as well as the specific covariance structure of Hermite polynomials (see e.g. \cite[Proposition 2.2.1]{NP}).
{\blue
\begin{proposition}\label{p:ls} Fix $n\in S$ and write
\begin{eqnarray*}
&& (X_0(x), X_1(x), X_2(x), Y_0(x),  Y_1(x), Y_2(x))\\
&&\quad\quad\quad := (T_{n}(x),\, \widetilde{\partial}_{1}T_{n}(x), \, \widetilde{\partial}_{2}T_{n}(x), \hT(x),\, \widetilde{\partial}_{1}\hT(x), \, \widetilde{\partial}_{2}\hT(x)), \quad x\in \Tb.
\end{eqnarray*}
Consider integers $p_0,p_1,p_2, q_0, q_1, q_2\geq 0$ and $a_0,a_1,a_2, b_0, b_1, b_2\geq 0$, and write
\begin{eqnarray*}
(X^\star_1(x),...,X^{\star}_{p_0+p_1+p_2}(x))&:=& (\underbrace{X_0(x),...,X_0(x)}_{p_0\mbox{ times}} ,\underbrace{X_1(x),...,X_1(x)}_{p_1\mbox{ times}},\underbrace{ X_2(x),...,X_2(x)}_{p_2 \mbox{ times}} )\\
(X^{\star\star}_1(y),...,X^{\star\star}_{a_0+a_1+a_2}(y))&:=& (\underbrace{X_0(y),...,X_0(y)}_{a_0 \mbox{ times}} ,\underbrace{X_1(y),...,X_1(y)}_{a_1\mbox{ times}},\underbrace{ X_2(y),...,X_2(y)}_{a_2 \mbox{ times}} ) \\ 
(Y^\star_1(x),...,Y^{\star}_{q_0+q_1+q_2}(x))&:=& (\underbrace{Y_0(x),...,Y_0(x)}_{q_0\mbox{ times}} ,\underbrace{Y_1(x),...,Y_1(x)}_{q_1\mbox{ times}},\underbrace{ Y_2(x),...,Y_2(x)}_{q_2 \mbox{ times}} )\\
(Y^{\star\star}_1(y),...,Y^{\star\star}_{b_0+b_1+b_2}(y))&:=& (\underbrace{Y_0(y),...,Y_0(y)}_{b_0 \mbox{ times}} ,\underbrace{Y_1(y),...,Y_1(y)}_{b_1\mbox{ times}},\underbrace{ Y_2(y),...,Y_2(y)}_{b_2 \mbox{ times}} ).
\end{eqnarray*}

 Then, for every $x,y \in \Tb$,
\begin{eqnarray}\label{e:ls}
&&\E\left[\prod_{j=0}^2 H_{p_j}(X_j(x))H_{a_j}(X_j(y))\prod_{k=0}^2 H_{q_k}(Y_k(x))H_{b_k}(Y_k(y))\right] \\
&& \notag  = {\bf 1}_{\{ p_0+p_1+p_2 = a_0+a_1+a_2\}} {\bf 1}_{\{ q_0+q_1+q_2 = b_0+b_1+b_2\}}  \times\\
&&\quad\quad\times \sum_{\sigma, \pi} \left( \prod_{j=1}^{p_0+p_1+p_2}  \, \E[X^{\star}_j(x) X^{\star\star}_{\sigma(j)}(y)]\right)\left( \prod_{k=1}^{q_0+q_1+q_2}  \, \E[Y^{\star}_k (x) Y^{\star\star}_{\pi(k)}(y)]\right),\notag
\end{eqnarray}
where the sum runs over all permutations $\sigma, \pi$ of $\{1,..., p_0+p_1+p_2\}$ and of $\{1,...,q_0+q_1+q_2\}$, respectively.
\end{proposition}
}

\subsection{Arithmetic facts}\label{ss:arit}

We will now present three results having an arithmetic {flavour}, that will be extensively used in the proofs of our main findings. To this end, for every $n\in S$ we introduce the $4$- and $6$-{\bf correlation set of frequencies}
\begin{eqnarray*}
S_4(n) &:=& \{{\boldsymbol{\lambda}}=(\lambda,\lambda',\lambda'',\lambda''')\in\Lambda_n^4 : \lambda- \lambda'+\lambda'' -\lambda''' =0\},\\
S_6(n) &:=& \{{\boldsymbol{\lambda}}=(\lambda,\lambda',\lambda'',\lambda''', \lambda'''', \lambda^{ v})\in\Lambda_n^6 : \lambda- \lambda'+\lambda'' -\lambda''' +\lambda^{{} iv} - \lambda^{ v}=0\}.
\end{eqnarray*}

The first statement exploited in our proofs yields an exact value for $|S_4(n)|$; the proof is based on an elegant geometric argument due to Zygmund \cite{Zy}, and is included for the sake of completeness.

\begin{lemma}\label{l:zygmund}
For every $n\in S$, every element of $S_4(n)$ has necessarily one of the following four (exclusive) structures:
\begin{itemize}
\item[\rm a)] $\lambda=\lambda'$, and $\lambda''=\lambda'''$;
\item[\rm b)] $\lambda = -\lambda'=-\lambda'' = \lambda'''$;
\item[\rm c)] $\lambda\notin \{\lambda',-\lambda'\}$, $\lambda=-\lambda''$, and $\lambda'=-\lambda'''$;
\item[\rm d)] $\lambda\notin \{\lambda',-\lambda'\}$, $\lambda=\lambda'''$, and $\lambda'=\lambda''$.
\end{itemize}
In particular, $|S_4(n)| = 3\mathcal{N}_n(\mathcal{N}_n-1)$.
\end{lemma}
\noindent\begin{proof}
We {{} partition} the elements of $S_4(n)$ into three disjoint subset: $S_1 = \{\boldsymbol{\lambda}:\lambda=\lambda'\}\, {\cap \, S_4(n)}$, $S_2 = \{\boldsymbol{\lambda}:\lambda=-\lambda'\} \, {\cap \, S_4(n)}$, and $S_3 = \{\boldsymbol{\lambda}:\lambda\notin \{\lambda',-\lambda'\} \} \, {\cap \, S_4(n)}$. If ${\boldsymbol{\lambda}} \in S_1$, then necessarily $\lambda''=\lambda'''$, and consequently $|S_1| = \mathcal{N}_n^2$. If ${\boldsymbol{\lambda}}\in S_2$, then the relations $2\lambda+\lambda''=\lambda'''$ and $2\lambda-\lambda'''=-\lambda''$ show that both $-\lambda''$ and $\lambda'''$ must belong to the intersection of the circle $C_0$ of radius $\sqrt{n}$ centered at the origin (since they are both in $\Lambda_n$) with the circle $C'$ of radius $\sqrt{n}$ centered in $2\lambda$; since $C_0\cap C' = \{\lambda\}$, we conclude that necessarily $\lambda=-\lambda'=-\lambda''=\lambda'''$. Plainly, this argument also yields $|S_2| = \mathcal{N}_n$. If ${\boldsymbol{\lambda}} \in S_3$, then the relations $\lambda-\lambda'+\lambda'' = \lambda'''$ and $\lambda-\lambda'-\lambda''' = -\lambda''$ show that both $-\lambda''$ and $\lambda'''$ must belong to the intersection of the circle $C_0$ of radius $\sqrt{n}$ centered at the origin (again, since they are both in $\Lambda_n$) with the circle $C''$ of radius $\sqrt{n}$ centered in $\lambda-\lambda'$; since $C_0\cap C'' = \{\lambda, -\lambda'\}$, we conclude that necessarily either $-\lambda''=\lambda$ and $-\lambda'''=\lambda'$ or $\lambda'''=\lambda$ and $\lambda''=\lambda'$ (the configurations such that $-\lambda''=\lambda'''$ are excluded by the requirement that $\lambda\neq -\lambda'$). This yields $|S_3| = 2\mathcal{N}_n(\mathcal{N}_n-2)$. Summing up,
$$
|S_4(n)| = |S_1|+|S_2|+|S_3| = \mathcal{N}_n^2+ \mathcal{N}_n+2\mathcal{N}_n (\mathcal{N}_n-2) = 3\mathcal{N}_n(\mathcal{N}_n-1).
$$
\end{proof}

The second estimate involves $6$-correlations, and follows from a deep result by Bombieri and Bourgain \cite[Theorem 1]{B-B} --- see also \cite[Theorem 2.2]{KKW} for a slightly weaker statement.

\begin{lemma}[\bf See \cite{B-B}]\label{lemma BB}
As $\mathcal{N}_n\to \infty$,
$$
|S_6(n)| = O\left(\mathcal N_n^{7/2}  \right).
$$
\end{lemma}

We will also need the following elementary fact about the behaviour of the correlation function $r_n$, as defined in \eqref{e:cov}, in a small square containing the origin.

\begin{proposition}\label{gio} Let $n\in S$, with $n\geq 1$, let $c_n= ( 1000 \sqrt {n} )^{-1} $, and $Q_n  := \{(x,y)\in \R^2 : |x|, |y| \leq c_n\}$. Assume that $z = (x,y) \in Q_n$ is such that $r_n(z) = \pm 1$; then, $z=0$.
\end{proposition}
\noindent\begin{proof} Assume first that $r_n(z) = 1$. Then, for every $(\lambda_1,\lambda_2) \in \Lambda_n$, one has necessarily that there exist $j,k,l\in \mathbb{Z}$ such that (i) $\lambda_1 x+\lambda_2 y = j$, (ii) $-\lambda_1 x+\lambda_2 y = k$, and (iii) $\lambda_1 y+\lambda_2 x = l$. Assume first that $\lambda_1 = 0$ (and therefore $|\lambda_2| = \sqrt{n} $): equation (i) implies that $|y| = |j|/\sqrt{n}$, and such an expression is $>c_n$ unless $j=y=0$; similarly, equation (iii) implies that $|x|>c_n$, unless $x=l=0$. The case where $\lambda_2 = 0$ is dealt with analogously. Now assume that $\lambda_1, \lambda_2\neq 0$: equations (i) and (ii) imply therefore that $y= (j+k)/2\lambda_2$ and $x = (j-k)/2\lambda_1$; since $|\lambda_i| \leq \sqrt{n}$, for $i=1,2$,  we infer that $|x|, |y|\leq c_n$ if and only if $j+k = 0= j-k$, and therefore $x=y=j=k=0$.
If $r_n(z) = -1$, then, for every $(\lambda_1,\lambda_2) \in \Lambda_n$, one has necessarily that there exist $j,k,l\in \mathbb{Z}+\frac12$ such that equations (i), (ii), (iii) above are verified: reasoning exactly as in the first part of the proof, we conclude that $\max\{|x|, |y|\}>c_n$, and consequently $z$ cannot be an element of $Q_n$.

\end{proof}

{{}

Finally, we will make use of the following result, corresponding to a special case of \cite[Corollary 9, p. 80]{Kov}: it yields a local ersatz of B\'ezout theorem for systems of equations involving trigonometric polynomials. We recall that, given a smooth mapping $U : \R^2 \to \R^2$ and a point $x\in \R^2$ such that $U(x) = (0,0)$, one says that $x$ is {\it nondegenerate} if the Jacobian matrix of $U$ at $x$ is invertible.

\begin{lemma}[See \cite{Kov}]\label{l:few} Fix $n\in S$, and consider two trigonometric polynomials on $\R^2$:
$$
P(x) = c+ \sum_{\lambda\in \Lambda_n} a_\lambda e_\lambda(x), \quad \mbox{and} \quad Q(x) = c'+ \sum_{\lambda\in \Lambda_n} b_\lambda e_\lambda(x),
$$
where $c, c'\in \R$ and the complex numbers $\{a_\lambda, b_\lambda\}$ verify the following:
\begin{itemize}
\item[--] for every $\lambda \in \Lambda_n$, one has that $a_\lambda = \overline{a_{-\lambda}}$ and $b_\lambda = \overline{b_{-\lambda}}$;

\item[--] every solution of the system $(P(x), Q(x)) = (0,0)$ such that $\|x\| < \pi/\sqrt{n}$ is nondegenerate.
\end{itemize}
Then, the number of solutions of the system $(P(x), Q(x)) = (0,0)$ contained in the open window $W := \{x\in \R^2 : \|x\|< \pi/\sqrt{n}\}$ is bounded by a universal constant $\alpha(n) \in (0,\infty)$ depending uniquely on $\mathcal{N}_n = | \Lambda_n|$.

\end{lemma}

}
\medskip

The next section contains a precise description of the strategy we will adopt in order to attack the proof of Theorem \ref{t:main}

\section{Structure of the proof of Theorem \ref{t:main}}\label{s:approach}


\subsection{{Chaotic projections and cancellation phenomena}}

We will start by showing in Lemma \ref{lemma approx} that $I_n$ can be formally obtained in $L^2(\P)$ as
\begin{equation}\label{formula integrale1}
I_n = \int_{\mathbb T} \delta_{\bf 0}({\bf T}_n(x)) \left | J_{{\bf T}_n}(x)\right |\,dx,
\end{equation}
where $\delta_{\bf 0}$ denotes the Dirac mass in ${\bf 0}=(0,0)$, $J_{{\bf T}_n}$ is the Jacobian matrix
$$
J_{{\bf T}_n}=\left( \begin{matrix}
&\partial_1 T_n &\partial_2 T_n\\
&\partial_1 \widehat T_n &\partial_2 \widehat T_n
\end{matrix}\  \right )
$$
and $| J_{{\bf T}_n}|$ is shorthand for the absolute value of its determinant. Since $I_n$ is a square-integrable functional of a Gaussian field, according to the general decomposition \eqref{e:chaos2} one has that
\begin{equation}\label{series}
I_n = \sum_{q\ge 0} I_n[q],
\end{equation}
where $I_n[q]=\text{proj}(I_n|C_q)$ denotes the orthogonal projection of $I_n$ onto the $q$-th Wiener chaos $C_q$. Since $I_n[0]=\E[I_n]$, the computation of the $0$-order chaos projection will allow us to conclude the proof of Part 1 of Theorem \ref{t:main} in Section \ref{proof media sec}.

One crucial point in our analysis is that, as proved in Lemma \ref{berry's cancellation}, the projections of $I_n$ onto odd-order Wiener chaoses vanish and, more subtly, also the second chaotic component disappears. {Namely}, we will show that, for every $n\in S$, it holds
$$
I_n[q]=0 \qquad \text{for odd } q\ge 1
$$
and moreover
\begin{equation}\label{eq chaos 2}
I_n[2]=0.
\end{equation}
Our proof of \paref{eq chaos 2} is based on Green's identity and the properties of Laplacian eigenfunctions (see also \cite[Section 7.3 and p.134]{Ro}).

\subsection{Leading term: fourth chaotic projections}

The first non-trivial chaotic projection of $I_n$ to investigate is therefore $I_n[4]$. One of the main achievements of our paper is an explicit computation of its asymptotic variance, as well as a proof that it gives the dominant term in the asymptotic behaviour of the total variance $\Var(I_n)=\sum_{q\ge 2} \Var(I_n[2q])$. The forthcoming Propositions \ref{varianze chaos}, \ref{asymptotic} and \ref{limite 4}, that we will prove in Section \ref{sec chaos 4}, are the key steps in order to achieve our goals.
\begin{proposition}\label{varianze chaos}
Let $\lbrace n_j\rbrace_j\subset S$ be such that $\mathcal N_{n_j}\to +\infty$ and $|\widehat{\mu}_{n_j}(4)|\to \eta$. Then
\begin{equation*}
\Var(I_{n_j}[4]) = d(\eta)\,  \frac{E_{n_j}^2}{\mathcal N_{n_j}^2}(1+o(1)),
\end{equation*}
where
\begin{equation*}
d(\eta) = \frac{3\eta^2 + 5}{128\pi^2}.
\end{equation*}
\end{proposition}
In view of Remark \ref{r:postmain}(2), Proposition \ref{varianze chaos} coincides with Part 2 of Theorem \ref{t:main}, once we replace $I_{n_j}[4]$ with $I_{n_j}$.
Let us now set, for $n\in S$,
\begin{eqnarray}\label{momento 4}
R_n(4) &:=& \int_{\mathbb T} r_n(x)^4\,dx= \frac{|S_n(4)|}{\mathcal {N}_n^4} =\frac{3\mathcal{N}_n(\mathcal{N}_n-1)}{\mathcal{N}_n^4} , \\
R_n(6)& :=& \int_{\mathbb T} r_n(x)^6\,dx= \frac{|S_n(6)|}{\mathcal N_n^6},\label{momento 6}
\end{eqnarray}
where $S_n(4), S_n(6)$ are the sets of $4$- and $6$-correlation coefficients defined in Section \ref{ss:arit}, and we have used Lemma \ref{l:zygmund} in \paref{momento 4}. The following result (Proposition \ref{asymptotic}), combined with Proposition \ref{varianze chaos} and Lemma \ref{lemma BB} allows us to conclude that, as $\mathcal{N}_n\to \infty$,
\begin{equation}\label{ciao}
\Var(I_{n}) \sim \Var(I_{n}[4]),
\end{equation}
thus achieving the proof of Part 2 of Theorem \ref{t:main}. Note that, by virtue of Lemma \ref{lemma BB} and \eqref{momento 6}, as $\mathcal{N}_n\to \infty$ one has that
$$
R_n(6) = o\left(\frac{1}{\mathcal{N}_n^2}\right) \quad \mbox{or, equivalently,} \quad R_n(6) = o\left(R_n(4)\right).
$$

\begin{proposition}\label{asymptotic}
As $\mathcal N_n\to +\infty$, we have
$$
\sum_{q\ge 3} \Var\left( I_n[2q]  \right) = O\left(E_n^2 R_n(6)  \right).
$$
\end{proposition}

Part 3 of Theorem \ref{t:main} follows immediately from the relation
$$
\mathbb{P}\left[ \left| \frac{I_{n}}{\pi n} - 1 \right| >\epsilon \right]\leq \frac{ \Var({I_{n}}/(\pi n) )^{1/2} }{\epsilon}
$$
(which is a consequence of the Markov inequality), as well as from Part 1 and Part 2 of the same Theorem. Finally, the proof of Part 4 of Theorem \ref{t:main} relies on a careful and technical investigation of $I_n[4]$, leading us to the following result, which indeed coincides with \eqref{e:j}, once replacing $\frac{I_{n_j}[4]}{\sqrt{\Var(I_{n_j}[4])}}$ with $\widetilde I_{n_j}$.
\begin{proposition}\label{limite 4}
Let $\lbrace n_j\rbrace_j\subset \lbrace n\rbrace$ be a subsequence such that $\mathcal N_{n_j}\to +\infty$ and $|\widehat \mu_{n_j}(4)|\to \eta$, then
$$
\frac{I_{n_j}[4]}{\sqrt{\Var(I_{n_j}[4])}}\, \stackrel{\rm law}{\Longrightarrow} \,
\mathcal{J}_\eta,
$$
where $\mathcal{J}_\eta$ is defined in \eqref{e:j}.
\end{proposition}

\subsubsection{Controlling the variance of higher-order chaoses}

In order to prove Proposition \ref{asymptotic}, we need to carefully control the remainder given by $\sum_{q\ge 3}\Var(I_n[2q])$; our argument (consisting in a substantial extension of techniques introduced in \cite[\S 6.1]{ORW} and \cite[\S 4.3]{RW2}) is the following.

We {{} partition} the torus into {a} union of disjoint squares $Q$ of side length $1/M$, where $M$ is proportional to $\sqrt{E_n}$. Of course
\begin{equation}\label{eq sum}
I_n=\sum_Q I_{n_{|_Q}},
\end{equation}
where $I_{n_{|_Q}}$ is the number of zeroes contained in $Q$. It holds that, for every $q\ge 0$, $
I_n[q]=\sum_Q I_{n_{|_Q}}[q]$ and hence
\begin{equation}\label{sum cov}
\Var\left( \sum_{q\ge 3} I_n[2q]  \right) = \sum_{Q,Q'} \Cov\left(\text{proj} \left(I_{n_{|_Q}}|C_{\ge 6}\right)   , \text{proj} \left(I_{n_{|_{Q'}}} |C_{\ge 6}\right)\right),
\end{equation}
where $\text{proj}\left(\cdot |C_{\ge 6}   \right)$ denotes the orthogonal projection onto $\bigoplus_{q\ge 6} C_q$, that is, the orthogonal sum of Wiener chaoses of order larger or equal than six.

We now split the double sum on the RHS of \paref{sum cov} into two parts: namely, one over {\it singular} pairs of cubes and the other one over {\it non-singular} pairs of cubes. Loosely speaking, for a pair of non-singular cubes $(Q,Q')$, we have that for every $(z,w)\in Q\times Q'$, the covariance function $r_n$ of the field $T_n$ and all its normalized derivatives up to the order two $\widetilde \partial_i r_n, \widetilde \partial_{ij} r_n:=(E_n/2)^{-1}\partial/\partial_{x_i x_j} r_n$ for $i,j=1,2$ are bounded away from 1 and $-1$, once evaluated in $z-w$  (see Definition \ref{sing point} and Lemma \ref{lemma sing}).
\begin{lemma}[Contribution of the singular part]\label{varianza sing}
As $\mathcal N_n\to +\infty$,
\begin{equation}\label{eq sing}
\left | \sum_{(Q,Q')\text{ sing.}} \Cov\left (\text{proj}\left (I_{n_{ |_Q}}|C_{\ge 6}\right ), \text{proj}\left (I_{n_{ |_{Q'}}}|C_{\ge 6}\right )\right ) \right | \ll E_n^2 R_n(6).
\end{equation}
\end{lemma}
In order to show Lemma \ref{varianza sing} (see Section \ref{proof varianza sec}), we use the Cauchy-Schwarz inequality and the stationarity of ${\bf T}_n$, in order to reduce the problem to the investigation of nodal intersections in a small square $Q_0$ around the origin: for the LHS of \paref{eq sing} we have
\begin{equation*}
\begin{split}
&\left | \sum_{(Q,Q')\text{ sing.}} \Cov\left (\text{proj}\left (I_{n_{ |_Q}}|C_{\ge 6}\right ), \text{proj}\left (I_{n_{ |_{Q'}}}|C_{\ge 6}\right )\right ) \right |\cr
&\ll  \sum_{(Q,Q')\text{ sing.}} \E\left[I_{n_{ |_{Q_0}}}\left(I_{n_{ |_{Q_0}}}-1   \right) \right] +  \E\left[I_{n_{ |_{Q_0}}} \right].
\end{split}
\end{equation*}
Thus, we need to (i) count the number of singular pairs of cubes, (ii) compute the expected number of nodal intersections in $Q_0$ and finally (iii) calculate the second factorial moment of $I_{n_{ |_{Q_0}}}$. Issue (i) will be dealt with by exploiting the definition of singular pairs of cubes and the behavior of the moments of the derivatives of $r_n$ on the torus (see Lemma \ref{bello}), thus obtaining that
$$
| \lbrace (Q,Q')\text{ sing.}  \rbrace | \ll E_n^2 R_n(6).
$$
Relations \eqref{e:locexp} and \paref{eq sum} yield immediately that  $\E\left[I_{n_{ |_{Q_0}}} \right]$ is bounded by a constant independent of $n$.

To deal with (iii) is much subtler matter. Indeed, we need first to check the assumptions for Kac-Rice formula (see \cite[Theorem 6.3]{AW} ) to hold in  Proposition \ref{gio}.  The latter allows us to write the second factorial moment $ \E\left[I_{n_{ |_{Q_0}}}\left(I_{n_{ |_{Q_0}}}-1   \right) \right]$ as an integral on $Q_0\times Q_0$ of the so-called {\bf two-point correlation} function $K_2$, given by
$$
K_2(x,y) := p_{({\bf T}_n(x),{\bf T}_n(y))}({\bf 0},{\bf 0})\E\left[ \left | J_{{\bf T}_n}(x)\right |  \left | J_{{\bf T}_n}(y)\right |\Big |{\bf T}_n(x)={\bf T}_n( y)={\bf 0}   \right],
$$
where {$x,y\in \mathbb T$ and} $p_{({\bf T}_n(x),{\bf T}_n(y))}$ is the density of $({\bf T}_n(x),{\bf T}_n(y))$.

The stationarity of ${\bf T}_n$ then reduces the problem to investigating $K_2(x):=K_2(x,0)$ around the origin. Hypercontractivity properties of Wiener chaoses and formulas for the volume of ellipsoids (see \cite{KZ}) yield the following estimation
\begin{equation}\label{psi n}
K_2(x)\ll \frac{|\Omega_n(x)|}{1-r_n(x)^2}=: \Psi_n(x),
\end{equation}
where $| \Omega_n(x)|$ stands for the absolute value of the determinant of the matrix $\Omega_n(x)$, defined as the covariance matrix of the vector $\nabla T_n(0)$, conditionally on $T_n(x)=T_n(0)=0$.  An explicit Taylor expansion at $0$ for $\Psi_n$  (made particularly arduous by the diverging integral in \eqref{e:explosion} --- see Lemma \ref{lemma taylor}) will allow us to prove that $ \E\left[I_{n_{ |_{Q_0}}}\left(I_{n_{ |_{Q_0}}}-1   \right) \right]$ is also bounded by a constant independent of $n$. This concludes the proof of Lemma \ref{varianza sing}.


To achieve the proof of Theorem \ref{t:main}, we will eventually show the following result, whose proof relies on Proposition \ref{p:ls}, on the definition of non-singular cubes, as well as on the behavior of even moments  of derivatives of the covariance function $r_n$.
\begin{lemma}[Contribution of the non-singular part]\label{varianza nonsing}
 As $\mathcal N_n\to +\infty$, we have
\begin{equation*}
  \left| \sum_{(Q,Q')\text{ non sing.}} \Cov\left (\text{proj}\left(I_{n_{|_{Q}}}|C_{\ge 6}\right),\text{proj}\left(I_{n_{|_{Q'}}}|C_{\ge 6}\right) \right) \right|=O\left( E_n^2 R_n(6) \right).
\end{equation*}
\end{lemma}

The rest of the paper contains the formal proofs of all the statements discussed in the present section.

\section{Phase singularities and Wiener chaos}\label{wiener sec}

\subsection{Chaotic expansions for phase singularities}\label{chaos sec}

In this part we find the chaotic expansion \paref{series} for $I_n$. The first achievement  in this direction is the following approximation result.

\subsubsection{An integral expression for the number of zeros}\label{sec approx}

For $\eps>0$ and $n\in S$, we consider the $\eps$-approximating random variable
\begin{equation}\label{formula integrale1eps}
I_n(\varepsilon):=\frac{1}{4\eps^2}
\int_{\mathbb T}{\mathbf 1}_{[-\varepsilon,\varepsilon]^2}({\mathbf T}_n(x))
|J_{{\mathbf T}_n}(x)|\,dx,
\end{equation}
where ${\bf 1}_{[-\eps, \eps]^2}$ denotes the indicator function of the square $[-\eps, \eps]^2$. The following result makes the formal equality in \paref{formula integrale1} rigorous.
\begin{lemma}\label{lemma approx}
For $n\in S$, with probability one, $I_n$ {{} is composed of a finite number of isolated points} and, as $\eps\to 0$,
\begin{equation}\label{e:cscs}
I_n(\varepsilon)\to I_n,
\end{equation}
both a.s. and in the {{} $L^p(\P)$-sense, for every $p\geq 1$. }
\end{lemma}
\noindent\begin{proof} {{} Fix $n\in S$. In order to directly apply some statements taken from \cite{AW}, we will canonically identify the random field $(x_1,x_2)\mapsto {\bf T}_n(x_1,x_2)$ with a random mapping from $\R^2$ to $\R^2$ that is 1-periodic in each of the coordinates $x_1, x_2$. In what follows, for $x\in \R^2$ we will write ${\bf T}_n(x,\omega)$ to emphasize the dependence of ${\bf T}_n(x)$ on $\omega\in \Omega$. We subdivide the proof into several steps, numbered from (i) to (vi).
\begin{itemize}

\item[(i)] First of all, since ${\bf T}_n$ is an infinitely differentiable stationary Gaussian field such that, for every $x\in \R^2$, the vector ${\bf T}_n(x)$ has a standard Gaussian distribution, one can directly apply \cite[Proposition 6.5]{AW} to infer that there exists a measurable set $\Omega_0\subset \Omega$ with the following properties: $\P(\Omega_0)=1$ and, for every $\omega \in \Omega_0$ and every $x\in \R^2$ such that ${\bf T}_n(x, \omega) =0$, one has necessarily that the Jacobian matrix $J_{{\bf T}_n}(x, \omega)$ is invertible.

\item[(ii)] A standard application of the inverse function theorem (see e.g. \cite[p. 136]{AT}) implies that, for every $\omega \in \Omega_0$, any bounded set $B\subset \R^2$ only contains a finite number of points $x$ such that ${\bf T}_n(x,\omega)=0$. This implies in particular that, with probability one, $\mathscr{I}_n$ (as defined in \eqref{e:zeroset}) is composed of a finite number of isolated points and $I_n <+\infty$.

\item[(iii)] Sard's Lemma yields that, for every $\omega \in \Omega_0$, there exists a set $ U_\omega \subset \R^2$ such that $U_\omega^c$ has Lebesgue measure $0$ and, for every $u\in U_\omega$ there is no $x\in \R^2$ such that ${\bf T}_n(x,\omega) = u$ and $J_{{\bf T}_n}(x,\omega)$ is not invertible. Note that, by definition, one has that $0\in U_\omega$ for every $\omega \in \Omega_0$.

{
\item[(iv)] Define $B := \{x = (x_1,x_2)\in \R^2 : 0\leq x_i<  1/L\}, \,\, i=1,2\}$, where $L$ is any positive integer such that $L>\sqrt{n}$. For every $u\in \R^2$, we set $I_{n,u}(B)$ to be the cardinality of the set composed of those $x\in B$ such that ${\bf T}_n(x) = u$; the quantity $I_{n,u}(\T)$ is similarly defined, in such a way that $I_{n,0}(\T)=I_n$.  Two facts will be crucial in order to conclude the proof: (a) for every $\omega \in \Omega_0$ and every $u=(u_1,u_2) \in U_\omega$, by virtue of Lemma \ref{l:few} as applied to the pair $(P,Q)$ given by
$$
P(x) = T_n(x,\omega)-u_1 \quad \mbox{and} \quad Q(x) = \widehat{T}_n(x,\omega)-u_2,
$$
as well as of the fact that $B\subset W$, one has that $I_{n,u}(B)(\omega)\leq \alpha(n)$, and (b) as a consequence of the inverse function theorem, for every $\omega \in \Omega_0$ there exists $\eta_\omega \in (0, \infty)$ such that the equality $I_n(\omega) = I_{n,u}(\T)(\omega)$ holds for every $u$ such that  $\| u \| \leq \eta_\omega$. Indeed, reasoning as in \cite[Proof of Theorem 11.2.3]{AT} if this was not the case, then there would exist a sequence $u_k \to 0$, $u_k \neq 0$,  and a point $x\in \T$ such that: (1) ${\bf T}_n(x,\omega) = 0$, and (2) for every neighborhood $V$ of $x$ (in the topology of $\T$) there exist $k\geq 1$ 
and $x_0, x_1\in V$ such that $x_0\neq x_1$ and ${\bf T}_n(x_0) = {\bf T}_n(x_1)  = u_k$ --- which is in contradiction with the inverse function theorem.


\item[(v)] By the area formula (see e.g. \cite[Proposition 6.1 and formula (6.2)]{AW}), one has that, for every $\omega \in \Omega_0$,
\begin{eqnarray}\label{e:asc}
&& \frac{1}{4\varepsilon^2}\int_{\T}{\mathbf 1}_{[-\eps, \eps]^2}({\mathbf T}_n(x, \omega))|J_{{\mathbf T}_n}(x, \omega )|dx\\
&&\quad\quad\quad\quad =\frac{1}{4\varepsilon^2}\int_{[-\varepsilon,\varepsilon]^2}I_{n,u}(\T)(\omega) \,du
= \frac{1}{4\varepsilon^2}\int_{[-\varepsilon,\varepsilon]^2\cap U_\omega}I_{n,u}(\T)(\omega) \,du,\notag
\end{eqnarray}
where we used the property that the complement of $U_\omega$ has Lebesgue measure 0. Since the integral on the right-hand side of \eqref{e:asc} equals $ I_{n}$ whenever $\eps \leq \eta_\omega/\sqrt{2}$, we conclude that \eqref{e:cscs} holds $\P$-a.s. .

\item[(vi)]  According to the discussion at Point (iv)-(a) above and using stationarity, one has that
$$
\P [I_n \leq L^2\alpha(n)] = \P\left[   \frac{1}{4\varepsilon^2}\int_{[-\varepsilon,\varepsilon]^2}I_{n,u}(\T) \,du \leq L^2\alpha(n)\right]= 1.
$$
The fact that \eqref{e:cscs} holds also in $L^p(\P)$ now follows from Point (v) and dominated convergence.}
\end{itemize}
}
\end{proof}

\subsubsection{Chaotic expansions}

Let us consider the collections of coefficients $\lbrace \beta_l : l\geq 0\rbrace$ and $\lbrace \alpha_{a, b, c, d} : a,b,c,d\geq 0\}$ defined as follows. For $l\ge 0$
\begin{equation}\label{e:beta}
\beta_{2l+1}:=0, \qquad \beta_{2l} := \frac{1}{\sqrt{2\pi}}H_{2l}(0),
\end{equation}
where (as before) $H_{2l}$ is the $2l$-th Hermite polynomial. For instance,
\begin{equation}\label{beta piccoli}
\beta_0=\frac{1}{{\sqrt{2\pi}}},\qquad \beta_2=-\frac{1}{{\sqrt{2\pi}}},\qquad \beta_4 = \frac{3}{\sqrt{2\pi}}.
\end{equation}
Also, we set
\begin{equation}\label{e:alpha}
\alpha_{a,b,c,d} := \E[|XW-YV|H_a(X)H_b(Y)H_c(V)H_d(W)],
\end{equation}
with $(X,Y,V, W)$  a standard real four-dimensional Gaussian vector. Note that
on the right-hand side of \paref{e:alpha}, $|XW-YV|$ is indeed the absolute value of the determinant of
the matrix
$$
\left ( \begin{matrix}
&X &Y\\
&V &W
\end{matrix} \ \right).
$$
\begin{lemma}\label{nullity}
If $a,b,c,d$ do not have the same parity, then $$\alpha_{a,b,c,d}=0.$$
\end{lemma}
\noindent\begin{proof}
Let us assume without loss of generality (by symmetry) that $a$ is odd and that at least one integer among $b$, $c$ and $d$ is even. We will exploit \eqref{e:sym}. If $b$ is even, then, using that
$(X,Y,V, W)\overset{\rm law}{=}(-X,-Y,V, W)$, one can write that
\begin{eqnarray*}
\alpha_{a,b,c,d}&=&\E[|XW-YV|H_a(X)H_b(Y)H_c(V)H_d(W)]\\
&=&\E[|YV-XW|H_a(-X)H_b(-Y)H_c(V)H_d(W)]\\
&=&-\E[|XW-YV|H_a(X)H_b(Y)H_c(V)H_d(W)]=-\alpha_{a,b,c,d},
\end{eqnarray*}
leading to $\alpha_{a,b,c,d}=0$.
If $c$ (resp. $d$) is even, the same reasoning based on
$(X,Y,V, W)\overset{\rm law}{=}(-X,Y,-V, W)$
(resp. $(X,Y,V, W)\overset{\rm law}{=}(-X,Y,V,-W)$)
 leads to the same conclusion.

\end{proof}

We will not need the explicit values of $\alpha_{a,b,c,d}$, unless $a+b+c+d\in \lbrace{0,2,4\rbrace }$. The following technical result will be proved in Section \ref{appendix}.

\begin{lemma}\label{alpha piccoli} It holds that
\begin{equation*}
\begin{split}
\alpha_{0,0,0,0}&=1,\\
\alpha_{2,0,0,0}&= \alpha_{0,2,0,0}=\alpha_{0,0,2,0}=\alpha_{0,0,0,2}=\frac12,\\
\alpha_{4,0,0,0} &=\alpha_{0,4,0,0}=\alpha_{0,0,4,0}=\alpha_{0,0,0,4}=-\frac{3}{8},\\
\alpha_{2,2,0,0}&=\alpha_{0,0,2,2}=\alpha_{2,2,0,0}=-\frac{1}{8},\\
\alpha_{2,0,2,0}&=\alpha_{0,2,0,2}-\frac{1}{8},\\
\alpha_{2,0,0,2}&=\alpha_{0,2,2,0}=\frac{5}{8},\\
\alpha_{1,1,1,1}&=-\frac{3}{8}.
\end{split}
\end{equation*}
\end{lemma}

\begin{lemma}[Chaotic expansion of $I_n$]\label{berry's cancellation}
For $n\in S$ and $q\ge 0$, we have
\begin{equation}\label{chaos dispari}
I_n[2q+1]=0,
\end{equation}
and
\begin{equation}\label{proiezione q}
\begin{split}
I_n[2q]=&\frac{E_n}{2}\sum_{i_1+i_2+i_3+j_1+j_2+j_3=2q} \frac{\beta_{i_1}\beta_{j_1}}{i_1! j_1!}\frac{\alpha_{i_2 i_3 j_2 j_3}}{i_2! i_3! j_2! j_3!}\times\cr
&\!\!\!\!\times \int_{\mathbb T} H_{i_1}(T_n(x))H_{j_1}(\widehat T_n(x))H_{i_2}(\widetilde \partial_1 T_n(x))H_{i_3}(\widetilde \partial_2 T_n(x))H_{j_2}(\widetilde \partial_1 \widehat T_n(x))H_{j_3}(\widetilde \partial_2 \widehat T_n(x))\,dx,
\end{split}
\end{equation}
where the sum can be restricted to the set of those indices $(i_1,j_1,i_2,i_3,j_2,j_3)$ such that $i_1, j_1$ are even and $i_2, i_3, j_2, j_3$ have the same parity.
In particular,
\begin{equation}\label{chaos 2}
I_n[2]=0.
\end{equation}
The chaotic expansion for $I_n$ is hence
\begin{equation}\label{chaos exp}
\begin{split}
I_n=&I_n[0]+\sum_{q\ge 2}\frac{E_n}{2}\sum_{i_1+i_2+i_3+j_1+j_2+j_3=2q} \frac{\beta_{i_1}\beta_{j_1}}{i_1! j_1!}\frac{\alpha_{i_2 i_3 j_2 j_3}}{i_2! i_3! j_2! j_3!}\times\cr
&\times \int_{\mathbb T} H_{i_1}(T_n(x))H_{j_1}(\widehat T_n(x))H_{i_2}(\widetilde \partial_1 T_n(x))H_{i_3}(\widetilde \partial_2 T_n(x))H_{j_2}(\widetilde \partial_1 \widehat T_n(x))H_{j_3}(\widetilde \partial_2 \widehat T_n(x))\,dx,
\end{split}
\end{equation}
where the sum runs over the set of those indices $(i_1,j_1,i_2,i_3,j_2,j_3)$ such that $i_1, j_1$ are even and $i_2, i_3, j_2, j_3$ have the same parity.
\end{lemma}
\noindent\begin{proof} The main idea is to deduce the chaotic expansion for $I_n$ from the chaotic expansion for \paref{formula integrale1eps} and Lemma \ref{lemma approx}.
Let us first rewrite \paref{formula integrale1eps} as
\begin{equation}\label{areaformula}
I_n(\eps)={\frac{E_n}{8\eps^2}}\int_{\T} 1_{[-\eps,\eps]^2}({\bf T}_n(x))
\big|
\widetilde{\partial}_1T_n(x)\widetilde{\partial}_2\widehat{T}_n(x)
-
\widetilde{\partial}_1\widehat{T}_n(x)\widetilde{\partial}_2T_n(x)
\big|dx.
\end{equation}
We recall the chaos decomposition of the indicator function (see e.g. \cite[Lemma 3.4]{MPRW}):
\begin{equation*}
\frac{1}{2\varepsilon}{1}_{[-\varepsilon ,\varepsilon ]}(\cdot
)=\sum_{l=0}^{+\infty }\frac{1}{l!}\beta_l^{\varepsilon}\,H_{l}(\cdot ),
\end{equation*}
where, for $l \geq 0$
\begin{equation}\label{e:beta eps}
\beta_0^{\varepsilon} = \frac{1}{2\varepsilon} \int_{-\varepsilon}^{\varepsilon} \phi(t)\,dt,\qquad \beta_{2l+1}^\eps = 0,\qquad \beta^\eps_{2l+2}=  -\frac{1}{\varepsilon}
\phi \left (\varepsilon \right) H_{2l+1} \left (\varepsilon \right),
\end{equation}
and $\phi$ is still denoting the standard Gaussian density.
For the indicator function of $[-\eps,\eps]^2$ appearing in (\ref{areaformula}), we thus
have
\begin{equation}\label{delta0}
\frac{1}{4\varepsilon^2}1_{[-\eps,\eps]^2}(x,y) = \sum_{l=0}^\infty \sum_{q=0}^l \frac{\beta^\eps_{2q}\beta^\eps_{2l-2q}}{(2q)!(2l-2q)!}H_{2q}(x)H_{2l-2q}(y).
\end{equation}
The chaotic expansion for the absolute value of the Jacobian determinant appearing in \paref{areaformula} is, thanks to Lemma \ref{nullity},
\begin{equation}\label{determinant}
\begin{split}
&\big|
\widetilde{\partial}_1T_n(x)\widetilde{\partial}_2\widehat{T}_n(x)
-
\widetilde{\partial}_1\widehat{T}_n(x)\widetilde{\partial}_2T_n(x)
\big|\\
&=
\sum_{q\ge 0}\sum_{\substack{a+b+c+d=2q\\  (a,b,c,d\text{ the same parity}) }}^\infty \frac{\alpha_{a,b,c,d}}{a!b!c!d!}
H_a(\widetilde{\partial}_1 T_n(x))H_b(\widetilde{\partial}_2 T_n(x))H_c(\widetilde{\partial}_1  \widehat T_n(x))H_d(\widetilde{\partial}_2 \widehat T_n(x)),
\end{split}
\end{equation}
where $\alpha_{a,b,c,d}$ are given in \paref{e:alpha}. In particular, observe that Lemma \ref{nullity} ensures that the odd chaoses vanish in the chaotic expansion for the Jacobian.

 It hence follows from \paref{delta0} and \paref{determinant} that the chaotic expansion for $I_n(\eps)$ in \paref{areaformula} is (taking sums over even $i_1, j_1$  and $i_2, i_3, j_2, j_3$ with the same parity)
\begin{equation}\label{proiezione q eps}
\begin{split}
I_n(\eps)=&\frac{E_n}{2}\sum_{q\ge 0}\sum_{i_1+i_2+i_3+j_1+j_2+j_3=2q} \frac{\beta^\eps_{i_1}\beta^\eps_{j_1}}{i_1! j_1!}\frac{\alpha_{i_2 i_3 j_2 j_3}}{i_2! i_3! j_2! j_3!}\times\cr
&\times \int_{\mathbb T} H_{i_1}(T_n(x))H_{j_1}(\widehat T_n(x))H_{i_2}(\widetilde \partial_1 T_n(x))H_{i_3}(\widetilde \partial_2 T_n(x))H_{j_2}(\widetilde \partial_1 \widehat T_n(x))H_{j_3}(\widetilde \partial_2 \widehat T_n(x))\,dx.
\end{split}
\end{equation}
Noting that, as $\varepsilon\to 0$,
$$
\beta_{l}^\varepsilon \to \beta_l,
$$
where $\beta_l$ are given in \paref{e:beta} and using Lemma \ref{lemma approx}, we prove both \paref{chaos dispari} and \paref{proiezione q}.

\

\noindent Let us now prove \paref{chaos 2} that allows to conclude the proof.
Equation \paref{proiezione q} with $q=1$ together with Equation \paref{beta piccoli} and Lemma \ref{alpha piccoli}, imply that the projection of $I_n$ on the second Wiener chaos equals the quantity
\begin{eqnarray*}
I_n[2]&:=& 2\pi^2 n \beta_0\beta_2  \alpha_{0,0,0,0} \int_{\T} H_2(T_n(x)) dx
 +2\pi^2 n \beta_2\beta_0  \alpha_{0,0,0,0} \int_{\T} H_2(\widehat{T}_n(x)) dx \\
&& +2\pi^2 n \beta_0^2 \alpha_{2,0,0,0} \int_{\T} H_2(\widetilde{\partial}_1 T_n(x))dx
+  2\pi^2 n \beta_0^2 \alpha_{0,2,0,0} \int_{\T} H_2(\widetilde{\partial}_2 T_n(x)) dx \\
&&+ 2\pi^2 n \beta_0^2 \alpha_{0,0,2,0} \int_{\T} H_2(\widetilde{\partial}_1\widehat{T}_n(x))dx
+ 2\pi^2 n \beta_0^2 \alpha_{0,0,0,2} \int_{\T}  + H_2(\widetilde{\partial}_2\widehat{T}_n(x) )  dx \\
&= &\frac{\pi n}{2}  \left\{ \int_{\T} \big[H_2(\widetilde{\partial}_1 T_n(x)) +  H_2(\widetilde{\partial}_2 T_n(x)) + H_2(\widetilde{\partial}_1\widehat{T}_n(x)) + H_2(\widetilde{\partial}_2\widehat{T}_n(x) ) \big] dx \right.\\
&& \left. \quad\quad\quad- 2\int_{\T} \big[H_2(T_n(x)) + H_2(\widehat{T}_n(x))\big] dx\right\} .
\end{eqnarray*}
According to Green's first identity (see e.g. \cite[p. 44]{Lee}),


$$\int_{\T}\nabla v\cdot\nabla w\ dx=-\int_{\T}w\,\Delta v\ dx.$$


\noindent Using the facts that $H_2(t)=t^2-1$ {and that $T_n$ and $\widehat{T}_n$ are eigenfunctions of $\Delta$}, we eventually infer that
\begin{eqnarray*}
I_n[2]&= &\frac{1}{4\pi}  \int_{\T} \big[\|\nabla T_n(x)\|^2 +  \|\nabla \widehat{T}_n(x)\|^2  \big] dx
-n\pi \int_{\T} \big[T_n(x)^2 + \widehat{T}_n(x)^2\big] dx\\
&= &-\frac{1}{4\pi}  \int_{\T} \big[T_n(x)\, \Delta T_n(x) +   \widehat{T}_n(x)  \Delta  \widehat{T}_n(x)  \big] dx
-n\pi \int_{\T} \big[T_n(x)^2 + \widehat{T}_n(x)^2\big] dx\\
&= &n\pi \int_{\T} \big[T_n(x)^2 +   \widehat{T}_n(x)^2 \big] dx
-n\pi \int_{\T} \big[T_n(x)^2 + \widehat{T}_n(x)^2\big] dx=0.
\end{eqnarray*}

\end{proof}

\subsection{Proof of Part 1 of Theorem \ref{t:main}}\label{proof media sec}

According to Lemma \ref{alpha piccoli} and Equation \paref{beta piccoli}, for every $n\in S$ one has that
$$
I_n[0]=\mathbb{E}[I_n] = 2\pi^2 n\,\beta_0^2\,\alpha_{0,0,0,0} =\pi n =  \frac{E_n}{4\pi},
$$
thus yielding the desired conclusion.

\section{Investigation of the fourth chaotic components}\label{sec chaos 4}

In this section we shall investigate fourth chaotic components. In particular, we shall prove Proposition \ref{varianze chaos} and Proposition \ref{limite 4}.

\subsection{Preliminary results}

For $n\in S$, from \paref{proiezione q} with $q=2$ we deduce that
\begin{equation}\label{eq 4 chaos}
\begin{split}
I_n[4]&=\frac{E_n}{2}\sum_{i_1+i_2+i_3+j_1+j_2+j_3=4} \frac{\beta_{i_1}\beta_{j_1}}{i_1! j_1!}\frac{\alpha_{i_2 i_3 j_2 j_3}}{i_2! i_3! j_2! j_3!}\times\cr
&\times \int_{\mathbb T} H_{i_1}(T_n(x))H_{j_1}(\widehat T_n(x))H_{i_2}(\widetilde \partial_1 T_n(x))H_{i_3}(\widetilde \partial_2 T_n(x))H_{j_2}(\widetilde \partial_1 \widehat T_n(x))H_{j_3}(\widetilde \partial_2 \widehat T_n(x))\,dx.
\end{split}
\end{equation}
where the sum only considers integers $i_1,j_1$ even and $i_2, i_3, j_2, j_3$ with the same parity. In order to compute an expression for $I_n[4]$ that is more amenable to analysis,
let us introduce, for $n\in S$, the following family of random variables:
\begin{eqnarray*}
W(n) &=& \frac{1}{\sqrt{\mathcal N_n}}\sum_{\lambda\in \Lambda_n}(|a_\lambda|^2-1),\\ \widehat W(n) &= &\frac{1}{\sqrt{\mathcal N_n}}\sum_{\lambda\in \Lambda_n}(|\widehat a_\lambda|^2-1)\\
W_j(n)&=&\frac{1}{n\sqrt{\mathcal N_n}}\sum_{\lambda\in \Lambda_n}\lambda_j^2(|a_\lambda|^2-1),\\ \widehat W_j(n)&=&\frac{1}{n\sqrt{\mathcal N_n}}\sum_{\lambda\in \Lambda_n}\lambda_j^2(|\widehat a_\lambda|^2-1),\quad j=1,2 \\
W_{1,2}(n)&=&\frac{1}{n\sqrt{\mathcal N_n}}\sum_{\lambda\in \Lambda_n}\lambda_1\lambda_2\,|a_\lambda|^2,\\ \widehat W_{1,2}(n)&=&\frac{1}{n\sqrt{\mathcal N_n}}\sum_{\lambda\in \Lambda_n}\lambda_1\lambda_2\,|\widehat a_\lambda|^2, \\
M(n)&=&\frac{1}{\sqrt{\mathcal N_n}} \sum_{\lambda\in \Lambda_n} a_\lambda \overline{\widehat{a}_\lambda} \\
M_j(n) &=& \frac{i}{\sqrt{n\mathcal N_n}} \sum_{\lambda\in \Lambda_n} \lambda_j a_\lambda \overline{\widehat{a}_\lambda},\quad j=1,2 \\
M_{\ell,j}(n)&=&\frac{1}{n\sqrt{\mathcal N_n}}\sum_{\lambda\in \Lambda_n} \lambda_\ell \lambda_j  \ a_\lambda \overline{\widehat{a}_\lambda}\qquad j,\ell=1,2.
\end{eqnarray*}
Note that
$$
W_{1,2}(n)=\frac{1}{n\sqrt{\mathcal N_n}}\sum_{\lambda\in \Lambda_n}\lambda_1\lambda_2\,(|a_\lambda|^2-1),\quad \mbox{and}\quad \widehat W_{1,2}(n)=\frac{1}{n\sqrt{\mathcal N_n}}\sum_{\lambda\in \Lambda_n}\lambda_1\lambda_2\,(|\widehat a_\lambda|^2-1),
$$
since $\sum_{\lambda\in \Lambda_n} \lambda_1\lambda_2 = 0$, and also that $M_j$ is real-valued for $j=1,2$.

Now, let us express each summand appearing on the {}{right-hand side} of \paref{eq 4 chaos} in terms of $W{}{(n)}$, $W_1{}{(n)}$, $W_2{}{(n)}$, $W_{1,2}{}{(n)}$, $\widehat W{}{(n)}$, $\widehat W_1{}{(n)}$, $\widehat W_2{}{(n)}$, $\widehat W_{1,2}{}{(n)}$, $M{}{(n)}$, $M_1{}{(n)}$, $M_2{}{(n)}$, $M_{1,1}{}{(n)}$, $M_{2,2}{}{(n)}$ and/or $M_{1,2}{}{(n)}$. The proof of the following result will be given in Section \ref{appendix}. In what follows, the symbol $o_{\P}(1)$ indicates a sequence of random variables converging to zero in probability. In view of Remark \ref{r:postmain}-7, we will focus on sequences $\{n_j\}$ such that $\widehat{\mu}_{n_j}(4)$ converges to some number $\eta\in [-1,1]$.

\begin{lemma}\label{lemma infinito}
Let $\lbrace n_j\rbrace\subset S$ be such that $\mathcal N_{n_j}\to +\infty$ and $\widehat{\mu}_{n_j}(4)\to \eta\in[-1,1]$, then
\begin{enumerate}
\item[(i)] $\int_{\T} H_4(T_{n_j}(x))\,dx=\frac{3}{\mathcal{N}_{n_j}}\big(W{}{(n_j)}^2-2+o_\mathbb{P}(1)\big)$;
\item[(ii)] $\int_{\T} H_4(\widetilde{\partial}_k T_{n_j}(x))\,dx=\frac{3}{\mathcal{N}_{n_j}}\big(4W_k{}{(n_j)}^2-3-\eta+o_\mathbb{P}(1)\big)$, $k=1,2$;
\item[(iii)] $\int_{\T} H_2(T_{n_j}(x))\big(H_2(\widetilde{\partial}_1 T_{n_j})(x))+H_2(\widetilde{\partial}_2 T_{n_j})(x))\big)\,dx=\frac{2}{\mathcal{N}_{n_j}}\big(W{}{(n_j)}^2-2+o_\mathbb{P}(1)\big)$;
\item[(iv)] $\int_{\T} H_2(\widetilde{\partial}_1 T_{n_j}(x))H_2(\widetilde{\partial}_2 T_{n_j}(x))\,dx=\frac{1}{\mathcal{N}_{n_j}}\big(4W_{1}{}{(n_j)}W_2{}{(n_j)}+8W_{1,2}{}{(n_j)}^2-3+3\eta+o_\mathbb{P}(1)\big)$;
\item[(v)] $\int_{\T} H_2(T_{n_j})(x))H_2(\widehat T_{n_j})(x))\,dx=\frac{1}{\mathcal{N}_{n_j}}\big(W{}{(n_j)}\widehat{W}{}{(n_j)}+2M{}{(n_j)}^2-2+o_\mathbb{P}(1)\big)$;
\item[(vi)] $\int_{\T} H_2(T_{n_j})(x))\big(H_2(\widetilde{\partial}_1 \widehat T_{n_j}(x))+H_2(\widetilde{\partial}_2 \widehat T_{n_j})(x))\big)\,dx=\frac{2}{\mathcal{N}_{n_j}}\big(W{}{(n_j)}\widehat{W}{}{(n_j)}+M_1{}{(n_j)}^2+M_2{}{(n_j)}^2-1+o_\mathbb{P}(1)\big)$;
\item[(vii)] $\int_{\T} H_2(\widetilde{\partial}_\ell T_{n_j}(x))H_2(\widetilde{\partial}_k \widehat T_{n_j}(x))\,dx=\frac{1}{\mathcal{N}_{n_j}}\big(4W_\ell{}{(n_j)} \widehat{W}_j{}{(n_j)}+8M_{\ell,k}{}{(n_j)}^2-(3+\eta){\bf 1}_{\{\ell= k\}}-(1-\eta){\bf 1}_{\{\ell\neq k\}}+o_\mathbb{P}(1)\big)$, $\ell,k=1,2$;
\item[(viii)] $\int_{\T} \widetilde{\partial}_1 T_{n_j}(x)\widetilde{\partial}_2 T_{n_j}(x)\widetilde{\partial}_1 \widehat T_{n_j}(x)\widetilde{\partial}_2 \widehat T_{n_j}(x)\,dx=\frac{1}{\mathcal{N}_{n_j}}\big(4W_{1,2}{}{(n_j)}\widehat{W}_{1,2}{}{(n_j)}+4M_{1,1}{}{(n_j)}M_{2,2}{}{(n_j)}+4M_{1,2}{}{(n_j)}^2-1+\eta+o_\mathbb{P}(1)\big)$.
\end{enumerate}
\end{lemma}
We are now able to give an explicit expression for $I_n[4]$ in \paref{eq 4 chaos}.
\begin{lemma}\label{carino}
Let $\lbrace n_j\rbrace\subset S$ such that $\mathcal N_{n_j}\to +\infty$ and $\widehat{\mu}_{n_j}(4)\to \eta\in [-1,1]$, then
\begin{equation}\label{in4}
\begin{split}
I_{n_j}[4]=
\frac{n_j\pi}{8\,\mathcal N_{n_j}}\Big(&
\frac12W{}{(n_j)}^2+\frac12\widehat{W}{}{(n_j)}^2-3W{}{(n_j)}\widehat{W}{}{(n_j)}- W_1{}{(n_j)}^2 -  W_2{}{(n_j)}^2  -\widehat{W}_1{}{(n_j)}^2\\
& - \widehat{W}_2{}{(n_j)}^2
 +6W_1{}{(n_j)}\widehat{W}_2{}{(n_j)}+6\widehat{W}_1{}{(n_j)}W_2{}{(n_j)} -2W_{1,2}{}{(n_j)}^2
-2\widehat{W}_{1,2}{}{(n_j)}^2 \\
&- 12 W_{1,2}{}{(n_j)}\widehat{W}_{1,2}{}{(n_j)}  - 4M_1{}{(n_j)}^2 -4M_2{}{(n_j)}^2 +4M{}{(n_j)}^2-2M_{1,1}{}{(n_j)}^2\\
&-2M_{2,2}{}{(n_j)}^2
-12M_{1,1}{}{(n_j)}M_{2,2}{}{(n_j)}+8M_{1,2}{}{(n_j)}^2+4+o_\P(1)
\Big).
\end{split}
\end{equation}
\end{lemma}
\noindent\begin{proof}
From Lemma \ref{alpha piccoli} and \paref{eq 4 chaos}, we find that

\begin{eqnarray}
I_n[4]&=&\frac{n\pi}{64}\,\Big(8\int_{\T} H_4(T_n(x))\,dx -8\int_{\T} H_2(T_n(x))H_2(\widetilde{\partial}_1 T_n(x))\,dx\notag\\
&&-8\int_{\T} H_2(T_n(x))H_2(\widetilde{\partial}_2 T_n(x))\,dx- 2\int_{\T} H_2(\widetilde{\partial}_1 T_n(x))H_2(\widetilde{\partial}_2 T_n(x))\,dx\notag\\
&&-\int_{\T} H_4(\widetilde{\partial}_1 T_n(x))\,dx
-\int_{\T} H_4(\widetilde{\partial}_2 T_n(x))\,dx \notag\\
&&+8\int_{\T} H_4(\widehat T_n(x))\,dx -8\int_{\T} H_2(\widehat T_n(x))H_2(\widetilde{\partial}_1 \widehat T_n(x))\,dx\notag
\end{eqnarray}
\begin{eqnarray}
&&-8\int _{\T}H_2(\widehat T_n(x))H_2(\widetilde{\partial}_2 \widehat T_n(x))\,dx- 2\int_{\T} H_2(\widetilde{\partial}_1\widehat T_n(x))H_2(\widetilde{\partial}_2 \widehat T_n(x))\,dx\notag\\
&&-\int_{\T} H_4(\widetilde{\partial}_1\widehat  T_n(x))\,dx -\int_{\T} H_4(\widetilde{\partial}_2 \widehat T_n(x))\,dx \notag\\
&&+16\int_{\T} H_2(T_n(x))H_2(\widehat T_n(x))\,dx
-8\int_{\T} H_2(T_n(x))\big(H_2(\widetilde{\partial}_1 \widehat T_n(x))+H_2(\widetilde{\partial}_2 \widehat T_n(x))\big)\,dx\notag\\
&&
-8\int_{\T} H_2(\widehat T_n(x))\big(H_2(\widetilde{\partial}_1 T_n(x))+H_2(\widetilde{\partial}_2 T_n(x))\big)dx\notag\\
&&
-2\int_{\T} H_2(\widetilde{\partial}_1 T_n(x))H_2(\widetilde{\partial}_1 \widehat T_n(x))\,dx
-2\int_{\T} H_2(\widetilde{\partial}_2 T_n(x))H_2(\widetilde{\partial}_2 \widehat T_n(x))\,dx\notag\\
&&+10\int_{\T} H_2(\widetilde{\partial}_1 T_n(x))H_2(\widetilde{\partial}_2 \widehat T_n(x))\,dx
+10\int_{\T} H_2(\widetilde{\partial}_2 T_n(x))H_2(\widetilde{\partial}_1 \widehat T_n(x))\,dx\notag\\
&&-24\int_{\T} \widetilde{\partial}_1 T_n(x)\widetilde{\partial}_2 T_n(x)\widetilde{\partial}_1 \widehat T_n(x)\widetilde{\partial}_2 \widehat T_n(x)\,dx
\Big ).\label{eq infinita}
\end{eqnarray}
Using the previous identities (i)-(viii) in Lemma \ref{lemma infinito} in \paref{eq infinita}, and also using that
$W_1{}{(n_j)} + W_2{}{(n_j)} =W{}{(n_j)}$ and
$\widehat W_1{}{(n_j)} + \widehat W_2{}{(n_j)} = \widehat W{}{(n_j)}$, one concludes the proof.

\end{proof}

\subsection{Proofs of Proposition \ref{varianze chaos} and Proposition \ref{limite 4}}

Let us first study the asymptotic distribution of the centered random vector, defined for $n\in S$ as follows
\begin{equation*}
\begin{split}
{\bf W}(n):=(&W(n),W_1(n), W_2(n), W_{1,2}(n), \widehat W(n), \widehat W_1(n), \widehat W_2(n), \widehat W_{1,2}(n),\cr
&M(n),M_1(n),M_2(n),M_{1,1}(n),M_{2,2}(n),M_{1,2}(n))\in \mathbb R^{14}.
\end{split}
\end{equation*}

\begin{lemma}\label{CLT}
Let $\lbrace n_j\rbrace\subset S$ be such that $\mathcal N_{n_j}\to +\infty$ and $\widehat{\mu}_{n_j}(4)\to \eta\in [-1,1]$. Then, as $\mathcal N_{n_j}\to \infty$,
$$
{\bf W}(n_j) \stackrel{\rm law}{\Longrightarrow} {\bf G},
$$
where ${\bf G}=(G_1,\ldots,G_{14})$ denotes a Gaussian real centered vector with
covariance matrix given by
\begin{equation}\label{matrix M}
{\bf M}(\eta)=\left(\begin{matrix}
&{\bf A}{}{(\eta)} &0&0\\
&0&{\bf A}{}{(\eta)} &0\\
&0&0 &{\bf B}{}{(\eta)} \\
\end{matrix}\ \ \right ),
\end{equation}
where
$$
{\bf A}(\eta):=\left(\begin{matrix}
&2 &1 &1 &0 \\
&1 &\frac{3+\eta}{4} &\frac{1-\eta}{4} &0\\
&1  &\frac{1-\eta}{4} &\frac{3+\eta}{4} &0 \\
&0 &0 &0 &\frac{1-\eta}{4} \\
\end{matrix}\ \ \right ),
$$
and
$$
{\bf B}(\eta):=\left(\begin{matrix} &1 &0 &0 &\frac12 &\frac12 &0\\
&0 &\frac12 &0 &0 &0 &0 \\
&0 &0 &\frac12  &0 &0 &0\\
&\frac12 &0 &0 &\frac{3+\eta}{8} & \frac{1-\eta}{8} &0\\
&\frac12 &0 &0 &\frac{1-\eta}{8} &\frac{3+\eta}{8} &0\\
&0 &0 &0 &0 &0  &\frac{1-\eta}{8}
\end{matrix}\ \ \right ).
$$
\end{lemma}
\noindent\begin{proof}
First, for reasons related to independence it is easy to check that the covariance matrix of  ${\bf W}(n)$ takes the form
\begin{equation}\label{sigma grande}
\Sigma_{{}{n}} =\left(\begin{matrix}
&{\bf A}_n &0&0\\
&0&{\bf A}_n&0\\
&0&0 &{\bf B}_n\\
\end{matrix}\ \ \right ),
\end{equation}
where ${\bf A}_n$ and ${\bf B}_n$ denote the covariance matrices of $(W(n),W_1(n),W_2(n),W_{1,2}(n))$ and
$(M(n),M_1(n),M_2(n),M_{1,1}(n),M_{2,2}(n),M_{1,2}(n))$ respectively. Let us first compute ${\bf A}_n$.
Since $\E[(|a_\lambda|^2-1)(|a_{\lambda'}|^2-1)]=1$ if $\lambda=\pm \lambda'$ and is zero otherwise, one has
$$
\E(W{}{(n)}^2) = \frac1{\mathcal{N}_n}\sum_{\lambda,\lambda'\in\Lambda_n}\E[(|a_\lambda|^2-1)(|a_{\lambda'}|^2-1)]
=2.
$$
Similarly,
$$
\E(W{}{(n)}W_j{}{(n)}) = \frac1{n\mathcal{N}_n}\sum_{\lambda,\lambda'\in\Lambda_n}\lambda_j^2\E[(|a_\lambda|^2-1)(|a_{\lambda'}|^2-1)]
=\frac2{n\mathcal{N}_n}\sum_{\lambda\in\Lambda_n}\lambda_j^2=1,
$$
whereas
$$
\E(W{}{(n)}W_{1,2}{}{(n)}) = \frac1{n\mathcal{N}_n}\sum_{\lambda,\lambda'\in\Lambda_n}\lambda_1\lambda_2\E[(|a_\lambda|^2-1)(|a_{\lambda'}|^2-1)]
=\frac2{n\mathcal{N}_n}\sum_{\lambda\in\Lambda_n}\lambda_1\lambda_2=0.
$$
We also have
$$
\E(W_j{}{(n)}^2) = \frac1{n^2\mathcal{N}_n}\sum_{\lambda,\lambda'\in\Lambda_n}\lambda_j^2{\lambda'_j}^2\E[(|a_\lambda|^2-1)(|a_{\lambda'}|^2-1)]
=\frac2{n^2\mathcal{N}_n}\sum_{\lambda\in\Lambda_n}\lambda_j^4.$$
To express $\E(W_j{}{(n)}^2)$ in a more suitable way, let us rely on $\widehat{\mu}_n(4)$:
\begin{eqnarray*}
\widehat{\mu}_n(4)&=&\int_{\mathcal{S}^1}z^4d\mu_n(z)=\frac{1}{n^2\mathcal{N}_n}\sum_{\lambda\in \Lambda_n}
(\lambda_1+i\lambda_2)^4
=\frac{1}{n^2\mathcal{N}_n}\sum_{\lambda\in \Lambda_n}
(\lambda_1^4-6\lambda_1^2\lambda_2^2+\lambda_2^4)\\
&=&\frac{1}{n^2\mathcal{N}_n}\sum_{\lambda\in \Lambda_n}(\lambda_1^2+\lambda_2^2)^2 -
\frac{8}{n^2\mathcal{N}_n}\sum_{\lambda\in \Lambda_n}\lambda_1^2\lambda_2^2
=1 -
\frac{8}{n^2\mathcal{N}_n}\sum_{\lambda\in \Lambda_n}\lambda_1^2\lambda_2^2.
\end{eqnarray*}
As a result,
$$
\sum_{\lambda\in \Lambda_n}\lambda_1^2\lambda_2^2 = \frac{n^2\mathcal{N}_n}{8}(1-\widehat{\mu}_n(4)),
$$
leading to
$$
\E(W_j{}{(n)}^2)  = \frac1{n^2\mathcal{N}_n}\sum_{\lambda\in\Lambda_n}(\lambda_1^4+\lambda_2^4)
=\widehat{\mu}_n(4)+\frac{6}{n^2\mathcal{N}_n}\sum_{\lambda\in \Lambda_n}\lambda_1^2\lambda_2^2
=\frac14(3+\widehat{\mu}_n(4)).
$$
Similarly,
\begin{eqnarray*}
\E(W_{1,2}{}{(n)}^2)  &= &\frac1{n^2\mathcal{N}_n}\sum_{\lambda,\lambda'\in\Lambda_n}\lambda_1\lambda_2\lambda'_1\lambda'_2
\E[|a_\lambda|^2|a_{\lambda'}|^2]=\frac2{n^2\mathcal{N}_n}\sum_{\lambda\in\Lambda_n}\lambda_1^2\lambda_2^2=\frac14(1-\widehat{\mu}_n(4)),
\end{eqnarray*}
as well as
\begin{eqnarray*}
\E(W_1{}{(n)}W_2{}{(n)})  &= &\frac1{n^2\mathcal{N}_n}\sum_{\lambda,\lambda'\in\Lambda_n}\lambda_1^2{\lambda'_2}^2
\E[(|a_\lambda|^2-1)(|a_{\lambda'}|^2-1)]\\
&=&\frac2{n^2\mathcal{N}_n}\sum_{\lambda\in\Lambda_n}\lambda_1^2\lambda_2^2=\frac14(1-\widehat{\mu}_n(4)),
\end{eqnarray*}
and
\begin{eqnarray*}
\E(W_j{}{(n)}W_{1,2}{}{(n)})  &= &\frac1{n^2\mathcal{N}_n}\sum_{\lambda,\lambda'\in\Lambda_n}\lambda_j^2\lambda'_1\lambda'_2
\E[(|a_\lambda|^2-1)(|a_{\lambda'}|^2-1)]\\
&=&\frac2{n^2\mathcal{N}_n}\sum_{\lambda\in\Lambda_n}\lambda_j^2\lambda_1\lambda_2=0.
\end{eqnarray*}
Taking all these facts into consideration, we deduce that
$${\bf A}_n=\left(\begin{matrix}
&2 &1 &1 &0 \\
&1 &\frac{3+\widehat \mu_n(4)}{4} &\frac{1-\widehat \mu_n(4)}{4} &0\\
&1  &\frac{1-\widehat \mu_n(4)}{4} &\frac{3+\widehat \mu_n(4)}{4} &0 \\
&0 &0 &0 &\frac{1-\widehat \mu_n(4)}{4} \\
\end{matrix}\ \ \right ).$$
Now, let us turn to the expression of ${\bf B}_n$.
Using that $\E[a_\lambda  a_{\lambda'}]=1$ if $\lambda'=-\lambda$ and is zero otherwise, we obtain
$$
\E(M{}{(n)}^2)= \frac1{\mathcal{N}_n}\sum_{\lambda,\lambda'\in\Lambda_n}\E[a_\lambda  a_{\lambda'}]\E[ \overline{\widehat{a}_\lambda} \overline{\widehat{a}_{\lambda'}}]
=1.
$$
Similarly,
$$
\E(M_j{}{(n)}^2)= -\frac1{n\mathcal{N}_n}\sum_{\lambda,\lambda'\in\Lambda_n}\lambda_j\lambda'_j\E[a_\lambda  a_{\lambda'}]\E[ \overline{\widehat{a}_\lambda} \overline{\widehat{a}_{\lambda'}}]
=\frac1{n\mathcal{N}_n}\sum_{\lambda\in\Lambda_n}\lambda_j^2=\frac12,
$$
as well as
\begin{eqnarray*}
\E(M_{j,j}{}{(n)}^2)&=& \frac1{n^2\mathcal{N}_n}\sum_{\lambda,\lambda'\in\Lambda_n}\lambda_j^2{\lambda'_j}^2\E[a_\lambda  a_{\lambda'}]\E[ \overline{\widehat{a}_\lambda} \overline{\widehat{a}_{\lambda'}}]\\
&=&\frac1{n^2\mathcal{N}_n}\sum_{\lambda\in\Lambda_n}\lambda_j^4=\frac12\E(W_j^2)=\frac18(3+\widehat{\mu}_n(4)),
\end{eqnarray*}
and
\begin{eqnarray*}
\E(M_{1,2}{}{(n)}^2)&=& \frac1{n^2\mathcal{N}_n}\sum_{\lambda,\lambda'\in\Lambda_n}\lambda_1\lambda_2{\lambda'_1}{\lambda'_2}\E[a_\lambda  a_{\lambda'}]\E[ \overline{\widehat{a}_\lambda} \overline{\widehat{a}_{\lambda'}}]\\
&=&\frac1{n^2\mathcal{N}_n}\sum_{\lambda\in\Lambda_n}\lambda_1^2\lambda_2^2=\frac12\E(W_{12}^2)=\frac18(1-\widehat{\mu}_n(4)).
\end{eqnarray*}
Besides, it is immediate to check that, for any $l,j$,
$$
\E(M{}{(n)}M_j{}{(n)})=\E(M{}{(n)}M_{12}{}{(n)})=\E(M_j{}{(n)}M_{l,j}{}{(n)})=\E(M_{j,j}{}{(n)}M_{1,2}{}{(n)})=0.
$$
Finally,
$$
\E(M{}{(n)}M_{j,j}{}{(n)})= \frac1{n\mathcal{N}_n}\sum_{\lambda,\lambda'\in\Lambda_n}\lambda_j^2\E[a_\lambda  a_{\lambda'}]\E[ \overline{\widehat{a}_\lambda} \overline{\widehat{a}_{\lambda'}}]
=\frac1{n^2\mathcal{N}_n}\sum_{\lambda\in\Lambda_n}\lambda_j^2=\frac12,
$$
whereas
\begin{eqnarray*}
\E(M_{1,1}{}{(n)}M_{2,2}{}{(n)})&=& \frac1{n^2\mathcal{N}_n}\sum_{\lambda,\lambda'\in\Lambda_n}\lambda_1^2{\lambda'_2}^2\,
\E[a_\lambda  a_{\lambda'}]\E[ \overline{\widehat{a}_\lambda} \overline{\widehat{a}_{\lambda'}}]\\
&=&\frac1{n^2\mathcal{N}_n}
\sum_{\lambda\in\Lambda_n}\lambda_1^2\lambda_2^2
=\frac12\E(W_{12}{}{(n)}^2)=\frac18(1-\widehat{\mu}_n(4)).
\end{eqnarray*}
Putting everything together, we arrive at the following expression for ${\bf B}_n$
$$
{\bf B}_n=\left(\begin{matrix} &1 &0 &0 &\frac12 &\frac12 &0\\
&0 &\frac12 &0 &0 &0 &0 \\
&0 &0 &\frac12  &0 &0 &0\\
&\frac12 &0 &0 &\frac{3+\widehat \mu_n(4)}{8} & \frac{1-\widehat \mu_n(4)}{8} &0\\
&\frac12 &0 &0 &\frac{1-\widehat \mu_n(4)}{8} &\frac{3+\widehat \mu_n(4)}{8} &0\\
&0 &0 &0 &0 &0  &\frac{1-\widehat \mu_n(4)}{8}
\end{matrix}\ \ \right ).
$$

{
Now, let us prove that each component of ${\bf W}_{n_j}$ is asymptotically Gaussian as $\mathcal{N}_{n_j}\to +\infty$.
Since all components of ${\bf W}_{n_j}$ belong to the same {{} Wiener} chaos (the second one) and have a converging variance  (see indeed the diagonal part of  ${\bf B}_n$ just above), according to the Fourth Moment Theorem (see, e.g., \cite[Theorem 5.2.7]{NP}) it suffices to show that  the fourth cumulant of each component of ${\bf W}_{n_j}$ goes to zero as $\mathcal{N}_{n_j}\to +\infty$. Since we are dealing with sum of independent random variables,
checking such a property is straightforward. For sake of illustration, let us only consider the case of $W_2(n_j)$ which is representative of the difficulty. {{} We recall that, given a real-valued random variable $Z$ with mean zero, the fourth cumulant of $Z$ is defined by $\kappa_4(Z) :=\E[Z^4] - 3\E[Z^2]$. }  Since the $a_\lambda$ are independent except for the relation $\overline{a_\lambda}=a_{-\lambda}$, we can write, setting $\Lambda_n^+=\{\lambda\in\Lambda_n:\,\lambda_2>0\}$,
\begin{eqnarray*}
\kappa_4(W_2(n))&=&
\kappa_4\bigg(\frac{2}{n\sqrt{\mathcal{N}_{n}}}
\sum_{\lambda\in\Lambda_n^+}\lambda_2^2(|a_\lambda|^2-1)\bigg)
=\frac{16\,\kappa_4(|N_\mathbb{C}(0,1)|^2)}{n^4\mathcal{N}_{n}^2}
\sum_{\lambda\in\Lambda_n^+}\lambda_2^8\\
&\leq&
\frac{8\,\kappa_4(|N_\mathbb{C}(0,1)|^2)}{\mathcal{N}_{n}};
\end{eqnarray*}
to obtain the last inequality, we have used that $\lambda_2^2\leq \lambda_1^2+\lambda_2^2= n$.
As a result, $\kappa_4(W_2(n_j))\to 0$ as $\mathcal{N}_{n_j}\to +\infty$
and it follows from the Fourth Moment Theorem that $W_2(n_j)$ is asymptotically Gaussian. It is not difficult to apply a similar strategy in order to prove that, actually, each component of ${\bf W}_{n_j}$ is asymptotically Gaussian as well; details are left to the reader.
}

{
Finally, we make use of
\cite[Theorem 6.2.3]{NP} to conclude the proof of Lemma \ref{CLT}.
}
Indeed, (i) all components of ${\bf W}_n$ belong to the same Wiener chaos (the second one), (ii) each component of ${\bf W}_{n_j}$ is asymptotically Gaussian (as $\mathcal{N}_{n_j}\to +\infty$), and finally (iii) $\Sigma_{k,l}(n_j)\to {\bf M}_{k,l}(\eta)$ for each pair of indices $(k,l)$.
\end{proof}

\bigskip


\noindent\begin{proof}[Proofs of Proposition \ref{varianze chaos} and Proposition \ref{limite 4}]
For each subsequence $\lbrace n'_j\rbrace\subset \lbrace n_j\rbrace$, there exists a subsubsequence $\lbrace n''_j\rbrace\subset \lbrace n'_j\rbrace$ such that it holds either (i) $\widehat{\mu}_{n''_j}(4)\to \eta$ or (ii) $\widehat{\mu}_{n''_j}(4)\to -\eta$.

Combining Lemma \ref{carino} with Lemma \ref{CLT}, we have, as $j\to +\infty$,
\begin{equation}\label{limit}
\begin{split}
\frac{8\mathcal{N}_{n''_j}}{n''_j\pi}I_{n''_j}[4]\Rightarrow &\frac12G_1^2+\frac12G_5^2-3G_1G_5- G_2^2 -  G_3^2 -G_6^2 - G_7^2  +6G_2G_7+6G_6G_3-2G_4^2-2G_8^2\notag\\
&-12G_4G_8-4G_{10}^2-4G_{11}^2+4G_9^2-2G_{12}^2-2G_{13}^2+8G_{14}^2-12G_{12}G_{13},
\end{split}
\end{equation}
where $(G_1,\ldots,G_{14})$ denotes a Gaussian centered vector with covariance matrix \paref{matrix M}.

Since $\left \lbrace \frac{8N_{n''_j}}{n''_j\pi}I_{n''_j}[4] \right \rbrace$ is a sequence of random variables belonging to a fixed Wiener chaos and converging in distribution, by
standard arguments based on uniform integrability, we also have
\begin{equation*}
\begin{split}
\Var&\left(\frac{8N_{n''_j}}{n''_j\pi}I_{n''_j}[4]\right)\to \Var\Big(\frac12G_1^2+\frac12G_5^2-3G_1G_5- G_2^2 -  G_3^2 -G_6^2 - G_7^2  +6G_2G_7+6G_6G_3\notag\\
&-2G_4^2-2G_8^2-12G_4G_8-4G_{10}^2-4G_{11}^2+4G_9^2-2G_{12}^2-2G_{13}^2+8G_{14}^2-12G_{12}G_{13}\Big);
\end{split}
\end{equation*}
the proof of Proposition \ref{varianze chaos} is then concluded, once computing
\begin{equation*}
\begin{split}
\Var\Big(&\frac12G_1^2+\frac12G_5^2-3G_1G_5- G_2^2 -  G_3^2 -G_6^2 - G_7^2  +6G_2G_7+6G_6G_3-2G_4^2-2G_8^2\\
&-12G_4G_8-4G_{10}^2-4G_{11}^2+4G_9^2-2G_{12}^2-2G_{13}^2+8G_{14}^2-12G_{12}G_{13}\Big)= 8(3\eta^2 +5),
\end{split}
\end{equation*}
and noting that the latter variance is the same in both cases (i) and (ii).

Let us now prove Proposition \ref{limite 4}.
Let $(Z_1,\ldots,Z_{11})\sim N_{11}(0,I)$ be a standard Gaussian vector of $\R^{11}$. Then one can check that the vector
$$
\left(
\begin{matrix}
\sqrt{2}\,Z_5\\
\frac{1}{\sqrt{2}}Z_5+\frac12\sqrt{\eta+1}\,Z_3\\
\frac{1}{\sqrt{2}}Z_5-\frac12\sqrt{\eta+1}\,Z_3\\
\frac12\sqrt{1-\eta}\,Z_8\\
\sqrt{2}\,Z_6\\
\frac{1}{\sqrt{2}}Z_6+\frac12\sqrt{\eta+1}\,Z_4\\
\frac{1}{\sqrt{2}}Z_6-\frac12\sqrt{\eta+1}\,Z_4\\
\frac12\sqrt{1-\eta}\,Z_9\\
Z_2\\
\frac{1}{\sqrt{2}}Z_{10}\\
\frac{1}{\sqrt{2}}Z_{11}\\
\frac12Z_2+\sqrt{\frac18(\eta+1)}\,Z_1\\
\frac12Z_2-\sqrt{\frac18(\eta+1)}\,Z_1\\
\sqrt{\frac18(1-\eta)}\,Z_7
\end{matrix}
\right)
$$
admits ${\bf M}{}{(\eta)}$ for covariance matrix as well.
Expressing (\ref{limit}) in terms of $(U_{1},\ldots,U_{11})$ leads to the fact that (\ref{limit}) has the same law
as $$
\frac{1+\eta}{2}A
+\frac{1-\eta}{2}B
-2(C-2),$$
 with $A,B,C$ independent and
$A\overset{\rm law}{=}B\overset{\rm law}{=}2Z_1^2-Z_2^2-Z_3^2-6Z_2Z_3$ and
$C\overset{\rm law}{=}Z_1^2+Z_2^2$.

Finally, noting that the law of the random variable $\frac{1+\eta}{2}A
+\frac{1-\eta}{2}B
-2(C-2)$ is the same for case (i) and case (ii) and using that $(Z_1,Z_2,Z_3)\overset{\rm law}{=}(Z_1,\frac1{\sqrt{2}}(Z_2-Z_3),\frac1{\sqrt{2}}(Z_2+Z_3))$,
we get the desired conclusion.

\end{proof}

\section{The variance of higher order chaoses}\label{proof varianza sec}

In this section we shall prove Proposition \ref{asymptotic}.
Let us decompose the torus $\mathbb T$ as a disjoint union of squares $Q_k$ of side length $1/M$ (where $M\approx \sqrt{E_n}$ is a large integer\footnote{{$M= {{} \lceil  d \sqrt{E_n} \rceil}$,  $d\in \R_{>0}$}}), obtained by translating along directions ${k}/M$, ${k}\in \mathbb Z^2$, the square $Q_0 := [0,1/M)\times [0,1/M)$ containing the origin. By construction, the south-west corner of each square is therefore situated at the point ${k}/M$.

\subsection{Singular points and cubes}
Let us first give some definitions, inspired by \cite[\S 6.1]{ORW} and \cite[\S 4.3]{RW2}. Let us denote by {\blue $0<\eps_1<\frac{1}{10^{10} }$} a very small number\footnote{{Let us now choose $d$ such that $d\ge \frac{16\pi^2}{\eps_1}$}.}  that will be fixed until the end. From now on, we shall use the simpler notation $r_j:=\partial_j r_n$, and $r_{ij} := \partial_{ij} r_n$ for $i,j=1,2$.

\begin{definition}[Singular pairs of points and cubes]\label{sing point}

\

\noindent i) A pair of points $(x,y)\in \mathbb T\times \mathbb T$ is called singular if either $|r(x-y)|>\eps_1$ or $|r_1(x-y)|>\eps_1\, \sqrt{n}$ or $|r_2(x-y)|>\eps_1\, \sqrt{n}$ or $|r_{12}(x-y)|>\eps_1\, n$ or $|r_{11}(x-y)|>\eps_1\, n$ or $|r_{22}(x-y)|>\eps_1\, n$.

\noindent ii) A pair of cubes $(Q,Q')$ is called singular if the product $Q\times Q'$ contains a singular pair of points.
\end{definition}
For instance, $(0,0)$ is a singular pair of points and hence $(Q_0,Q_0)$ is a singular pair of cubes. In what follows we will often drop the dependence of $k$ from $Q_k$.

\begin{lemma}\label{lemma sing}
Let $(Q, Q')$ be a singular pair of cubes, then $|r(z-w)|>\frac12\eps_1$ or $|r_1(z-w))|>\frac12\,\eps_1\, \sqrt{n}$ or $|r_2(z-w))|>\frac12\eps_1\, \sqrt{n}$ or $|r_{12}(z-w))|>\frac12\eps_1\, n$ or $|r_{11}(z-w))|>\frac12\eps_1\, n$ or $|r_{22}(z-w))|>\frac12\eps_1\, n$, for every $(z,w)\in Q\times Q'$.
\end{lemma}
\noindent
\begin{proof} First note that the function $\mathbb T\ni s\mapsto r(s/\sqrt{n})$ and  its derivatives up to the order two are Lipschitz with a universal Lipschitz constant  $c=8\pi^3$ (in particular, independent of $n$).
Let us denote by $(x,y)$ the singular pair of points contained in $Q\times Q'$ and suppose that $r(x-y)>\eps_1$.  For every $(z,w)\in Q\times Q'$,
\begin{equation*}
\begin{split}
|r(z-w) - r(x-y)| &= \left |r\left( \frac{(z-w)\cdot \sqrt n}{\sqrt n}  \right)-r\left( \frac{(x-y)\cdot \sqrt n}{\sqrt n}  \right) \right |\cr
& \le c\sqrt n |(z-x) - (w-y)|\le 2c \sqrt{n}  \frac{1}{M}.
\end{split}
\end{equation*}
{Since $d\ge \frac{16\pi^2}{\eps_1}$ in  $M= {{} \lceil  d \sqrt{E_n} \rceil}$}, then
\begin{equation*}
\begin{split}
r(z-w) \ge r(x-y) - \eps_1 /2 > \eps_1 /2.
\end{split}
\end{equation*}
The case $r(x-y) < -\eps_1$ in indeed analogous.
The rest of the proof for derivatives follows the same argument.

\end{proof}

Let us now denote by $B_Q$ the union of all squares $Q'$ such that $(Q,Q')$ is a singular pair.  The number of such cubes $Q'$ is $M^2 {\rm Leb}(B_Q)$, the area of each cube being $1/M^2$.

\begin{lemma}\label{bello}
It holds that
$\text{Leb}(B_Q) \ll  \int_{\mathbb T} r(x)^6\,dx$.
\end{lemma}
\noindent
\begin{proof}
Let us first note that
$$
B_Q \subset B_{Q}^0 \cup B_Q^1 \cup B_Q^2 \cup B_Q^{12} \cup B_Q^{11} \cup B_Q^{22},
$$
where $B_{Q}^0$ is the union of all cubes $Q'$ such that there exists $(x,y)\in Q\times Q'$ enjoying $|r(x-y)|>\eps_1$ and for $i,j=1,2$, $B_Q^{i}$ is the union of all cubes $Q'$ such that there exists $(x,y)\in Q\times Q'$ enjoying $|r_i(x-y)|>\eps_1\, \sqrt n$ and finally $B_Q^{ij}$ is the union of all cubes $Q'$ such that there exists $(x,y)\in Q\times Q'$ enjoying $|r_{ij}(x-y)|>\eps_1\,  n$. We can hence write
 \begin{equation*}
\text{Leb}(B_Q) \le \text{Leb}(B_{Q}^0)+\text{Leb}(B_Q^1)+ \text{Leb}(B_Q^2)+ \text{Leb}(B_Q^{12})+ \text{Leb}(B_Q^{11})+ \text{Leb}(B_Q^{22}).\end{equation*}
Let us now fix $z\in Q$;  then Lemma \ref{lemma sing} yields
 \begin{equation*}
\text{Leb}(B_Q^0) = \int_{B_Q^0} \frac{|r(z-w)|^6}{|r(z-w)|^6}\,dw\le \eps_1^{-6} \int_{B_Q^0} |r(z-w)|^6\,dw \le \eps_1^{-6}\int_{\mathbb T} |r(x)|^6\,dx.
\end{equation*}
Moreover, for $i=1,2$,
\begin{equation*}
\text{Leb}(B_Q^i) = \int_{B_Q^i} \frac{|\widetilde r_i(z-w)|^6}{|\widetilde r_i(z-w)|^6}\,dw\le \eps_1^{-6} \int_{B_Q^i} |\widetilde r_i(z-w)|^6\,dw \le \eps_1^{-6}\int_{\mathbb T} |\widetilde r_i(x)|^6\,dx,
\end{equation*}
where $\widetilde r_i:=r_i/\sqrt{E_n}$ are the normalized derivatives.
Since
\begin{equation*}
\begin{split}
\int_{\mathbb T} \widetilde r_i(x)^6\,dx&= \frac{1}{\mathcal N_n^6}\sum_{\lambda, \lambda',\dots, \lambda^{v}} \frac{\lambda_i}{\sqrt n} \frac{\lambda_i'}{\sqrt n}\cdots
 \frac{\lambda_i^v}{\sqrt n}\int_{\mathbb T} \e^{i2\pi\langle \lambda -\lambda'+\dots -\lambda^v,x \rangle}\,dx \cr
&= \frac{1}{\mathcal N_n^6}\sum_{\lambda-\lambda'+\dots-\lambda^{v}=0} \frac{\lambda_i}{\sqrt n} \frac{\lambda_i'}{\sqrt n}\cdots
 \frac{\lambda_i^v}{\sqrt n}\cr
&\le \frac{|S_6(n)|}{\mathcal N_n^6}= \int_{\mathbb T}  r(x)^6\,dx,
\end{split}
\end{equation*}
we have
$$
\text{Leb}(B_Q^i) \ll \int_{\mathbb T}  r(x)^6\,dx.
$$
An analogous argument applied to $B_Q^{ij}$ for $i,j=1,2$ allows to conclude the proof.

\end{proof}

The number of cubes $Q'$ such that the pair $(Q,Q')$ is singular is hence {negligible {{} with respect to}} $E_n\, R_n(6)$.

\subsection{Variance and cubes}

We write the total number $I_n$ of nodal intersections  as the sum of the number $I_{n_{ |_Q}}$ of nodal intersections restricted to each square $Q$, i.e.
\begin{equation*}
I_n = \sum_Q I_{n_{ |_Q}}.
\end{equation*}
We have
\begin{equation*}
\text{proj}\left(I_n|C_{\ge 6}\right )= \sum_Q \text{proj}\left (I_{n_{ |_Q}}|C_{\ge 6} \right),
\end{equation*}
so that
\begin{equation*}
\Var\left(\text{proj}(I_n|C_{\ge 6}) \right )= \sum_{Q,Q'} \Cov\left (\text{proj}\left (I_{n_{ |_Q}}|C_{\ge 6}\right ), \text{proj}\left (I_{n_{ |_{Q'}}}|C_{\ge 6}\right )\right ).
\end{equation*}
We are going to separately investigate the contribution of the singular pairs and the non-singular pairs of cubes:
\begin{equation*}
\begin{split}
\Var(\text{proj}(I_n|C_{\ge 6})) =& \sum_{(Q,Q')\text{ sing.}} \Cov(\text{proj}(I_{n_{ |_Q}}|C_{\ge 6}), \text{proj}(I_{n_{ |_{Q'}}}|C_{\ge 6})) \cr
&+ \sum_{(Q,Q')\text{ non sing.}} \Cov(\text{proj}(I_{n_{ |_Q}}|C_{\ge 6}), \text{proj}(I_{n_{ |_{Q'}}}|C_{\ge 6})).
\end{split}
\end{equation*}

\subsubsection{The contribution of singular pairs of cubes}

\begin{proof}[Proof of Lemma \ref{varianza sing}]
By Cauchy-Schwarz inequality and the stationarity of ${\bf T}_n$, recalling moreover
 Lemma \ref{bello}, we have
\begin{equation*}
\begin{split}
&\left | \sum_{(Q,Q')\text{ sing.}} \Cov\left (\text{proj}\left (I_{n_{ |_Q}}|C_{\ge 6}\right ), \text{proj}\left (I_{n_{ |_{Q'}}}|C_{\ge 6}\right )\right ) \right | \cr
&\le  \sum_{(Q,Q')\text{ sing.}} \left |\Cov\left (\text{proj}\left (I_{n_{ |_Q}}|C_{\ge 6}\right ), \text{proj}\left (I_{n_{ |_{Q'}}}|C_{\ge 6}\right )\right ) \right |\cr
& \le  \sum_{(Q,Q')\text{ sing.}} \sqrt{\Var\left (\text{proj}\left (I_{n_{ |_Q}}|C_{\ge 6}\right )\right ) \Var\left (\text{proj}\left (I_{n_{ |_{Q'}}}|C_{\ge 6}\right )\right )}\cr &\ll E_n^2 R_n(6)   \Var\left (\text{proj}\left (I_{n_{ |_{Q_0}}}|C_{\ge 6}\right )\right ),
\end{split}
\end{equation*}
where, from now on, $Q_0$ denotes the square containing the origin.
Now,
\begin{equation*}
\begin{split}
\Var\left (\text{proj}\left (I_{n_{ |_{Q_0}}}|C_{\ge 6}\right )\right )\le \E\left [I_{n_{ |_{Q_0}}}^2\right ]= \underbrace{\E\left [I_{n_{ |_{Q_0}}}^2\right ] - \E\left [I_{n_{ |_{Q_0}}}\right ]}_{=:A}+\E\left [I_{n_{ |_{Q_0}}}\right ].
\end{split}
\end{equation*}
It is immediate to check that
\begin{equation*}
\E\left [  I_{n_{ |_{Q_0}}}        \right ] = \frac{2\pi n}{M^2},
\end{equation*}
in particular $\E\left [  I_{n_{ |_{Q_0}}}        \right ] =O(1)$.
Note that $A$ is the $2$-th factorial moment of $I_{n_{ |_{Q_0}}}$:
\begin{equation*}
A = \E\left[I_{n_{ |_{Q_0}}}\left(I_{n_{ |_{Q_0}}}-1    \right)    \right].
\end{equation*}
Applying \cite[Theorem 6.3]{AW}, we can write
\begin{equation}\label{media}
\begin{split}
A=\E\left [  I_{n_{ |_{Q_0}}}  \left(  I_{n_{ |_{Q_0}}}  -1   \right)     \right ] = \int_{Q_0}\int_{Q_0} K_2(x,y)\,dx dy,
\end{split}
\end{equation}
where
\begin{equation*}\label{corr1}
K_2(x,y) := p_{({\bf T}_n(x),{\bf T}_n(y))}(0,0)\, \E\left [\left | J_{{\bf T}_n}(x)\right |\cdot \left | J_{{\bf T}_n}(y)\right |\Big | {\bf T}_n(x) = {\bf T}_n(y) = 0  \right ]
\end{equation*}
is the so-called $2$-point correlation function. {Indeed, Proposition \ref{gio} ensures that for $(x,y)\in Q_0\times Q_0$, the vector $({\bf T}_n(x), {\bf T}_n(y))$ is non-degenerate {except} on the diagonal $x=y$}.

Note that, by stationarity of the model, we can write \paref{media} as
\begin{equation*}
\begin{split}
\E\left [  I_{n_{ |_{Q_0}}}  \left(  I_{n_{ |_{Q_0}}}  -1   \right)     \right ] = \text{Leb}(Q_0)\int_{\widetilde Q_0} K_2(x)\,dx,
\end{split}
\end{equation*}
where $K_2(x):=K_2(x,0)$ and $\widetilde Q_0$ is $2 Q_0$.

Let us first check that the function $x\mapsto K_2(x)$  is integrable around the origin.
Note that, by Cauchy-Schwarz inequality,
\begin{equation*}
\begin{split}
K_2(x) &= \frac{1}{1-r^2(x)}\, \E\left [\left | J_{{\bf T}_n}(x)\right |\cdot \left | J_{{\bf T}_n}(0)\right |\Big | {\bf T}_n(x) = {\bf T}_n(0) = 0  \right ]\cr
&\le  \frac{1}{1-r^2(x)}\, \E\left [\left | J_{{\bf T}_n}(0)\right |^2\Big | {\bf T}_n(x) = {\bf T}_n(0) = 0  \right ].
\end{split}
\end{equation*}
Hypercontractivity on Wiener chaoses \cite{NP} ensures that there exists $c>0$ such that
\begin{equation*}
K_2(x) \le  c \frac{1}{1-r^2(x)}\, \left( \E\left [\left | J_{{\bf T}_n}(0)\right |\Big | {\bf T}_n(x) = {\bf T}_n(0) = 0  \right ]\right)^2.
\end{equation*}
Now, thanks to \cite[(1.3)]{KZ}
\begin{equation*}
K_2(x) \le  c \frac{1}{1-r^2(x)}\, \left( \sqrt{|\Omega_n(x)|}\right)^2
= c \frac{|\Omega_n(x)|}{1-r^2(x)},
\end{equation*}
where $\Omega_n(x)$  denotes the covariance matrix of $\nabla T_n(0)$ conditioned to ${T}_n(x) = { T}_n(0) = 0$ (see \paref{matrix omega} for a precise expression).

The Taylor expansion in Lemma \ref{lemma taylor} gives that, as $\|x\|\to 0$,
\begin{equation*}
\frac{|\Omega_n(x)|}{1-r^2(x)} = c E_n^2 + E_n^3O(\|x\|^2),
\end{equation*}
for some other constant $c>0$, where the constants involving in the `O' notation do not depend on $n$, so that
\begin{equation*}
\begin{split}
\E\left [  I_{n_{ |_{Q_0}}}  \left(  I_{n_{ |_{Q_0}}}  -1   \right)     \right ] = \text{Leb}(Q_0)\int_{\widetilde Q_0} K_2(x)\,dx\ll \frac{E_n^2}{M^4},
\end{split}
\end{equation*}
which is the result we looked for.

\end{proof}

\subsubsection{The contribution of non-singular pairs of cubes}

\begin{proof}[Proof of Lemma \ref{varianza nonsing}]
For any square $Q$, we can write
\begin{eqnarray*}
&& \text{proj}\left(I_{n_{|_{Q}}}|C_{\ge 6}\right) \\
&& =\frac{E_n}{2} \sum_{q\ge 3}\sum_{i_1+i_2+i_3 +j_1+j_2+j_3=2q} \frac{\beta_{i_1}\beta_{j_1}\alpha_{i_2,i_3,j_2,j_3}}{i_1! i_2! i_3! j_1! j_2! j_3!}\times\\
&& \quad\times \int_{Q} H_{i_1}(T_n(x)) H_{i_2}(\partial_1 \widetilde T_n(x)) H_{i_3}(\partial_2 \widetilde T_n(x))H_{j_1}(\widehat T_n(x)) H_{j_2}(\partial_1 \widetilde {\widehat T}_n(x)) H_{j_3}(\partial_2 \widetilde {\widehat T}_n(x))\,dx,
\end{eqnarray*}
for even $i_1,j_1$ and $i_2,i_3,j_2,j_3$ with the same parity.
Recall that $\beta_l=0$ for odd $l$, and that $\beta_{2l}^2/{(2l)!}\approx 1/\sqrt{l}$, as $l\to \infty$. {\blue 
We have
\begin{equation}\label{lungo1}
\begin{split}
& \left| \sum_{(Q,Q')\text{ non sing.}}  \Cov\left (\text{proj}\left(I_{n_{|_{Q}}}|C_{\ge 6}\right),\text{proj}\left(I_{n_{|_{Q'}}}|C_{\ge 6}\right) \right) \right|\cr
&\le E_n^2 \sum_{q\ge 3} \sum_{i_1+i_2+i_3+j_1+j_2+j_3=2q}\sum_{a_1+a_2+a_3+b_1+b_2+b_3=2q} \left |\frac{\beta_{i_1}\beta_{j_1}\alpha_{i_2,i_3,j_2,j_3}}{i_1! i_2! i_3! j_1! j_2! j_3!}\right |\cdot \left |\frac{\beta_{a_1}\beta_{b_1}\alpha_{a_2,a_3,b_2,b_3}}{a_1! a_2! a_3! b_1! b_2! b_3!}\right |\cr
&\times\Big|\sum_{(Q,Q')\text{ non sing.}}\cr
& \int_Q \int_{Q'}  \E\Big [ H_{i_1}(T_n(x)) H_{i_2}(\widetilde{\partial_1}  T_n(x)) H_{i_3}(\widetilde{\partial_2  }T_n(x))H_{j_1}(\widehat T_n(x)) H_{j_2}(\widetilde{\partial_1}  {\widehat T}_n(x)) H_{j_3}(\widetilde{\partial_2} {\widehat T}_n(x))\cr
&\times H_{a_1}(T_n(y)) H_{a_2}(\widetilde{\partial_1}  T_n(y)) H_{a_3}(\widetilde{\partial_2} T_n(y))H_{b_1}(\widehat T_n(y)) H_{b_2}(\widetilde{\partial_1}  {\widehat T}_n(y)) H_{b_3}(\widetilde{\partial_2} {\widehat T}_n(y))
\Big ]\,dxdy\Big|.
\end{split}
\end{equation}
Let us now adopt the same notation as in Proposition \ref{p:ls}.
For $n\in S$ we set
\begin{eqnarray*}
&& (X_0(x), X_1(x), X_2(x), Y_0(x),  Y_1(x), Y_2(x))\\
&&\quad\quad\quad := (T_{n}(x),\, \widetilde\partial_{1} T_{n}(x), \, \widetilde\partial_{2} T_{n}(x), \hT(x),\, \widetilde \partial_{1}\hT(x), \, \widetilde \partial_{2}\hT(x)), \quad x\in \Tb.
\end{eqnarray*}
From Proposition \ref{p:ls} and \paref{lungo1}, we have
\begin{eqnarray}\label{lungo2}
&& \left| \sum_{(Q,Q')\text{ non sing.}} \Cov\left (\text{proj}\left(I_{n_{|_{Q}}}|C_{\ge 6}\right),\text{proj}\left(I_{n_{|_{Q'}}}|C_{\ge 6}\right) \right) \right|\\
&&\le E_n^2 \sum_{q\ge 3} \sum_{i_1+i_2+i_3+j_1+j_2+j_3=2q}\sum_{a_1+a_2+a_3+b_1+b_2+b_3=2q} \left |\frac{\beta_{i_1}\beta_{j_1}\alpha_{i_2,i_3,j_2,j_3}}{i_1! i_2! i_3! j_1! j_2! j_3!}\right |\cdot \left |\frac{\beta_{a_1}\beta_{b_1}\alpha_{a_2,a_3,b_2,b_3}}{a_1! a_2! a_3! b_1! b_2! b_3!}\right | \notag \\
&&\quad\times {\bf 1}_{\{i_1+i_2+i_3 = a_1+a_2+a_3\}}{\bf 1}_{\{j_1+j_2+j_3 = b_1+b_2+b_3\}}  \Big| V(i_1,i_2,i_3 ; j_1,j_2,j_3 ; a_1,a_2,a_3 ; b_1,b_2,b_3)  \Big| ,\notag \\
&& :=E_n^2 \times Z,\label{e:rol}
\end{eqnarray}
where each of the terms $V = V(i_1,i_2,i_3 ; j_1,j_2,j_3 ; a_1,a_2,a_3 ; b_1,b_2,b_3)$ is the sum of no more than $(2q)!$ terms of the type 
\begin{equation}\label{e:sergio}
v =\sum_{(Q,Q')\text{ non sing.}} \int_Q\int_{Q'} \prod_{u=1}^{2q} R_{l_u, k_u}({\green x-y}) \, dxdy,
\end{equation}
where $k_u, l_u \in \{0,1,2\}$ and {\green where}, for $l,k=0,1,2$ and $x,y\in \mathbb T$, we set
$$
{\green R_{l,k}(x-y)}:= \E\left[X_l(x) X_k(y)   \right]= \E\left[Y_l(x) Y_k(y)   \right].
$$
Note that, for any {even $p\in \N$}, we have
\begin{equation}\label{deriv}
\int_{\mathbb T} R_{l,k}(x)^{p}\,dx \leq \int_{\mathbb T} {\green r_n}(x)^p\,dx{\green =:R_n(p)}
\end{equation}
and recall moreover that, for $x,y\in \T$, $|R_{l,k}(x-y)|\leq 1$, and, for $(x,y)\in Q\times Q'$,
\begin{equation}\label{epsuno}
|R_{l,k}(x-y)|<\eps_1.
\end{equation}
Using the definition of a non-singular pair of cubes, as well as the fact that the sum defining $Z$ in \eqref{e:rol} involves indices $q\geq 3$, one deduces that, for $v$ as in \eqref{e:sergio}, 
\begin{eqnarray*}
|v| &\leq& \eps_1^{2q-6} \sum_{(Q,Q')\text{ non sing.}} \int_Q\int_{Q'} \prod_{u=1}^{6} \left |R_{l_u, k_u}({\green x-y})\right| dxdy\\
&\leq &  \eps_1^{2q-6} \int_\T \prod_{u=1}^{6} \left |R_{l_u, k_u}(x)\right| dx \leq  \eps_1^{2q-6} R_n(6),
\end{eqnarray*}
where we have applied a generalized H\"older inequality together with \eqref{deriv} in order to deduce the last estimate. This bound implies that each of the terms $V$ contributing to $Z$ can be bounded as follows:
\begin{eqnarray*}
&&\Big| V(i_1,i_2,i_3 ; j_1,j_2,j_3 ; a_1,a_2,a_3 ; b_1,b_2,b_3) \Big|\\
&&\quad\quad \leq (2q)! \frac{R_n(6)}{\eps_1^6} \eps_1^{2q} = (2q)! \frac{R_n(6)}{\eps_1^6} ( \sqrt{\eps_1} )^{i_1+\cdots+ j_3}( \sqrt{\eps_1} )^{a_1+\cdots+ b_3}.
\end{eqnarray*}
One therefore infers that 
\begin{eqnarray*}
&& Z\leq \frac{R_n(6)}{\eps_1^6} \sum_{q\geq 3} (2q)! \sum_{i_1+i_2+i_3+j_1+j_2+j_3=2q}\sum_{a_1+a_2+a_3+b_1+b_2+b_3=2q} \left |\frac{\beta_{i_1}\beta_{j_1}\alpha_{i_2,i_3,j_2,j_3}}{i_1! i_2! i_3! j_1! j_2! j_3!}\right | \times \notag \\
&&\quad\quad \quad\quad\quad \quad \left |\frac{\beta_{a_1}\beta_{b_1}\alpha_{a_2,a_3,b_2,b_3}}{a_1! a_2! a_3! b_1! b_2! b_3!}\right |\times  ( \sqrt{\eps_1} )^{i_1+\cdots+ j_3}( \sqrt{\eps_1} )^{a_1+\cdots+ b_3} =:\frac{R_n(6)}{\eps_1^6}\times S  .\notag
\end{eqnarray*}
In order to show that $S$ is finite, we write
\begin{eqnarray*}
S& = & \sum_{q \geq 3} (2q)! \sum_{i_1+i_2+i_3+j_1+j_2+j_3=2q}\sum_{a_1+a_2+a_3+b_1+b_2+b_3=2q} \left |\frac{\beta_{i_1}\beta_{j_1}\alpha_{i_2,i_3,j_2,j_3}}{i_1! i_2! i_3! j_1! j_2! j_3!}\right | \times \notag \\
&&\quad\quad \quad\quad\quad \quad \left |\frac{\beta_{a_1}\beta_{b_1}\alpha_{a_2,a_3,b_2,b_3}}{a_1! a_2! a_3! b_1! b_2! b_3!}\right |\times  ( \sqrt{\epsilon_1} )^{i_1+\cdots+ j_3}( \sqrt{\epsilon_1} )^{a_1+\cdots+ b_3}\\
&\leq & \sum_{q\geq 0}  \sum_{i_1+i_2+i_3+j_1+j_2+j_3=2q}\sum_{a_1+a_2+a_3+b_1+b_2+b_3=2q} \left |\frac{\beta_{i_1}\beta_{j_1}\alpha_{i_2,i_3,j_2,j_3}}{i_1! i_2! i_3! j_1! j_2! j_3!}\right | \times \notag \\
&&\sqrt{(i_1+\cdots+j_3)!} \sqrt{(a_1+\cdots+b_3)!}  \left| \frac{\beta_{a_1}\beta_{b_1}\alpha_{a_2,a_3,b_2,b_3}}{a_1! a_2! a_3! b_1! b_2! b_3!}\right |\times  ( \sqrt{\epsilon_1} )^{i_1+\cdots+ j_3+a_1+\cdots+ b_3}\\
&\leq & \sum_{i_1,\ldots ,j_3, a_1\ldots, b_3} \left |\frac{\beta_{i_1}\beta_{j_1}\alpha_{i_2,i_3,j_2,j_3}}{i_1! i_2! i_3! j_1! j_2! j_3!}\right | \times \notag \\
&&\sqrt{(i_1+\cdots+j_3)!} \sqrt{(a_1+\cdots+b_3)!}  \left| \frac{\beta_{a_1}\beta_{b_1}\alpha_{a_2,a_3,b_2,b_3}}{a_1! a_2! a_3! b_1! b_2! b_3!}\right |\times  ( \sqrt{\epsilon_1} )^{i_1+\cdots+ j_3+a_1+\cdots+ b_3} \\
& \leq & \Bigg ( \sum_{i_1,\ldots ,j_3, a_1\ldots, b_3} \left |\frac{\beta_{i_1}\beta_{j_1}\alpha_{i_2,i_3,j_2,j_3}}{i_1! i_2! i_3! j_1! j_2! j_3!}\right |^2  (i_1+\cdots+j_3)!  ( \sqrt{\epsilon_1} )^{i_1+\cdots+ j_3+a_1+\cdots+ b_3}\Bigg)^{1/2}\times \\
&& \times \Bigg( \sum_{i_1,\ldots ,j_3, a_1\ldots, b_3}\left| \frac{\beta_{a_1}\beta_{b_1}\alpha_{a_2,a_3,b_2,b_3}}{a_1! a_2! a_3! b_1! b_2! b_3!}\right |^2 (a_1+\cdots+b_3)!( \sqrt{\epsilon_1} )^{i_1+\cdots+ j_3+a_1+\cdots+ b_3}\Bigg)^{1/2}
\\ & =&  \sum_{i_1,\ldots ,j_3, a_1\ldots, b_3} \left |\frac{\beta_{i_1}\beta_{j_1}\alpha_{i_2,i_3,j_2,j_3}}{i_1! i_2! i_3! j_1! j_2! j_3!}\right |^2  (i_1+\cdots+j_3)!  ( \sqrt{\epsilon_1} )^{i_1+\cdots+ j_3+a_1+\cdots+ b_3} < \infty,
\end{eqnarray*}
where: (a) the third inequality follows by applying the Cauchy-Schwarz inequality to the symmetric finite measure $\mu$ on $\mathbb{N}^{12}$ such that
$$
\mu\{ (k_1,...,k_{12} )  \} =    (\sqrt{\epsilon_1} )^{k_1+\cdots+ k_{12}}, 
$$
and, (b) writing $m = m(i_1,...,j_3) :=  i_1+\cdots +j_3$ for every $i_1,...,j_3$, the finiteness of the last sum is a consequence of the standard estimate
$$
\frac{ (i_1+\cdots +j_3)!}{i_1! i_2! i_3! j_1! j_2! j_3!} \leq \sum_{\substack{k_1,...,k_6\geq 0 \\ k_1+\cdots+k_6 = m} } \frac{ m!}{k_1! \cdots k_6!} = 6^m = 6^{i_1+\cdots +j_3},
$$
as well as of the fact that the mapping 
$$
(i_1,...,j_3)\mapsto \frac{\beta^2_{i_1}\beta^2_{j_1}\alpha^2_{i_2,i_3,j_2,j_3}}{i_1! i_2! i_3! j_1! j_2! j_3!} 
$$
is bounded, and $6\sqrt{\eps_1}<1$ by assumption. This concludes the proof.
}
\end{proof}

\section{End of the Proof of Theorem \ref{t:main}}\label{proofs main sec}

\subsection{Proof of Part 2}

From Lemma \ref{berry's cancellation}, for $n\in S$ the chaotic expansion for $I_n$ is
$$
I_n = \E[I_n] + \sum_{q\ge 2} I_n[2q],
$$
where $I_n[2q]$ is given in \paref{proiezione q}. Proposition \ref{varianze chaos},  Proposition \ref{asymptotic} together with Lemma \ref{lemma BB} immediately conclude the proof, once we recall that, by orthogonality of different Wiener chaoses
$$
\Var(I_{n}) = \Var(I_{n}[4]) + \sum_{q\ge 3} \Var(I_{n}[2q]).
$$

\subsection{Proof of Part 4}

Part 2 of Theorem \ref{t:main} yields that, as $\mathcal{N}_{n_j}\to +\infty$,
$$
\frac{I_{n_j}-\E[I_{n_j}]}{\sqrt{\Var(I_{n_j})}} = \frac{I_{n_j}[4]}{\sqrt{\Var(I_{n_j}[4])}} + o_\P(1).
$$
Proposition \ref{limite 4} hence allows to conclude the proof.

\section{Some technical computations}\label{appendix}

\subsection{Technical proofs}

\begin{proof}[Proof of Lemma \ref{alpha piccoli}]
We have
\begin{eqnarray*}
\alpha_{0,0,0,0}&=&\mathbb{E} [|XW-YV|] \\ &=& \frac{1}{(2\pi)^2} \left( \int_0^\infty \rho^2 e^{-\rho^2/2} d\rho\right)^2 \int_0^{2\pi}\int_0^{2\pi} | \sin\theta \cos \theta' - \sin \theta'\cos\theta| d\theta d\theta'\\
&=& \frac{1}{(2\pi)^2} \left( \int_0^\infty \rho^2 e^{-\rho^2/2} d\rho\right)^2 \int_0^{2\pi}\int_0^{2\pi} | \sin( \theta  - \theta') | d\theta d\theta' = 1.
\end{eqnarray*}
Setting $Z$ to be any of the variables $X,Y,V,W$ and $\varphi_Z(u)$ to be $\cos(u)$ if $Z=X,V$, or $\sin(u)$ if $Z = Y,W$, we have that
\begin{eqnarray*}
&&\mathbb{E} [|XW-YV| H_2(Z) ]\\
&=& \frac{1}{(2\pi)^2}  \int_0^\infty \rho^2 e^{-\rho^2/2} d\rho \int_0^\infty \gamma^4 e^{-\gamma^2/2} d\gamma \int_0^{2\pi}\int_0^{2\pi} | \sin( \theta  - \theta')| \varphi_Z(\theta)^2 d\theta d\theta' -1= \frac12.
\end{eqnarray*}
As a result, we deduce that
$$
\alpha_{2,0,0,0}= \alpha_{0,2,0,0}=\alpha_{0,0,2,0}=\alpha_{0,0,0,2}=\frac12.
$$
Let us now concentrate on $\alpha_{4,0,0,0} $. We have
\begin{eqnarray*}
\alpha_{4,0,0,0} &=&\E[|XW-YV|H_4(X)]\\
&=&
\E[|XW-YV|X^4] - 6 \underbrace{\E[|XW-YV|X^2]}_{=\frac32\text{ from above}}+3\underbrace{\E[|XW-YV|]}_{=1}.
\end{eqnarray*}
Thus, it remains to calculate
\begin{eqnarray*}
&&\E[|XW-YV|X^4]\\
&=&\frac{1}{(2\pi)^2}\int_{\mathbb R^4}|xw-yv| x^4 \e^{-x^2/2}\e^{-y^2/2}\e^{-v^2/2}\e^{-w^2/2}\,dxdydvdw\\
&=&\frac{1}{2\pi}\underbrace{ \int_0^{2\pi}\cos^4\theta d\theta}_{=\frac{3\pi}{4}}\underbrace{\int_0^{2\pi} | \sin( \theta' - \theta)| d\theta'}_{=4}\underbrace{\frac{1}{\sqrt{2\pi}}  \int_0^\infty(\rho')^2  \e^{-(\rho')^2/2}d\rho'}_{=\frac12}\underbrace{\frac{1}{\sqrt{2\pi}}\int_0^\infty \rho^6  \e^{-\rho^2/2}\, d\rho}_{=\frac{15}{2}}\\
&=&\frac{45}{8}.
\end{eqnarray*}
Plugging into the previous expression, we deduce
$$
\alpha_{4,0,0,0} =-\frac{3}{8}.
$$
Since
$\int_0^{2\pi}\cos^4\theta d\theta = \int_0^{2\pi}\sin^4\theta d\theta$,
it is immediate to check that
$$
\alpha_{4,0,0,0}=\alpha_{0,4,0,0}=\alpha_{0,0,4,0}=\alpha_{0,0,0,4}.
$$
Let us now compute $\alpha_{2,2,0,0} $. We have
\begin{eqnarray*}
&&\alpha_{2,2,0,0} = \E[|XW-YV|H_2(X)H_2(Y)]\\
&=&\E[|XW-YV|X^2 Y^2]
- \underbrace{\E[|XW-YV|X^2]}_{=\frac32}
-\underbrace{\E[|XW-YV|Y^2] }_{=\frac32}+\underbrace{\E[
|XW-YV|]}_{=1},
\end{eqnarray*}
whereas
\begin{eqnarray*}
&&\E[|XW-YV|X^2 Y^2] \\
&=&\frac{1}{(2\pi)^2}\int_{\mathbb R^4}|xw-yv| x^2y^2 \e^{-x^2/2}\e^{-y^2/2}\e^{-v^2/2}\e^{-w^2/2}\,dxdydvdw\\
&=&\frac{1}{2\pi}\underbrace{ \int_0^{2\pi} \!\!\!\cos^2\theta \, \sin^2\theta d\theta}_{=\frac{\pi}{4}}\underbrace{\int_0^{2\pi} \!\!\!| \sin( \theta' - \theta)| d\theta'}_{=4}\underbrace{\frac{1}{\sqrt{2\pi}}  \int_0^\infty  \!\!\!(\rho')^2  \e^{-(\rho')^2/2}d\rho'}_{=\frac12}\underbrace{\frac{1}{\sqrt{2\pi}}\int_0^\infty \!\!\!\rho^6  \e^{-\rho^2/2}\, d\rho}_{=\frac{15}{2}}\\
&=&\frac{15}{8}.
\end{eqnarray*}
Therefore
$$
\alpha_{2,2,0,0}=-\frac{1}{8}.
$$
Similarly,
$$
\alpha_{0,0,2,2}=\alpha_{2,2,0,0}=-\frac{1}{8}.
$$
Now, let us compute $\alpha_{2,0,2,0} $. We can write
\begin{eqnarray*}
\alpha_{2,0,2,0} &=& \E[|XW-YV|H_2(X)H_2(V)]\\
&=&\E[|XW-YV|X^2 V^2] - \underbrace{\E[|XW-YV|X^2]}_{=\frac32}-\underbrace{\E[|XW-YV|V^2] }_{=\frac32}+\underbrace{\E[|XW-YV|]}_{=1},
\end{eqnarray*}
whereas
\begin{eqnarray*}
&&\E[|XW-YV|X^2V^2] \\
&=&\frac{1}{(2\pi)^2}\int_{\mathbb R^4}|xw-yz| x^2v^2 \e^{-x^2/2}\e^{-y^2/2}\e^{-v^2/2}\e^{-w^2/2}\,dxdydvdw\\
&=&\frac{1}{2\pi}\underbrace{\int_0^{2\pi} \!\!\!\cos^2\theta \,  d\theta\int_0^{2\pi} \!\!\!| \sin( \theta' - \theta)|\cos^2\theta' d\theta'}_{=\frac{5\pi}{3}}\underbrace{\frac{1}{\sqrt{2\pi}}  \int_0^\infty\!\!\!(\rho')^4  \e^{-(\rho')^2/2}d\rho'}_{=\frac32}\underbrace{\frac{1}{\sqrt{2\pi}}\int_0^\infty \!\!\!\rho^4  \e^{-\rho^2/2}\, d\rho}_{=\frac{3}{2}}\\
&=&\frac{15}{8}.
\end{eqnarray*}
Then,
$$
\alpha_{2,0,2,0}=\alpha_{0,2,0,2}-\frac{1}{8}.
$$
We also compute
\begin{eqnarray*}
\alpha_{2,0,0,2} &=& \E[|XW-YV|H_2(X)H_2(W)]\\
&=&\E[|XW-YV|X^2W^2]
- \underbrace{\E[|XW-YV|X^2 ]}_{=\frac32}
-\underbrace{\E[|XW-YV|W^2] }_{=\frac32}
+\underbrace{\E[|XW-YV|]}_{=1}.
\end{eqnarray*}
We have
\begin{eqnarray*}
&&\E[|XW-YV|X^2W^2] \\
&=&\frac{1}{(2\pi)^2}\int_{\mathbb R^4}|xw-yv| x^2w^2 \e^{-x^2/2}\e^{-y^2/2}\e^{-v^2/2}\e^{-w^2/2}\,dxdydvdw\\
&=&\frac{1}{2\pi}\underbrace{\int_0^{2\pi} \!\!\!\cos^2\theta \,  d\theta\int_0^{2\pi} | \sin( \theta' - \theta)|\sin^2\theta' d\theta'}_{=\frac{7\pi}{3}}\underbrace{\frac{1}{\sqrt{2\pi}}  \int_0^\infty\!\!\!(\rho')^4  \e^{-(\rho')^2/2}d\rho'}_{=\frac32}\underbrace{\frac{1}{\sqrt{2\pi}}\int_0^\infty \!\!\!\rho^4  \e^{-\rho^2/2}\, d\rho}_{=\frac{3}{2}}\\
&=&\frac{21}{8},
\end{eqnarray*}
so that
$$
\alpha_{2,0,0,2}=\alpha_{0,2,2,0}=\frac{5}{8}.
$$
Finally, let us consider the case where $a=b=c=d=1$.
We have
\begin{eqnarray*}
&&\alpha_{1,1,1,1}=\E[|XW-YV|XYVW] \\
&=&\frac{1}{(2\pi)^2}\int_{\mathbb R^4}|xw-yv| xyvw \e^{-x^2/2}\e^{-y^2/2}\e^{-v^2/2}\e^{-w^2/2}\,dxdydvdw\\
&=&\frac{1}{2\pi}\underbrace{\int_{[0,2\pi]^2} | \sin( \theta' - \theta)|\cos\theta\cos\theta'\sin\theta\sin\theta' d\theta d\theta'}_{=-\frac{\pi}{3}}\underbrace{\frac{1}{\sqrt{2\pi}}  \int_0^\infty(\rho')^4  \e^{-(\rho')^2/2}d\rho'}_{=\frac32}\\
&&\hskip8.5cm\times\underbrace{\frac{1}{\sqrt{2\pi}}\int_0^\infty \rho^4  \e^{-\rho^2/2}\, d\rho}_{=\frac{3}{2}}\\
&=&-\frac{3}{8}.
\end{eqnarray*}
\end{proof}

\subsection{Proof of Lemma \ref{lemma infinito}}

\noindent{\it Proof of (i)}. We have

\begin{eqnarray*}
&&\int_{\T} H_4(T_n(x))\,dx = \int_{\T}  (T_n(x)^4 - 6T_n(x)^2+3)\,dx\\
&=&\frac{1}{\mathcal N_n^2}\!\sum_{\lambda, \lambda',\lambda'',\lambda'''\in\Lambda_n} \!\!\!\!\!\!a_\lambda \overline{a_{\lambda'}}\, a_{\lambda''} \overline{{a}_{\lambda'''}}\!\int \!\e_{\lambda - \lambda'+\lambda'' -\lambda'''}(x)\,dx -\frac{6}{\mathcal N_n}\sum_{\lambda, \lambda'\in\Lambda_n} \!\!\!a_\lambda \overline{a_{\lambda'}}\int \e_{\lambda - \lambda'}(x)\,dx +3\\
&=&\frac{1}{\mathcal N_n^2}\sum_{\lambda,\lambda''} |a_\lambda|^2
|a_{\lambda''}|^2
+\frac{1}{\mathcal N_n^2}\sum_{\lambda} |a_\lambda|^4
+\frac{2}{\mathcal N_n^2}\sum_{\lambda\neq \pm\lambda'} |a_\lambda|^2 |a_{\lambda'}|^2
 -\frac{6}{\mathcal N_n}\sum_{\lambda} |a_\lambda|^2 +3\\
&=&\frac{3}{\mathcal N_n^2}\sum_{\lambda,\lambda''} (|a_\lambda|^2-1)
(|a_{\lambda''}|^2-1)
-\frac{3}{\mathcal N_n^2}\sum_{\lambda} |a_\lambda|^4
=\frac{3}{\mathcal N_n} W{}{(n)}^2 -\frac{3}{\mathcal N_n^2}\sum_{\lambda} |a_\lambda|^4.
\end{eqnarray*}
Since $\frac1{\mathcal N_n}\sum_{\lambda} |a_\lambda|^4\to 2$ by the law of large numbers, the claim $(i)$ follows.

\bigskip
\newpage

\noindent{\it Proof of (ii)}. We have

\begin{eqnarray*}
&&\int_{\T} H_4(\widetilde{\partial}_j T_n(x))\,dx = \int_{\T}  (\widetilde{\partial}_j T_n(x)^4 - 6\,\widetilde{\partial}_j T_n(x)^2+3)\,dx\\
&=&\frac{4}{n^2\mathcal N_n^2}\sum_{\lambda, \lambda',\lambda'',\lambda'''\in\Lambda_n} \lambda_j \lambda'_j\lambda''_j\lambda'''_j\,a_\lambda \overline{a_{\lambda'}}\, a_{\lambda''} \overline{{a}_{\lambda'''}}\int \e_{\lambda - \lambda'+\lambda'' -\lambda'''}(x)\,dx \\
&&-\frac{12}{n\mathcal N_n}\sum_{\lambda, \lambda'\in\Lambda_n}  \lambda_j \lambda'_j\,a_\lambda \overline{a_{\lambda'}}\int \e_{\lambda - \lambda'}(x)\,dx +3\\
&=&\frac{4}{n^2\mathcal N_n^2}\sum_{\lambda,\lambda''} \lambda_j^2{\lambda''_j}^2|a_\lambda|^2
|a_{\lambda''}|^2
+\frac{4}{n^2\mathcal\mathcal N_n^2}\sum_{\lambda} \lambda_j^4 |a_\lambda|^4
+\frac{8}{n^2\mathcal\mathcal N_n^2}\sum_{\lambda\neq \pm\lambda'} \lambda_j^2{\lambda'_j}^2|a_\lambda|^2 |a_{\lambda'}|^2
\\
&& -\frac{12}{n\mathcal N_n}\sum_{\lambda}\lambda_j^2 |a_\lambda|^2 +3
\end{eqnarray*}
\begin{eqnarray*}
&=&\frac{12}{n^2\mathcal N_n^2}\sum_{\lambda,\lambda''} \lambda_j^2{\lambda''_j}^2|a_\lambda|^2
|a_{\lambda''}|^2
-\frac{12}{n^2\mathcal\mathcal N_n^2}\sum_{\lambda} \lambda_j^4 |a_\lambda|^4
-\frac{12}{n\mathcal N_n}\sum_{\lambda} \lambda_j^2|a_\lambda|^2 +3\\
&=&\frac{12}{n^2\mathcal N_n^2}\sum_{\lambda,\lambda''}  \lambda_j^2{\lambda''_j}^2(|a_\lambda|^2-1)
(|a_{\lambda''}|^2-1)
-\frac{12}{n^2\mathcal\mathcal N_n^2}\sum_{\lambda} \lambda_j^4 |a_\lambda|^4\\
&=&\frac{12}{\mathcal N_n} W_j{}{(n)}^2 -\frac{12}{n^2\mathcal\mathcal N_n^2}\sum_{\lambda} \lambda_j^4 |a_\lambda|^4 .
\end{eqnarray*}
Since $\frac1{n^2\mathcal N_n}\sum_{\lambda} \lambda_j^4 |a_\lambda|^4\to \frac14(3+\widehat{\mu}_\infty(4))$ by the law of large numbers, the claim $(ii)$ follows.

\bigskip

\noindent{\it Proof of (iii)}. We have

\begin{eqnarray*}
&&\int_{\T} H_2(T_n(x))\big(H_2(\widetilde{\partial}_1 T_n(x)+H_2(\widetilde{\partial}_2 T_n(x))\big)\,dx
\\
&=&\int_{\T} \big(T_n(x)^2\widetilde{\partial}_1 T_n(x)^2+T_n(x)^2\widetilde{\partial}_2 T_n(x)^2
-2T_n(x)^2-\widetilde{\partial}_1 T_n(x)^2-\widetilde{\partial}_2 T_n(x)^2+2
\big)\,dx
\\
&=&\frac{2}{n\mathcal N_n^2}\sum_{\lambda, \lambda',\lambda'',\lambda'''\in\Lambda_n} (\lambda''_1\lambda'''_1+\lambda''_2\lambda'''_2)\,a_\lambda \overline{a_{\lambda'}}\, a_{\lambda''} \overline{{a}_{\lambda'''}}\int \e_{\lambda - \lambda'+\lambda'' -\lambda'''}(x)\,dx \\
&&-\frac{2}{\mathcal N_n}\sum_{\lambda, \lambda'\in\Lambda_n}  a_\lambda \overline{a_{\lambda'}}\int \e_{\lambda - \lambda'}(x)\,dx
-\frac{2}{n\mathcal N_n}\sum_{\lambda, \lambda'\in\Lambda_n} (\lambda_1\lambda'_1+\lambda_2\lambda'_2) a_\lambda \overline{a_{\lambda'}}\int \e_{\lambda - \lambda'}(x)\,dx
+2\\
&=&\frac{2}{\mathcal N_n^2}\sum_{\lambda, \lambda''} |a_\lambda|^2 | a_{\lambda''} |^2
- \frac{2}{\mathcal N_n^2}\sum_{\lambda} |a_\lambda|^4
+\frac{4}{n\mathcal N_n^2}\sum_{\lambda\neq \pm\lambda'} (\lambda_1\lambda'_1+\lambda_2\lambda'_2)|a_\lambda|^2 | a_{\lambda'} |^2\\
&&-\frac{4}{\mathcal N_n}\sum_{\lambda}  |a_\lambda|^2 +2\\
&=&\frac{2}{\mathcal N_n^2}\sum_{\lambda, \lambda''} (|a_\lambda|^2 -1)(| a_{\lambda''} |^2 -1)
- \frac{2}{\mathcal N_n^2}\sum_{\lambda} |a_\lambda|^4
=\frac{2}{\mathcal N_n}\left\{W{}{(n)}^2 - \frac{1}{\mathcal N_n}\sum_{\lambda} |a_\lambda|^4 \right\}.
\end{eqnarray*}
Since $\frac1{\mathcal N_n}\sum_{\lambda} |a_\lambda|^4\to 2$ by the law of large numbers, the claim $(iii)$ follows.

\bigskip\newpage
\noindent{\it Proof of (iv)}. We have

\begin{eqnarray*}
&&\int_{\T} H_2(\widetilde{\partial}_1 T_n(x))H_2(\widetilde{\partial}_2 T_n(x))\,dx
=\int_{\T} \big(\widetilde{\partial}_1 T_n(x)^2\widetilde{\partial}_2 T_n(x)^2-\widetilde{\partial}_1 T_n(x)^2-\widetilde{\partial}_2 T_n(x)^2+1)\,dx
\\
&=&\frac{4}{n^2\mathcal N_n^2}\sum_{\lambda, \lambda',\lambda'',\lambda'''\in\Lambda_n} \lambda_1\lambda'_1\lambda_2''\lambda'''_2\,a_\lambda \overline{a_{\lambda'}}\, a_{\lambda''} \overline{{a}_{\lambda'''}}\int \e_{\lambda - \lambda'+\lambda'' -\lambda'''}(x)\,dx \\
&&-\frac{2}{n\mathcal N_n}\sum_{\lambda, \lambda'\in\Lambda_n}  \lambda_1\lambda'_1a_\lambda \overline{a_{\lambda'}}\int \e_{\lambda - \lambda'}(x)\,dx
-\frac{2}{n\mathcal N_n}\sum_{\lambda, \lambda'\in\Lambda_n}  \lambda_2\lambda'_2a_\lambda \overline{a_{\lambda'}}\int \e_{\lambda - \lambda'}(x)\,dx
+1
\\
&=&\frac{4}{n^2\mathcal N_n^2}\sum_{\lambda, \lambda''} \lambda_1^2{\lambda_2''}^2|a_\lambda|^2 | a_{\lambda''} |^2
+ \frac{4}{n^2\mathcal N_n^2}\sum_{\lambda} \lambda_1^2\lambda_2^2|a_\lambda|^4
+\frac{8}{n^2\mathcal N_n^2}\sum_{\lambda\neq \pm\lambda'} \lambda_1\lambda_2\lambda'_1\lambda'_2|a_\lambda|^2 | a_{\lambda'} |^2\\
&&-\frac{2}{\mathcal N_n}\sum_{\lambda} |a_\lambda|^2 +1\\
\end{eqnarray*}
\begin{eqnarray*}
&=&\frac{4}{n^2\mathcal N_n^2}\sum_{\lambda, \lambda''} \lambda_1^2{\lambda_2''}^2(|a_\lambda|^2-1) (| a_{\lambda''} |^2-1)
- \frac{12}{n^2\mathcal N_n^2}\sum_{\lambda} \lambda_1^2\lambda_2^2|a_\lambda|^4
+\frac{8}{n^2\mathcal N_n^2}\left(\sum_{\lambda} \lambda_1\lambda_2|a_\lambda|^2 \right)^2\\
&=&\frac{4}{\mathcal N_n}\left\{W_1{}{(n)}W_2{}{(n)} +2W_{1,2}{}{(n)}^2 - \frac{3}{n^2\mathcal N_n}\sum_{\lambda} \lambda_1^2\lambda_2^2|a_\lambda|^4
\right\}.
\end{eqnarray*}
Since $\frac{1}{n^2\mathcal N_n}\sum_{\lambda} \lambda_1^2\lambda_2^2|a_\lambda|^4 \to \frac14(1-\widehat{\mu}_\infty(4))$ by the law of large numbers, the claim $(iv)$ follows.

\bigskip

\noindent{\it Proof of (v)}. We have

\begin{eqnarray*}
&&\int_{\T} H_2(T_n(x))H_2(\widehat T_n(x))\,dx =\int_{\T} \big(T_n(x)^2 \widehat T_n(x)^2- T_n(x)^2- \widehat T_n(x)^2 + 1\big)dx\\
&=&\frac{1}{\mathcal N_n^2}\sum_{\lambda, \lambda',\lambda'',\lambda'''\in\Lambda_n} a_\lambda \overline{a_{\lambda'}}\, \widehat{a}_{\lambda''} \overline{\widehat{a}_{\lambda'''}}\int \e_{\lambda - \lambda'+\lambda'' -\lambda'''}(x)\,dx
-\frac{1}{\mathcal N_n}\sum_{\lambda, \lambda'\in\Lambda_n}  a_\lambda \overline{a_{\lambda'}}\int \e_{\lambda - \lambda'}(x)\,dx \\
&&-\frac{1}{\mathcal N_n}\sum_{\lambda, \lambda'\in\Lambda_n}  \widehat{a}_\lambda \overline{\widehat{a}_{\lambda'}}\int \e_{\lambda - \lambda'}(x)\,dx
+1\\
&=&\frac{1}{\mathcal N_n^2}\sum_{\lambda,\lambda'} (|a_\lambda|^2-1)  (|\widehat{a}_{\lambda'}|^2-1)
+\frac{1}{\mathcal N_n^2}\sum_{\lambda} a_\lambda^2\, \overline{\widehat{a}_{\lambda}}^2
+\frac{2}{\mathcal N_n^2}\sum_{\lambda\neq \pm \lambda'} a_\lambda\, \overline{a_{\lambda'}}\,\overline{\widehat{a}_{\lambda}}
\widehat{a}_{\lambda'}\\
&=&\frac{1}{\mathcal N_n^2}\sum_{\lambda,\lambda'} (|a_\lambda|^2-1)  (|\widehat{a}_{\lambda'}|^2-1)
-\frac{1}{\mathcal N_n^2}\sum_{\lambda} a_\lambda^2\, \overline{\widehat{a}_{\lambda}}^2
+\frac{2}{\mathcal N_n^2}\left(\sum_{\lambda} a_\lambda\, \overline{\widehat{a}_{\lambda}}\right)^2 - \frac{2}{\mathcal N_n^2}\sum_{\lambda} |a_\lambda|^2|\widehat{a}_{\lambda} |^2\\
&=&\frac{1}{\mathcal N_n}\left\{W{}{(n)}\widehat{W}{}{(n)}+2M{}{(n)}^2-\frac{1}{\mathcal N_n}\sum_{\lambda} a_\lambda^2\, \overline{\widehat{a}_{\lambda}}^2 - \frac{2}{\mathcal N_n}\sum_{\lambda} |a_\lambda|^2|\widehat{a}_{\lambda} |^2\right\}.
\end{eqnarray*}
Since $\frac{1}{\mathcal N_n}\sum_{\lambda}a_\lambda^2\, \overline{\widehat{a}_{\lambda}}^2 \to 0$
and $\frac{1}{\mathcal N_n}\sum_{\lambda} |a_\lambda|^2|\widehat{a}_{\lambda} |^2 \to 1$ by the law of large numbers, the claim $(v)$ follows.

\bigskip
\newpage

\noindent{\it Proof of (vi)}. We have

\begin{eqnarray*}
&&\int_{\T} H_2(T_n(x))H_2(\widetilde{\partial}_j \widehat T_n(x))\,dx=\int_{\T} \big(T_n(x)^2 \widetilde{\partial}_j\widehat T_n(x)^2\,-T_n(x)^2\,- \widetilde{\partial}_j\widehat T_n(x)^2\, + 1\big)dx\\
&=&\frac{2}{n\mathcal N_n^2}\sum_{\lambda, \lambda',\lambda'',\lambda'''} \lambda_j'' \lambda_j''' a_\lambda \overline{a_{\lambda'}}\, \widehat{a}_{\lambda''} \overline{\widehat{a}_{\lambda'''}}\int \e_{\lambda - \lambda'+\lambda'' -\lambda'''}(x)\,dx -\frac{1}{\mathcal N_n}\sum_\lambda |a_\lambda|^2 \\
&&-\frac{2}{n\mathcal N_n}\sum_\lambda \lambda_j^2 |\widehat{a}_\lambda|^2 +1\\
&=&\frac{2}{n\mathcal N_n^2}\sum_{\lambda, \lambda''} {\lambda_j''}^2 |a_\lambda|^2  |\widehat{a}_{\lambda''}|^2-\frac{2}{n\mathcal N_n^2}\sum_{\lambda} \lambda_j^2 a_\lambda^2 {\widehat{a}_{\lambda}}^2
+ \frac{2}{n\mathcal N_n^2}\sum_{\lambda\neq  \pm\lambda'} \lambda_j \lambda_j' a_\lambda \overline{a_{\lambda'}}\, \widehat{a}_{\lambda'} \overline{\widehat{a}_{\lambda}}\\
&&-\frac{1}{\mathcal N_n}\sum_\lambda |a_\lambda|^2 -\frac{2}{n\mathcal N_n}\sum_\lambda \lambda_j^2 |\widehat{a}_\lambda|^2 +1.
\end{eqnarray*}
Thus,
\begin{eqnarray*}
&&\int_{\T} H_2(T_n(x))\big(H_2(\widetilde{\partial}_1 \widehat T_n(x))+H_2(\widetilde{\partial}_2 \widehat T_n(x))\big)\,dx\\
&=&\frac{2}{\mathcal N_n^2}\sum_{\lambda, \lambda''}  (|a_\lambda|^2-1)  (|\widehat{a}_{\lambda''}|^2-1)-\frac{2}{\mathcal N_n^2}\sum_{\lambda} a_\lambda^2 {\widehat{a}_{\lambda}}^2
+ \frac{2}{n\mathcal N_n^2}\sum_{\lambda\neq  \pm\lambda'} (\lambda_1 \lambda_1' + \lambda_2 \lambda_2' )a_\lambda \overline{a_{\lambda'}}\, \widehat{a}_{\lambda'} \overline{\widehat{a}_{\lambda}}\\
&=&\frac{2}{\mathcal N_n^2}\sum_{\lambda, \lambda''}  (|a_\lambda|^2-1)  (|\widehat{a}_{\lambda''}|^2-1)
- \frac{2}{\mathcal N_n^2}\sum_{\lambda} |a_\lambda |^2\, |\widehat{a}_{\lambda}|^2
+ \frac{2}{n\mathcal N_n^2}\sum_{\lambda,\lambda'} (\lambda_1 \lambda_1' + \lambda_2 \lambda_2' )a_\lambda \overline{a_{\lambda'}}\, \widehat{a}_{\lambda'} \overline{\widehat{a}_{\lambda}}\\
&=&\frac{2}{\mathcal N_n}\left\{ W{}{(n)}\widehat{W}{}{(n)}+M_1{}{(n)}^2+M_2{}{(n)}^2-\frac{1}{\mathcal N_n}\sum_{\lambda} |a_\lambda |^2\, |\widehat{a}_{\lambda}|^2  \right\}.
\end{eqnarray*}
Since $\frac{1}{\mathcal N_n}\sum_{\lambda} |a_\lambda |^2\, |\widehat{a}_{\lambda}|^2  \to 1$
 by the law of large numbers, the claim $(vi)$ follows.

 \bigskip

\noindent{\it Proof of (vii)}. We have

\begin{eqnarray*}
&&\int_{\T}  H_2(\widetilde{\partial}_\ell T_n(x))H_2(\widetilde{\partial}_j \widehat T_n(x))\,dx\\
&=&\frac{4}{n^2\mathcal N_n^2}\sum_{\lambda, \lambda',\lambda'',\lambda'''} \lambda_\ell \lambda_\ell'\lambda_j'' \lambda_j''' a_\lambda \overline{a_{\lambda'}}\, \widehat{a}_{\lambda''} \overline{\widehat{a}_{\lambda'''}}\int \e_{\lambda - \lambda'+\lambda'' -\lambda'''}(x)\,dx\\
&&-\frac{2}{n\mathcal N_n}\sum_\lambda \lambda_\ell^2|a_\lambda|^2 -\frac{2}{n\mathcal N_n}\sum_\lambda \lambda_j^2 |\widehat{a}_\lambda|^2 +1\\
&=&\frac{4}{n^2\mathcal N_n^2}\sum_{\lambda,\lambda''} \lambda_\ell^2 {\lambda_j''}^2 |a_\lambda|^2\, |\widehat{a}_{\lambda''}|^2
+\frac{4}{n^2\mathcal N_n^2}\sum_{\lambda} \lambda_\ell^2 \lambda_j^2
a_\lambda^2 \overline{\widehat{a}_{\lambda}}^2
+\frac{8}{n^2\mathcal N_n^2}\sum_{\lambda\neq \pm \lambda'} \lambda_\ell \lambda_\ell'\lambda_j \lambda_j' a_\lambda \overline{a_{\lambda'}}\, \overline{\widehat{a}_{\lambda}} \widehat{a}_{\lambda'}
\\
&&-\frac{2}{n\mathcal N_n}\sum_\lambda \lambda_\ell^2|a_\lambda|^2 -\frac{2}{n\mathcal N_n}\sum_\lambda \lambda_j^2 |\widehat{a}_\lambda|^2 +1\\
&=&\frac{4}{n^2\mathcal N_n^2}\sum_{\lambda,\lambda''} \lambda_\ell^2 {\lambda_j''}^2 (|a_\lambda|^2-1)\, (|\widehat{a}_{\lambda''}|^2-1)
-\frac{4}{n^2\mathcal N_n^2}\sum_{\lambda} \lambda_\ell^2 \lambda_j^2
a_\lambda^2 \overline{\widehat{a}_{\lambda}}^2\\
&&+\frac{8}{n^2\mathcal N_n^2}\sum_{\lambda, \lambda'} \lambda_\ell \lambda_\ell'\lambda_j \lambda_j' a_\lambda \overline{a_{\lambda'}}\, \overline{\widehat{a}_{\lambda}} \widehat{a}_{\lambda'}
-\frac{8}{n^2\mathcal N_n^2}\sum_{\lambda} \lambda_\ell^2\lambda_j^2 |a_\lambda|^2\,|\widehat{a}_{\lambda}|^2
\\
\end{eqnarray*}
\begin{eqnarray*}
&=&\frac{4}{\mathcal N_n}\left\{ W_l{}{(n)}\widehat{W}_j{}{(n)}+2M_{\ell,j}{}{(n)}^2+
\frac{1}{n^2\mathcal N_n}\sum_{\lambda} \lambda_{\ell}^2\lambda_j^2\,a_\lambda ^2\, \widehat{a}_{\lambda}^2-\frac{2}{n^2\mathcal N_n}\sum_{\lambda} \lambda_{\ell}^2\lambda_j^2|a_\lambda |^2\, |\widehat{a}_{\lambda}|^2   \right\}.
\end{eqnarray*}
Since $\frac{1}{n^2\mathcal N_n}\sum_{\lambda} \lambda_{\ell}^2\lambda_j^2\,a_\lambda ^2\, \widehat{a}_{\lambda}^2\to 0$ and $$\frac{1}{n^2\mathcal N_n}\sum_{\lambda} \lambda_{\ell}^2\lambda_j^2|a_\lambda |^2\, |\widehat{a}_{\lambda}|^2  \to \frac18(1-\widehat{\mu}_\infty(4)){\bf 1}_{\{l\neq j\}}+\frac18(3+\widehat{\mu}_\infty(4)){\bf 1}_{\{l= j\}}$$
 by the law of large numbers, the claim $(vii)$ follows.

 \bigskip

\noindent{\it Proof of (viii)}. We have

\begin{eqnarray*}
&&\int_{\T} \widetilde{\partial}_1 T_n(x)\widetilde{\partial}_2 T_n(x)\widetilde{\partial}_1 \widehat T_n(x)\widetilde{\partial}_2 \widehat T_n(x)\,dx\\
&=&\frac{4}{n^2\mathcal N_n^2}\sum_{\lambda, \lambda',\lambda'',\lambda'''} \lambda_1 \lambda_2'\lambda_1'' \lambda_2''' a_\lambda \overline{a_{\lambda'}}\, \widehat{a}_{\lambda''} \overline{\widehat{a}_{\lambda'''}}\int \e_{\lambda - \lambda'+\lambda'' -\lambda'''}(x)\,dx \\
&=&\frac{4}{n^2\mathcal N_n^2}\sum_{\lambda, \lambda''} \lambda_1\lambda_2 {\lambda_1''}{\lambda_2''} |a_\lambda |^2|\widehat{a}_{\lambda''}|^2
+\frac{4}{n^2\mathcal N_n^2}\sum_{\lambda} \lambda_1^2 \lambda_2^2 a_\lambda^2  \overline{\widehat{a}_{\lambda}}^2
+\frac{4}{n^2\mathcal N_n^2}\sum_{\lambda\neq\pm \lambda'} \lambda_1^2 {\lambda_2'}^2 a_\lambda \overline{a_{\lambda'}}\, \widehat{a}_{\lambda'} \overline{\widehat{a}_{\lambda}}\\
&&
+\frac{4}{n^2\mathcal N_n^2}\sum_{\lambda\neq\pm \lambda'} \lambda_1 \lambda_2\lambda_1' \lambda_2' a_\lambda \overline{a_{\lambda'}}\, \widehat{a}_{\lambda'} \overline{\widehat{a}_{\lambda}}\\
&=&\frac{4}{n^2\mathcal N_n^2}\sum_{\lambda, \lambda''} \lambda_1\lambda_2 {\lambda_1''}{\lambda_2''} |a_\lambda |^2|\widehat{a}_{\lambda''}|^2
+\frac{4}{n^2\mathcal N_n^2}\sum_{\lambda, \lambda'} \lambda_1^2 {\lambda_2'}^2 a_\lambda \overline{a_{\lambda'}}\, \widehat{a}_{\lambda'} \overline{\widehat{a}_{\lambda}}\\
&&+\frac{4}{n^2\mathcal N_n^2}\sum_{\lambda, \lambda'} \lambda_1 \lambda_2\lambda_1' \lambda_2' a_\lambda \overline{a_{\lambda'}}\, \widehat{a}_{\lambda'} \overline{\widehat{a}_{\lambda}}
-\frac{8}{n^2\mathcal N_n^2}\sum_{\lambda} \lambda_1^2 \lambda_2^2 |a_\lambda|^2\, |\widehat{a}_{\lambda}|^2
-\frac{4}{n^2\mathcal N_n^2}\sum_{\lambda} \lambda_1^2 \lambda_2^2 a_\lambda^2\overline{\widehat{a}_{\lambda}}^2\\
&=&\frac{4}{\mathcal N_n}
\left\{
W_{1,2}{}{(n)}\widehat{W}_{1,2}{}{(n)}+M_{11}{}{(n)}M_{22}{}{(n)}+M_{12}{}{(n)}^2 - \frac{2}{n^2\mathcal N_n}\sum_{\lambda} \lambda_1^2 \lambda_2^2 |a_\lambda|^2\, |\widehat{a}_{\lambda}|^2\right.\\
&&\left. -\frac{1}{n^2\mathcal N_n}\sum_{\lambda} \lambda_1^2 \lambda_2^2 a_\lambda^2\overline{\widehat{a}_{\lambda}}^2
\right\}.
\end{eqnarray*}
Since $\frac{1}{n^2\mathcal N_n}\sum_{\lambda} \lambda_{1}^2\lambda_2^2\,a_\lambda ^2\, \widehat{a}_{\lambda}^2\to 0$ and $\frac{1}{n^2\mathcal N_n}\sum_{\lambda} \lambda_{1}^2\lambda_2^2|a_\lambda |^2\, |\widehat{a}_{\lambda}|^2  \to \frac18(1-\widehat{\mu}_\infty(4))$
 by the law of large numbers, the claim $(viii)$ follows.


\subsection{Taylor expansions for the two-point correlation function}\label{taylor}

{The matrix
\begin{equation}\label{matrix omega}
\Omega_n({\bf x}) := \begin{pmatrix}&\frac{E_n}{2} - \frac{(\partial_1 r_n({\bf x}))^2}{1-r_n({\bf x})^2} & - \frac{\partial_1 r_n({\bf x}) \partial_2 r_n({\bf x})}{1-r_n({\bf x})^2}\\
&- \frac{\partial_1 r_n({\bf x}) \partial_2 r_n({\bf x})}{1-r_n({\bf x})^2} &\frac{E_n}{2} - \frac{(\partial_2 r_n({\bf x}))^2}{1-r_n({\bf x})^2}
\end{pmatrix}
\end{equation}
is the covariance matrix of the random vector $\nabla T_n(x)$ conditioned to $T_n(x)=T_n(0)=0$ (see \cite[Equation (24)]{KKW}).}

We will write $\Omega= \Omega_n({\bf x})$, $E=E_n$ and $r=r_n({\bf x})$ for brevity. The determinant of $\Omega$ is
$$
\det \Omega = \frac{E}{2}\left ( \frac{E}{2} - \frac{(\partial_1 r)^2 + (\partial_2 r)^2}{1-r^2} \right).
$$
It is easy to show that $$
\lim_{\|{\bf x}\|\to 0} \det \Omega({\bf x}) = 0.
$$
However, we need the speed of convergence to $0$ of $\det \Omega$. Hence, we will use a Taylor expansion argument around $0$. We denote ${\bf x}=(x,y)$.
\begin{lemma}\label{lemma taylor} As $\|x\|\to 0$, we have
$$\displaylines{
\det \Omega_n(x) = c E_n^3 \|{\bf x}\|^2 +E_n^4\,O( \|{\bf x}\|^4),
}$$
and hence
$$
\Psi_n(x) :=\frac{|\Omega(x)|}{1-r^2(x)}= cE_n^2 + E_n^3 O\left (\|x\|^2\right),
$$
where both $c>0$ and the constants in the `$O$' notation do not depend on $n$.
\end{lemma}
\noindent
\begin{proof}
Let us start with $r$.
$$\displaylines{
r({\bf x}) = r(0,0) + \frac{1}{2} \langle \text{Hess}_r(0,0) {\bf x}, {\bf x} \rangle +\cr +\frac{1}{4!} \partial_{1111} r(0,0) x^4
+ \frac{1}{4!} \partial_{2222} r(0,0) y^4 +\frac{1}{2! 2!} \partial_{1122} r(0,0) x^2 y^2 + o_n(\|{\bf x}\|^4).
}$$
We have
$$
\text{Hess}_r(0,0) =\begin{pmatrix} &-\frac{E}{2} &0\\
&0 &-\frac{E}{2}
\end{pmatrix}
$$
and moreover
$$\displaylines{
 \partial_{1111} r(0,0) = (2\pi)^4 n^2 \psi_n\cr
 \partial_{2222} r(0,0) = (2\pi)^4 n^2 \psi_n\cr
\partial_{1122}r(0,0) = (2\pi)^4 n^2 (1/2 - \psi_n),
}$$
where $$
\psi_n := \frac{1}{n^2 \mathcal N_n}\sum_{\lambda \in \Lambda_n} \lambda_1^4 =\frac{1}{n^2 \mathcal N_n}\sum_{\lambda \in \Lambda_n} \lambda_2^4.
$$
Therefore we can write, as $\|{\bf x}\|\to 0$,
$$\displaylines{
r({\bf x}) = 1 - \frac{E}{4}(x^2 + y^2)  + \frac{1}{4!}(2\pi)^4 n^2 \psi_n x^4
+ \frac{1}{4!} (2\pi)^4 n^2 \psi_n y^4 +\frac{1}{2! 2!} (2\pi)^4 n^2 (1/2 - \psi_n) x^2 y^2 + R_n^r.
}$$
{More precisely, the remainder $R_n^r $ is of the form
$$
R_n^r = O \left(\sup \|\partial^6 r_n\| \| {\bf x }\|^6 \right ).
$$
It is easy to check that
$$
\Big |\partial^6 r_n \Big | \le E_n^3
$$
and hence
$$
R_n^r = E_n^3 O \left( {\bf x }\|^6 \right ),
$$
where the constants involved in the 'O' notation do not depend on $n$. }
Analogously, we find that
$$\displaylines{
\partial r_1 ({\bf x}) = -\frac{E}{2}x + \frac{1}{3!}  (2\pi)^4 n^2 \psi_n (x^3+y^3) +\frac{1}{2}  (2\pi)^4 n^2 (1/2 - \psi_n)xy^2 + R_n^1
}$$
and
$$\displaylines{
\partial r_2 ({\bf x}) = -\frac{E}{2}y + \frac{1}{3!}  (2\pi)^4 n^2 \psi_n (x^3+y^3) +\frac{1}{2}  (2\pi)^4 n^2 (1/2 - \psi_n)x^2 y+ R_n^2.
}$$
Also here, the remainders $R_n^1 $ and $R_n^2$ are both of the form
$$
R_n^j \le \sup \|\partial^6 r_n\| \cdot O(\| {\bf x }\|^5).
$$
where also here the constants involved in the 'O' notation do not depend on $n$. Squaring previous Taylor expansions we hence get
$$\displaylines{
r^2({\bf x})= 1 +\left( \frac{E}{4}\right)^2 (x^2+y^2)^2 - \frac{E}{2}(x^2 + y^2) +\cr
+2\left( \frac{1}{4!}(2\pi)^4 n^2 \psi_n x^4
+ \frac{1}{4!} (2\pi)^4 n^2 \psi_n y^4 +\frac{1}{2! 2!} (2\pi)^4 n^2 (1/2 - \psi_n) x^2 y^2\right) + o_n(\|{\bf x}\|^4)=\cr
=1 - \frac{E}{2}(x^2+y^2) + f_n(x,y) +E_n^3 \cdot O(\|{\bf x}\|^6),
}$$
where $f_n(x,y)$ is defined as
$$
\left( \frac{E}{4}\right)^2 (x^2+y^2)^2
+2\left( \frac{1}{4!}(2\pi)^4 n^2 \psi_n x^4
+ \frac{1}{4!} (2\pi)^4 n^2 \psi_n y^4 +\frac{1}{2! 2!} (2\pi)^4 n^2 (1/2 - \psi_n) x^2 y^2\right).
$$
{Therefore,
\begin{equation}\label{uno meno err quadro}
1-r^2 = \frac{E}{2}(x^2+y^2) - f_n(x,y) +E_n^3 \cdot O(\|{\bf x}\|^6).
\end{equation}}
Let us now investigate the derivatives {}{$(\partial_i r) ^2$, $i=1,2$. Firstly,}
$$\displaylines{
(\partial_1 r) ^2 =\left( \frac{E}{2}\right)^2 x^2 - E x \left( \frac{1}{3!}  (2\pi)^4 n^2 \psi_n (x^3+y^3) +\frac{1}{2}  (2\pi)^4 n^2 (1/2 - \psi_n)xy^2  \right) + E_n \cdot E_n^3 O(\|{\bf x}\|^6)
}$$
where the constants involved in the 'O' notation still do not depend on $n$.
{}{Secondly,}
$$\displaylines{
(\partial_2 r) ^2 =\left( \frac{E}{2}\right)^2 y^2 - E y \left( \frac{1}{3!}  (2\pi)^4 n^2 \psi_n (x^3+y^3) +\frac{1}{2}  (2\pi)^4 n^2 (1/2 - \psi_n)x^2y  \right) + E_n \cdot E_n^3 O(\|{\bf x}\|^6).
}$$
For brevity, let us denote
$$
a_n(x,y):= - E x \left( \frac{1}{3!}  (2\pi)^4 n^2 \psi_n (x^3+y^3) +\frac{1}{2}  (2\pi)^4 n^2 (1/2 - \psi_n)xy^2  \right)
$$
and
$$
b_n(x,y):= - E y \left( \frac{1}{3!}  (2\pi)^4 n^2 \psi_n (x^3+y^3) +\frac{1}{2}  (2\pi)^4 n^2 (1/2 - \psi_n)x^2y  \right),
$$
so that
$$\displaylines{
(\partial_1 r) ^2 =\left( \frac{E}{2}\right)^2 x^2 +a_n(x,y) + E_n \cdot E_n^3 O(\|{\bf x}\|^6),
}$$
and moreover
$$\displaylines{
(\partial_2 r) ^2 =\left( \frac{E}{2}\right)^2 y^2 + b_n(x,y) + E_n \cdot E_n^3 O(\|{\bf x}\|^6).
}$$
Thus we have, for fixed $n$, as $\|{\bf x}\|\to 0$, using also \paref{uno meno err quadro},
\begin{equation*}
\begin{split}
&\det \Omega = \frac{E}{2}\left ( \frac{E}{2} - \frac{(\partial_1 r)^2 + (\partial_2 r)^2}{1-r^2} \right)\cr
&= \frac{E}{2}\left ( \frac{E}{2} - \frac{\left( \frac{E}{2}\right)^2 (x^2+y^2) +a_n(x,y) + b_n(x,y) + E_n \cdot E_n^3 O(\|{\bf x}\|^6)}{\frac{E}{2}(x^2+y^2) - f_n(x,y) +E_n^3 O(\|{\bf x}\|^6)} \right)\cr
&= \frac{E}{2}\left ( \frac{E}{2} - \frac{E}{2} \frac{1 +\left(\frac{2}{E}\right)^2\left(\frac{a_n(x,y) + b_n(x,y)}{x^2 + y^2} + E_n \cdot E_n^3 O(\|{\bf x}\|^4)\right)}{1 - \frac{2}{E}\left(\frac{f_n(x,y)}{x^2+y^2} +E_n^3 O(\|{\bf x}\|^4)\right)} \right)\cr
&=\left(\frac{E}{2}\right)^2 \left( 1-   \frac{1 +\left(\frac{2}{E}\right)^2\left(\frac{a_n(x,y) + b_n(x,y)}{x^2 + y^2} + E_n^4 O(\|{\bf x}\|^4)\right)}{1 - \frac{2}{E}\left(\frac{f_n(x,y)}{x^2+y^2} +E_n^3 O(\|{\bf x}\|^4)\right)}  \right) \cr
&=\left(\frac{E}{2}\right)^2 \left[1-  \left(1 +\left(\frac{2}{E}\right)^2\left(\frac{a_n(x,y) + b_n(x,y)}{x^2 + y^2} + E_n^4 O(\|{\bf x}\|^4)\right)\right)\right.\\
&\hskip6.3cm\times \left.\left(1 +\frac{2}{E}\left(\frac{f_n(x,y)}{x^2+y^2} +E_n^3 O(\|{\bf x}\|^4)\right)\right)  \right]\cr
&=\left(\frac{E}{2}\right)^2 \left[1-  \left(1 +\left(\frac{2}{E}\right)^2\frac{a_n(x,y) + b_n(x,y)}{x^2 + y^2} +\frac{2}{E}\frac{f_n(x,y)}{x^2+y^2} +E_n^2 O(\|{\bf x}\|^4)\right) \right]\cr
&=\left(\frac{E}{2}\right)^2 \left[ -\left(\frac{2}{E}\right)^2\frac{a_n(x,y) + b_n(x,y)}{x^2 + y^2} -\frac{2}{E}\frac{f_n(x,y)}{x^2+y^2} +E_n^2 O(\|{\bf x}\|^4)\right]\cr
&=c E_n^3 \|{\bf x}\|^2 +E_n^4\,O( \|{\bf x}\|^4).
\end{split}
\end{equation*}
\end{proof}

\bigskip


\begin{thebibliography}{99}

\bibitem[A-T]{AT} Adler, R. J. and Taylor, J.E. (2007). {\it Random fields and geometry}. Springer-Verlag.

\bibitem[Ar-W]{ar-w}
Armentano, D. and Wschebor, M. (2009).  Random systems of polynomial equations. The expected number of roots under smooth analysis. {\it Bernoulli}, 15(1), 249-266.


\bibitem[A-L-W]{ALW} Azais, J.-M., L\'eon, and Wschebor, M. (2011). Rice formulas an Gaussian waves. {\it Bernoulli}, 17(1), 170-193.


\bibitem[A-W1]{aw-pol}
Aza\"{i}s, J.M. and Wschebor, M. (2005).
On the roots of a random system of equations. The theorem of Shub and Smale and some extensions. {\it Found. Comput. Math.}, 5(2), 125-144.

\bibitem[A-W2]{AW} Aza\"{i}s, J.M. and Wschebor, M. (2009). {\it Level Sets and Extrema of Random Processes and Fields}. Wiley-Blackwell.


\bibitem[Be1]{Berry 1977} Berry, M.V. (1977). Regular and irregular semiclassical wavefunctions. {\it J. Phys. A}, 10(12), 2083-2091.

\bibitem[Be2]{Berry 1978} Berry, M.V. (1978). Disruption of wavefronts: statistics of dislocations in incoherent Gaussian random waves. {\it J. Phys. A}, 11(1), 27-37.

\bibitem[Be3]{Berry 2002} Berry, M.V. (2002). Statistics of nodal lines and points in chaotic quantum billiards:
perimeter corrections, fluctuations, curvature. {\it J. Phys. A}, 35, 3025-3038.

\bibitem[B-D]{BD} Berry, M.V., and Dennis, M. R. (2000). Phase singularities in isotropic random waves. {\it Proc. Roy. Soc. Lond. A} 456, 2059-2079.

\bibitem[B-B]{B-B} Bombieri, E. and Bourgain, J. (2015). A problem on the sum of two squares. \emph{International Math. Res. Notices}, 11, 3343-3407.

\bibitem[B-R]{B-R} Bourgain, J. and Rudnick, Z. (2011). On the geometry of the nodal lines of eigenfunctions on the two-dimensional torus. {\it Ann. Henri Poincar\'e}, 12, 1027-1053

\bibitem[C]{C} Cheng, S.Y. (1976) Eigenfunctions and nodal sets. \emph{Comment. Math. Helv.}, 51(1), 43-55




\bibitem[D-O-P]{D-survey}Dennis, M. R., O'Holleran, K., and Padgett, M. J. (2009). Singular Optics: Optical Vortices and Polarization Singularities. In: {\it Progress in Optics} Vol. 53., Elsevier, 293-363.



\bibitem[D]{D} Dudley, R.M. (2011). {\it Real analysis and probability}. Cambridge University Press.

\bibitem[E-H]{EH} Erd\"os, P., and Hall, R.R (1999). On the angular distribution of Gaussian integers with fixed norm. {\it Discrete Math.}, 200(1-3), 87-94.

\bibitem[G-H]{GH} Geman, D., and Horowitz, J. (1980). Occupation densities. {\it Ann. Probab.}, 8(1), 1-67.

\bibitem[H-W]{H-W} Hardy, G.H. and Wright, E.M. (2008). {\it An introduction to the theory of numbers} (6th Edition). Oxford University Press..

\bibitem[K-Z]{KZ} Kabluchko, Z., and Zaporozhet, D. (2014). Random determinants, mixed volumes of ellipsoids and zeros of Gaussian random fields. {\it Journal of Math. Sci.} 199(2), 168-173.

\bibitem[K]{k} Kostlan, E. (2002). On the expected number of real roots of a system of random polynomial equations. In: {\it Foundations of computational mathematics (Hong Kong, 2000)}, World Scientific, 149-188.

\bibitem[Ko]{Kov} Khovanski{\u \i}, A. G. (1991). {\it Fewnomials}. American Mathematical Society.

\bibitem[K-K-W]{KKW} Krishnapur, M., Kurlberg, P., and Wigman, I. (2013).Nodal length fluctuations for arithmetic random waves. \emph{Ann. of Math.} 177(2), 699-737.

\bibitem[K-W]{KW} Kurlberg, P.; Wigman, I. (2016). On probability measures arising from lattice points on circles. To appear in: {\it Mathematische Annalen}.

\bibitem[L]{Lee} Lee, J.M. (1997). {\it Riemannian Manifolds}. Springer-Verlag.

\bibitem[Ma]{Ma} Maffucci, R.W. (2016). {Nodal intersections of random eigenfunctions against a segment on the 2-dimensional torus}. {\it ArXiv: 1603.09646}.

\bibitem[McL]{mcl}
McLennan, A. (2002). The expected number of real roots of a multihomogeneous system of polynomial equations.  {\it Amer. J. Math.}, 124(1), 49-73.

\bibitem[M-P-R-W]{MPRW} Marinucci, D., Peccati, G., Rossi, M., and Wigman, I. (2015). Non-Universality of nodal length distribution for arithmetic random waves. To appear in: {\it Geom. Funct. Anal.}

\bibitem[N-S]{ST2} Nazarov, F. and Sodin, M. (2011). Fluctuations in random complex zeroes: Asymptotic normality revisited. {\it Int. Math. Res. Notices}, 24, 5720-5759.

\bibitem[N-P]{NP} Nourdin, I., and Peccati, G. (2012). {\it Normal approximations with Malliavin calculus. From Stein's method to universality}. Cambridge University Press.

\bibitem[N]{N-survey} Nonnemacher, S. (2013). Anatomy of quantum chaotic eigenstates. In: B. Duplantier, S. Nonnemacher, V. Rivasseau (Eds), {\it Chaos}, {Prog. Math. Phys.} 66, 194-238. Birkh\"auser.

\bibitem[N-V]{NV} Nonnemacher, S. and Voros, A. (1998). Chaotic eigenfunctions in phase space. {\it J. Stat. Phys.} 92, 431-518.

\bibitem[N-B]{NB} Nye, J.F., and Berry, M.V. (1974). Dislocations in wave trains. {\it Proc. R. Soc. Lond. A}, 336, 165-190.

\bibitem[O-R-W]{ORW} Oravecz, F., Rudnick, Z., and Wigman, I. (2008). The Leray measure of nodal sets for random eigenfunctions on the torus. {\it Ann. Inst. Fourier (Grenoble)} 58(1), 299-335.

\bibitem[P-T]{PT} Peccati, G., and Taqqu, M.S. (2010). {\it Wiener chaos: moments, cumulants and diagrams}. Springer-Verlag.


\bibitem[Ro]{ro}
Rojas, J. (1996). On the average number of real roots of certain random sparse polynomial systems.
In: {\it The mathematics of numerical analysis (Park City, UT, 1995)}, Lectures in Appl. Math. 32, 689-699.
 Amer. Math. Soc.

\bibitem[S-S]{ss}
Shub, M., and Smale, S. (1993).
Complexity of B\'ezout's theorem. II. In: {\it Volumes and
Computational algebraic geometry (Nice, 1992)}, Progr. Math. 109, 267-285.
Birkh\"auser.

\bibitem[Ro]{Ro} Rossi, M. (2015). {\it The Geometry of Spherical Random Fields}. PhD thesis, University
of Rome Tor Vergata. ArXiv: 1603.07575.

\bibitem[Ro-W]{RoW} Rossi, M., and Wigman, I. On the asymptotic distribution of nodal intersections of arithmetic random waves against smooth curves.  {\it In preparation}.

\bibitem[R-W]{RW} Rudnick, Z., and Wigman, I. (2008).  On the volume of nodal sets for eigenfunctions of the Laplacian on the torus. {\it Ann. Henri Poincar\'e}  9(1), 109-130.

\bibitem[R-W2]{RW2} Rudnick, Z., and Wigman, I. (2014).  Nodal intersections for random eigenfunctions on the torus. To appear in: {\it Amer. J. of Math.}

\bibitem[S-T]{ST1} Sodin, M., and Tsirelson, B. (2004). Random complex zeroes, I. Asymptotic normality. {\it Israel J. Math.}, 144(1), 125-149.




\bibitem[S-Z1]{SZ1} Shiffman, B, and Zelditch, S. (2008). Number variance of random zeros on complex
manifolds. {\it Geom. Funct. Anal.} 18(4), 1422-1475.


\bibitem[S-Z2]{SZ2} Shiffman, B, and Zelditch, S. (2010). Number variance of random zeros on complex
manifolds, II: smooth statistics. {\it Pure and Applied Mathematics Quarterly}, 6(4), 1145-1167.


\bibitem[U-R]{UR-survey} Urbina, J., and Richter, K. (2013). Random quantum states: recent developments and applications. {\it Advances in Physics}, 62, 363-452 (2013)

\bibitem[W]{wig} Wigman, I. (2010).  Fluctuations of the nodal length of random spherical harmonics.
 {\it Communications in Mathematical Physics}, 298(3), 787-831.

\bibitem[Ws1]{w}
Wschebor, M. (2005). On the Kostlan-Shub-Smale model for random polynomial systems.
Variance of the number of roots. {\it J. Complexity} 21(6), 773-789.




%















%



%






















\bibitem[Ya]{Yau} Yau, S.T. (1982). Survey on partial differential equations in differential geometry. {\it Seminar on Differential Geometry}, Ann. of Math. Stud. 102, 3-71. Princeton Univ. Press, Princeton.

\bibitem[Zy]{Zy} Zygmund, A. (1974). On Fourier coefficients and transforms of functions of two variables,
{\it Studia Math.} 50, 189-201.



\end{thebibliography}
\end{document}